\pdfoutput=1
\RequirePackage{ifpdf}
\ifpdf % We~are running pdfTeX in pdf mode
\documentclass[pdftex]{sigma}
\else
\documentclass{sigma}
\fi

\numberwithin{equation}{section}

\newtheorem{Theorem}{Theorem}[section]
\newtheorem*{Theorem*}{Theorem}
\newtheorem{Corollary}[Theorem]{Corollary}
\newtheorem{Lemma}[Theorem]{Lemma}
\newtheorem{Proposition}[Theorem]{Proposition}
\newtheorem{Conjecture}[Theorem]{Conjecture}
 { \theoremstyle{definition}
\newtheorem{Definition}[Theorem]{Definition}

\newtheorem{Remark}[Theorem]{Remark} }

\newcommand{\leftexp}[2]{{\vphantom{#2}}^{#1}{#2}}

\usepackage{amscd,amsxtra}
\usepackage{mathtools}
\usepackage{tensor}
\usepackage{faktor}
\usepackage{dsfont}
\usepackage[all]{xy}
\usepackage{tikz, tikz-cd}
\usetikzlibrary{arrows, matrix, intersections, math, calc, decorations.pathmorphing, decorations.markings, positioning}
\usepackage{enumitem}
\newcommand{\CC}{\mathbb{C}}
\newcommand{\ZZ}{\mathbb{Z}}
\newcommand{\PP}{\mathbb{P}}
\newcommand{\RR}{\mathbb{R}}
\newcommand{\QQ}{\mathbb{Q}}
\newcommand{\ii}{\mathbf{i}}
\newcommand{\one}{\mathbf{1}}
\newcommand{\op}{\operatorname}
\newcommand{\A}{\mathcal{A}}

\newcommand{\T}{\mathcal{T}}

\renewcommand{\O}{\mathcal{O}}

\newcommand{\F}{\mathcal{F}}
\newcommand{\M}{\mathcal{M}}
\newcommand{\E}{\mathcal{E}}
\newcommand{\I}{\mathcal{I}}
\begin{document}

\allowdisplaybreaks

\newcommand{\arXivNumber}{2304.04365}

\renewcommand{\PaperNumber}{029}

\FirstPageHeading

\ShortArticleName{Reflection Vectors and Quantum Cohomology of Blowups}

\ArticleName{Reflection Vectors and Quantum Cohomology\\ of Blowups}

\Author{Todor MILANOV and Xiaokun XIA}
\AuthorNameForHeading{T.~Milanov and X.~Xia}
\Address{Kavli IPMU (WPI), UTIAS, The University of Tokyo, Kashiwa, Chiba 277-8583, Japan}
\Email{\href{todor.milanov@ipmu.jp}{todor.milanov@ipmu.jp}, \href{xia.xiaokun@ipmu.jp}{xia.xiaokun@ipmu.jp}}

\ArticleDates{Received May 30, 2023, in final form March 14, 2024; Published online April 05, 2024}

\Abstract{Let $X$ be a smooth projective variety with a semisimple quantum cohomology. It is known that the blowup $\operatorname{Bl}_{\rm pt}(X)$ of $X$ at one point also has semisimple quantum cohomology. In particular, the monodromy group of the quantum cohomology of $\operatorname{Bl}_{\rm pt}(X)$ is a~reflection
group. We found explicit formulas for certain generators of the monodromy group of the quantum cohomology of $\operatorname{Bl}_{\rm pt}(X)$ depending only on the geometry of the exceptional divisor.}


\Classification{14N35; 35Q53}

\section{Introduction}

The notion of a Frobenius manifold was invented by Dubrovin in order
to give a geometric formulation of the properties of quantum
cohomology (see \cite{Du}). Later on, it was discovered by Dubrovin and
Zhang (see \cite{DuZ}) that if the Frobenius manifold is in addition semisimple, then
the corresponding Frobenius structure has very important applications
to the theory of integrable hierarchies of KdV type. Our main interest
is in a certain system of vectors which we call {\em reflection
 vectors}, associated to any semisimple Frobenius manifold. The most general
problem is to obtain a classification of the set of reflection vectors
corresponding to a semisimple Frobenius manifold. In fact, the set of
reflection vectors contain the information about the monodromy group
of the so-called {\em second structure connection}, so by solving an
appropriate classical Riemann--Hilbert problem, the reflection vectors uniquely
determine the corresponding semisimple Frobenius structure.

\subsection{Period vectors}
The main motivation to define period vectors for semisimple Frobenius
manifolds comes from the work of Givental \cite{Giv-AnKdV}. Using the period
integrals of a simple singularity of type $A$, Givental was able to
construct an integrable
hierarchy in the form of Hirota bilinear equations. This result generalizes
to simple singularities too (see~\cite{GM2005}). The key to the
constructions in these two papers are the period integrals of
K.~Saito (see~\cite{Saito1983}). Given a singularity $f\in \O_{\CC^n}$, Saito has
invented an extension of the classical residue pairing, called {\em
 higher residue pairing}, and used it to
construct a period map for the hypersurface
$\{f(x)=\lambda\}\subset \CC^n$, where $\lambda$ is a regular value.
Since the hypersurface is a Stein manifold, the vector space of holomorphic forms is
infinite dimensional and it is not clear at all that a
good notion of a period map exists. Saito's remarkable idea was to
ask for a set of holomorphic forms, called {\em good basis},
$(\omega_1,\dots,\omega_N)$, where~$N$ is the Milnor number of the singularity, such that, the
higher-residue pairing of $\omega_i$ and~$\omega_j$ vanish for all $1\leq
i,j\leq N$. The existence of a good basis was proved by M.~Saito~\cite{MSaito1989}. The existence of a good basis implies that the
space of miniversal deformations of the singularity $f$ has a~{\em
 flat structure} which turns out to be semisimple Frobenius
structure in the sense of Dubrovin (see~\cite{hertling_2002}). Moreover, the Gauss--Manin
connection in vanishing cohomology is identified with
the second structure connection~\eqref{2nd_str_conn:1}--\eqref{2nd_str_conn:2} with some $m\in \ZZ$ or
$m\in \tfrac{1}{2}+\ZZ$ depending on whether the number $n$ of variables
of the singularity $f$ is odd or even. Therefore, the second structure
connection of the Frobenius structure in singularity theory admits
solutions in terms of period integrals, see~\cite[Section~3.1]{FGM2010} for a nice overview of the construction of such
solutions. The period integrals used by
Givental and Milanov correspond to singularities with odd number of
variables, i.e., they are solutions to the second structure connection
$\nabla^{(m)}$ for some $m\in \ZZ$. Let us point out that Dubrovin
has introduced the so-called {\em Gauss--Manin} and {\em
extended Gauss--Manin systems} (see \cite[equations~(5.9), (5.11) and (5.12)]{Dubrovin1999}). In our notation, the Gauss--Manin system
 is the system of linear differential equations \eqref{2nd_str_conn:1} with $m=-1$
 and $\lambda=0$, while the extended Gauss--Manin system is the system
\eqref{2nd_str_conn:1}--\eqref{2nd_str_conn:2} with $m=0$. However,
in order to obtain a more satisfactory characterization of the period
integrals, we need to
consider the entire sequence of second structure connections
\eqref{2nd_str_conn:1}--\eqref{2nd_str_conn:2} with $m\in \ZZ$!

\begin{Remark}
The entire set of connections $\nabla^{(m)}$ $(m\in \CC)$ was introduced by Manin and
Merkulov \cite{MM1997} in order to classify semisimple Frobenius
manifolds as special solutions to the Schlesinger equations.
\end{Remark}

It is easy to check that if $I^{(m)}(t,\lambda)$ is a solution
to $\nabla^{(m)}$, then $I^{(m+1)}(t,\lambda):=\partial_\lambda
I^{(m)}(t,\lambda)$ is a solution to $\nabla^{(m+1)}$. Moreover, if
\smash{$m+\tfrac{1}{2}$} is not an eigenvalue of the grading operator
$\theta$, then the differential
operator $\partial_\lambda$ defines an isomorphism between the
solutions of~$\nabla^{(m)}$ and~$\nabla^{(m+1)}$. In particular, the
monodromy representations of these two connections are isomorphic.
In singularity theory, even if \smash{$m+\tfrac{1}{2}$} is an eigenvalue of
$\theta$ and $\partial_\lambda$ might fail to be surjective, the
period integrals are always in the image of $\partial_\lambda$! The
reason is the following. If
$I^{(m+1)}$ is a solution to $\nabla^{(m+1)}$ defined in terms
of period integrals, then by {\em stabilizing} the singularity once,
i.e., adding a square of a new variable to $f$, we get a new period
integral $I^{(m+1/2)}$ which is a solution to~$\nabla^{(m+1/2)}$. Moreover, by stabilizing twice we get a period
integral $I^{(m)}$ which is a solution to~$\nabla^{(m)}$ satisfying
$I^{(m+1)}=\partial_\lambda I^{(m)}$. By stabilizing twice
we get that if a period integral is a~solution to~$\nabla^{(m+1)}$ for
some $m\in \ZZ$, then it must be a derivative of a solution of~$\nabla^{(m)}$ which itself is also a period integral.
The above discussion motivates the following definition.
\begin{Definition}\label{def:pv}
For a given semisimple Frobenius manifold,
a sequences $I^{(m)}(t,\lambda)$ ($m\in \ZZ$) satisfying the
following two conditions:
\begin{enumerate}\itemsep=0pt
\item[(i)] flatness:
 $I^{(m)}(t,\lambda)$ is a solution to $\nabla^{(m)}$ for all
 $m\in \ZZ$,
\item[(ii)] translation invariance:
 $\partial_\lambda I^{(m)}(t,\lambda)=I^{(m+1)}(t,\lambda)$ for all
 $m\in \ZZ$
\end{enumerate}
is said to be a {\em period vector} of the Frobenius manifold.
\end{Definition}

\subsection{Reflection vectors}
The notion of a reflection vector was suggested by the first author in
\cite{Milanov:p2}. The definition depends on the choice of a {\em
 calibration} and it will be recalled in Section \ref{sec:pv}. In
this section we would like to give an alternative definition which has
the advantage of being independent of the choice of a calibration. The
relation to the original definition will be explained in Section
\ref{sec:cfm} below.

Suppose that $M$ is a semisimple Frobenius manifold and let $\A$ be the set of all period vectors of $M$. Given $\alpha\in \A$, we denote by \smash{$I^{(m)}_\alpha(t,\lambda)$} ($m\in \ZZ$) the corresponding sequence of solutions. To be more precise, we fix a reference point $(t^\circ,\lambda^\circ)\in M\times \CC$ in the complement of the discriminant (see Section \ref{sec:pv}) and let each \smash{$I^{(m)}_\alpha(t,\lambda)$} be an analytic solution to $\nabla^{(m)}$ defined in a neighborhood of $(t^\circ,\lambda^\circ)$. Clearly, $\A$ has a vector space structure: if $\alpha,\beta\in \A$, then we define \smash{$I^{(m)}_{\alpha+\beta}(t,\lambda):=I^{(m)}_\alpha(t,\lambda)+I^{(m)}_\beta(t,\lambda)$} and
\smash{$I^{(m)}_{c\alpha}(t,\lambda):=cI^{(m)}_\alpha(t,\lambda)$} where $c\in \CC$ is a scalar. As we already explained above, if $m$ is a sufficiently negative integer, then the map \smash{$\alpha\mapsto I^{(m)}_\alpha(t,\lambda)$} gives an isomorphism between $\A$ and the space of solutions to $\nabla^{(m)}$. In particular, $\A$ is a finite dimensional vector space and the analytic continuation along closed loops in the complement of the discriminant defines a representation of $\pi_1((M\times \CC)',(t^\circ,\lambda^\circ))$ on $\A$ isomorphic to the monodromy representation of $\nabla^{(m)}$. Here $(M\times \CC)'$ denotes the complement of the discriminant in $M\times \CC$ (see Section \ref{sec:pv}). The monodromy group, i.e., the image of the fundamental group of this representation will be called the {\em stable monodromy group} of $M$, or just the {\em monodromy group} of $M$.
\begin{Remark}\label{re:monodromy}
Dubrovin has defined the monodromy group of a Frobenius manifold to be the monodromy group of the Gauss--Manin system (see \cite[Definition~5.6]{Dubrovin1999}), i.e., the monodromy group of the second structure connection with $m=-1$.
In this paper, by monodromy of a Frobenius manifold we mean the
monodromy group of the connection $\nabla^{(m)}$, where~$m$ is a~sufficiently
negative integer, such that, $m+\tfrac{1}{2}$ is not an eigenvalue of
the grading operator~$\theta$. Our definition of a monodromy group
will coincide with Dubrovin's one if the eigenvalues of the grading
operator~$\theta$ do not take values in $\tfrac{1}{2}+\ZZ$. Since
passing to more negative values of $m$ corresponds to stabilization in
singularity theory, if one needs to make a clear distinction between our definition and Dubrovin's one, we suggest to refer to our monodromy group as the {\em stable} monodromy group.
\end{Remark}
The space of period vectors $\A$ has the following remarkable pairing, called {\em intersection pairing},
\[
(\alpha|\beta):=\big(I^{(0)}_\alpha(t,\lambda), (\lambda-E\bullet) I^{(0)}_\beta(t,\lambda)\big),
\]
where $E$ is the Euler vector field (see Sections \ref{sec:flat_connections} and \ref{sec:pv} for more details). Using the differential equations \eqref{2nd_str_conn:1}--\eqref{2nd_str_conn:2}, it is easy to check that the right-hand side of the above formula is independent of $t$ and $\lambda$. In particular, the intersection pairing is monodromy invariant. The local structure of the period vectors near a generic point on the discriminant is described by the following simple lemma.
\begin{Lemma}\label{le:local_mon}
Suppose that $C$ is a reference path from $(t^\circ,\lambda^\circ)$ to a point $(t,\lambda)\in (M\times \CC)'$ sufficiently close to a generic point $(t',u')$ on the discriminant. Then
\begin{enumerate}\itemsep=0pt
\item[$(a)$]
The subspace of $\beta\in \A$, such that, \smash{$I^{(m)}_\beta(t,\lambda)$} extends analytically in a neighborhood of~$(t',u')$ is a co-dimension~$1$ subspace of $\A$.
\item[$(b)$]
Up to a sign, there exists a unique period vector $\alpha\in \A$, such
that, the analytic continuation along a small loop around $(t',u')$
transforms \smash{$I^{(m)}_\alpha(t,\lambda)\mapsto
-I^{(m)}_\alpha(t,\lambda)$} $(\forall m)$ and ${(\alpha|\alpha)=2}$.\looseness=1
\end{enumerate}
\end{Lemma}
In the case $m=0$, this is exactly Lemma 5.3 in
\cite{Dubrovin1999}. Dubrovin's proof works in general too after
some minor modifications. Moreover, the solutions to $\nabla^{(0)}$ constructed by
Dubrovin in \cite[Lemma~5.3]{Dubrovin1999}, are in fact periods, i.e.,
we can include them in a sequence satisfying the conditions of
Definition \ref{def:pv} --- see Section \ref{sec:ai_sol} where this sequence
is constructed in terms of Givental's $R$-matrix.
\begin{Definition}\label{def:rv}
A period vector $\alpha\in \A$ is said to be a {\em reflection vector}
if there exists a reference path $C$ approaching a generic point on
the discriminant, such that, $\pm\alpha$ is the unique vector from
part~(b) in Lemma \ref{le:local_mon}.

\end{Definition}
It is an easy corollary of Lemma \ref{le:local_mon} that if $C$ is a
simple loop in $(M\times \CC)'$ around a generic point on the
discriminant, then the analytic continuation along $C$ defines a
linear transformation
\[
w_\alpha\colon\ \A\to \A, \qquad x\mapsto x-(x|\alpha)\alpha,
\]
where $\alpha$ is the reflection vector corresponding to $C$, i.e.,
this is an orthogonal reflection in $\A$ with respect to the
hyperplane $\alpha^\perp:=\{ y\in \A\mid  (\alpha|y)=0\}$. This is our
main motivation to call $\alpha$ a reflection vector. Note that if
$\pi_1(M,t^\circ)=1$, then the monodromy group $\pi_1((M\times
\CC)',\allowbreak(t^\circ,\lambda^\circ))$ is generated by simple loops around
the discriminant. Therefore, in this case, the stable monodromy group
of $M$ is a reflection group.

\subsection{Calibrated Frobenius manifolds}\label{sec:cfm}
We would like to construct an isomorphism $\A\cong H$ where
$H=T_{t^\circ}M$ is the tangent space at the reference point, or
equivalently the space of flat vector fields. This isomorphism depends
on the choice of a fundamental solution $S(t,z) z^\theta z^{-\rho}$ of
Dubrovin's connection near $z=\infty$, such that, the operator series
satisfies the symplectic condition $S(t,z) S(t,-z)^T=1$ (see Section
\ref{sec:pv} for more details). In the case of quantum
cohomology of a manifold $X$, there is a canonical choice of such
solution: $\rho$ is the operator of classical cup product
multiplication by $c_1(TX)$ and $S(t,z)$ is defined in terms of
genus-0, 1-point descendent Gromov--Witten invariants of $X$ (see
\eqref{calibration}). In general, there is an ambiguity in the choice of such
a solution (see \cite{Dubrovin1999}, Lemma 2.7). Following Givental
\cite{Givental:qqh}, we call a Frobenius manifold equipped with the
choice of a fundamental solution {\em calibrated}. In this case, the
operator series $S$ is called {\em calibration}. Let us point out that
in all examples of Frobenius
manifolds coming from geometry, we have the following commutation
relation: $[\theta,\rho]=-\rho$. Following Givental again, we will say
that $\theta$ is a {\em Hodge grading operator} if there exists a
calibration for which $[\theta,\rho]=-\rho.$

Suppose now that $M$ is a calibrated Frobenius manifold for which
$\theta$ is a Hodge grading operator. Then we can construct an
$\op{End}(H)$-valued solution $I^{(m)}(t,\lambda)$ of the second structure
connection $\nabla^{(m)}$, such that, $\partial_\lambda
I^{(m)}(t,\lambda)=I^{(m+1)}(t,\lambda) $ --- see formula
\eqref{fundamental_period}. Moreover, if $m$ is sufficiently negative,
then the operator $I^{(m)}(t,\lambda)$ is invertible. Therefore, the
map
\begin{equation}\label{H-periods}
H\to \A,\qquad a\mapsto I^{(m)}(t,\lambda) a, \qquad m\in \ZZ
\end{equation}
is an isomorphism of vector spaces. Under this isomorphism, the
intersection pairing on $\A$ can be expressed in terms of the
operators $\theta$ and $\rho$ --- see formulas \eqref{euler-pairing} and
\eqref{inters-pairing:numerical}.

\subsection{Quantum cohomology}

Suppose that $X$ is a
smooth projective variety with semisimple quantum cohomology. There
is no geometric interpretation of the reflection vectors in
this case unless the manifold admits a mirror in the sense of
Givental. Nevertheless, there is a remarkable conjectural description
of the set of reflection vectors partially motivated by the examples
from mirror symmetry. Let us give a precise statement. Since the
quantum cohomology of $X$ is semisimple, the Dolbeault cohomology
groups $H^{p,q}(X)=0$ for $p\neq q$ (see \cite{HerMaTe}). In particular, there exists a set
of ample line bundles $L_1,\dots, L_r$, such that, the first Chern
classes $p_i:=c_1(L_i)$ ($1\leq i\leq r$) form a $\ZZ$-basis of
$H^2(X,\ZZ)_{\rm t.f.}$ (the torsion free part). Let $q_1,\dots, q_r$ be formal
variables. Following Iritani, we introduce the following map (see
\cite{Iritani:gamma_structure}):
\begin{equation}\label{psi}
\Psi_q\colon\ K^0(X) \to H^*(X,\CC),
\end{equation}
defined by
\[
\Psi_q(E) := (2\pi)^{\tfrac{1-n}{2}} \widehat{\Gamma}(X)\cup
e^{-\sum_{i=1}^r p_i\log q_i} \cup (2\pi\ii)^{\operatorname{deg}}
(\operatorname{ch}(E)),
\]
where $\operatorname{deg}$ is the complex degree operator, that is,
$\operatorname{deg}(\phi) = i \phi$ for $\phi\in H^{2i}(X;\CC)$,
$\ii:=\sqrt{-1}$, $n=\operatorname{dim}_\CC(X)$, and
$\widehat{\Gamma}(X) = \widehat{\Gamma}(TX)$ is the $\Gamma$-class of
$X$. Recall that for a vector bundle $E$, the $\Gamma$-class of $E$ is defined by
\[
\widehat{\Gamma}(E) := \prod_{x\colon \text{Chern roots of } E } \Gamma(1+x).
\]
The map $\Psi_q$ is multivalued with respect to $q$. If $q_i$ is
sufficiently close to $1$ for all $1\leq i\leq r$, then we define
$\log q_i$ via the principal branch of the logarithm. In general, one
has to fix a reference path in $(\CC^*)^r$ between $q=(q_1,\dots, q_r)$ and
$(1,\dots,1)$ and define $\Psi_q$ via analytic continuation along the
reference path.
Let us introduce the following pairing
\[%\label{euler_pairing}
\langle\ ,\ \rangle\colon\ H^*(X,\CC)\otimes H^*(X,\CC)\to \CC,\qquad
\langle a, b \rangle := \frac{1}{2\pi} \int_X a\cup {\rm e}^{\pi \ii
 \theta}\circ {\rm e}^{\pi\ii\rho} (b),
\]
where the linear operators $\theta$ and $\rho$ are defined
respectively by
\[
\theta\colon\ H^*(X,\CC)\to H^*(X,\CC),\qquad
\theta(\phi):=\frac{n \phi }{2}-\operatorname{deg}(\phi),
\]
and
\[
\rho\colon\ H^*(X,\CC)\to H^*(X,\CC),\qquad \rho(\phi) := c_1(TX)\cup \phi.
\]
By using the Hierzebruch--Riemann--Roch formula, we get
\[
\langle \Psi_q(E),\Psi_q(F) \rangle = \chi(E^\vee \otimes F),
\]
where $\chi$ is the holomorphic Euler characteristic, that is,
$\chi(E)=\sum_{i=0}^\infty (-1)^i \operatorname{dim} H^i(X,E)$. We
will refer to $\langle\ ,\ \rangle$ as the {\em Euler pairing}. In
case the manifold $X$ admits a mirror model in the sense of Givental,
the Euler pairing $\langle\ ,\ \rangle$ can be identified with the
Seifert form and therefore its symmetrization
\[
(a|b):= \langle a, b\rangle + \langle b,a\rangle,\qquad a,b\in H^*(X,\CC)
\]
corresponds to the intersection pairing. For that reason we refer to
the symmetrization $(\ |\ )$ of the Euler pairing as the {\em
 intersection pairing}.

Let us denote by $D^b(X)$
the derived category of the category of bounded complexes of coherent
sheaves on $X$, that is, the bounded derived category of $X$ (see
\cite{GM} for some background on derived categories). For $\E, \F \in
D^b(X)$ we denote by $\E[i]$ the shifted complex:
$(\E[i])^k:=\E^{k+i}$ and
$
\operatorname{Ext}^k(\E,\F):=\operatorname{Hom}(\E,\F[k])
$
where $\operatorname{Hom}$ is computed in the derived category
$D^b(X)$. Recall that an object $\E\in D^b(X)$ is called {\em
 exceptional} if
\[
\operatorname{Ext}^k(\E,\E)=
\begin{cases}
 \CC & \mbox{if } k=0,\\
 0 & \mbox{otherwise}.
\end{cases}
\]
A sequence of exceptional objects $(\E_1,\dots,\E_N)$ in $D^b(X)$ is called an
{\em exceptional collection} if~$\operatorname{Ext}^k(\E_i,\E_j)=0$
for all $i>j$ and $k\in \ZZ$. An exceptional collection
$(\E_1,\dots,\E_N)$ is called {\em full exceptional collection} if the
smallest subcategory of $D^b(X)$ that
contains $\E_i$ ($1\leq i\leq N$) and is closed under isomorphisms,
shifts, and cones, is $D^b(X)$ itself.
\begin{Conjecture}\label{conj:rv}\quad
\begin{itemize}\itemsep=0pt
\item[$(a)$] If the quantum cohomology of $X$ is convergent and semisimple,
 then $\Psi_q(\E)$ is a reflection vector for every exceptional
 object $\E\in D^b(X)$.
\item[$(b)$] If $(\E_1,\dots, \E_N)$ is a full exceptional collection in
 $D^b(X)$, then the reflection vectors $\alpha_i:=\Psi_q(\E_i)$
 $(1\leq i\leq N)$ generate the set $\mathcal{R}$ of all reflection
 vectors in the following sense:
 \begin{enumerate}\itemsep=0pt
 \item[$(i)$]
 The reflections $x\mapsto x-(x|\alpha_i) \alpha_i$ $(1\leq i\leq
 N)$ generate the monodromy group $W$ of quantum cohomology.
 \item[$(ii)$]
 For every $\alpha\in \mathcal{R}$ there exists $w\in W$, such
 that, $w(\alpha) \in \{\alpha_1,\dots,\alpha_N\}$.
 \end{enumerate}
 \end{itemize}
 \end{Conjecture}
The reflection vectors $\mathcal{R}$ in quantum cohomology are defined
through the Frobenius manifold structure on the domain $M\subset
H^*(X,\CC)$ of convergence of the big quantum cup product. Using the
calibration \eqref{calibration} we turn $M$ into a calibrated
Frobenius manifold. The space of period vectors of $M$ is identified with
$H^*(X,\CC)$ via the isomorphism \eqref{H-periods} which give us an
embedding of the set of reflection vectors $\mathcal{R}$ in $H^*(X,\CC)$. More
details will be given in Sections \ref{sec:pv} and~\ref{sec:qcoh}. Conjecture \ref{conj:rv} follows easily from the work
of Iritani (see \cite{Iritani:gamma_structure}) in the case when $X$
is a weak compact Fano toric orbifold that admits a full exceptional
collection consisting only of line bundles. In general, since the second structure
connection is a Laplace transform of Dubrovin's connection, Conjecture
\ref{conj:rv} should be equivalent to the so-called Dubrovin's
conjecture (see \cite[Conjecture 4.2.2]{Dubrovin1998}) or to its
improved version proposed by Galkin--Golyshev--Iritani (see
\cite[Conjecture 4.6.1]{GGI2016}). Dubrovin's conjecture was
originally stated for Fano manifolds but shortly afterwards Arend
Bayer suggested that the Fano condition should be removed (see
\cite{Ba}). Dubrovin already proved that the intersection pairing can
be expressed in terms of the Stokes multipliers for the first
structure connection (see \cite[Lemma 5.4]{Dubrovin1999}). We expect
that Dubrovin's argument already has the necessary ingredients to
prove that Conjecture~\ref{conj:rv} is equivalent to Dubrovin's
conjecture. We are planning to return to this problem in the near
future.

Let us state the main result in our paper. Let $\operatorname{Bl}(X)$
be the blowup of $X$ at one point, $\pi\colon \operatorname{Bl}(X)\to X$ be
the blowup map, and $j\colon \PP^{n-1}\to \operatorname{Bl}(X)$ be the
closed embedding that identifies~$\PP^{n-1}$ with the exceptional
divisor $E$.
\begin{Theorem}\label{thm}
 If the quantum cohomology of $X$ is convergent and semisimple and
 the quantum cohomology of $\operatorname{Bl}(X)$ is convergent, then
 $\Psi_q(\O_E(k))$, where $\O_E(k):= j_* \O_{\PP^{n-1}}(k)$, $k\in \ZZ$, are
 reflection vectors for the quantum cohomology of
 $\operatorname{Bl}(X)$.
\end{Theorem}
Several remarks are in order. It is known by the results of Bayer (see
\cite{Ba}) that the blowup at a point preserves semi-simplicity of the
quantum cohomology. We believe that our requirement that the quantum
cohomology of $\operatorname{Bl}(X)$ is convergent is redundant, that
is, the blowup operation preserves the convergence in quantum
cohomology. Let us point out that recently Giordano Cotti (see
\cite[Theorem 6.6]{Cotti2021})
was able to prove the convergence of quantum cohomology under the
assumption that the {\em small} quantum cohomology is semisimple and
convergent. His result is not quite sufficient for our purposes but
it might be interesting to apply his techniques to study convergence
of blowups in general. We will return to this problem in the near future.
Furthermore, we would like to prove
that Conjecture \ref{conj:rv} is compatible with the blowup
operation. Let us recall that by the work of Orlov (see \cite{Orlov}) if
$(\E_1,\dots,\E_N)$ is a full exceptional collection of $X$, then
$(\O_E(-n+1), \dots, \O_E(-1), \pi^*\E_1,\dots,\pi^*\E_N)$ is a full
exceptional collection of $\operatorname{Bl}(X)$. In order to complete
the proof of Conjecture \ref{conj:rv} for the blowups at finitely many
points we still have to prove that $\pi^*\E_i$ are reflection
vectors. The methods used in the current paper, after some modification,
should be sufficient to do this. Nevertheless, our attempts to modify
the arguments were unsuccessful so far, so we left this problem for a~separate project too.

The paper is organized as follows. Sections
\ref{sec:frobenius_manifolds} and \ref{sec:blowups} contain the
background which we need to formulate and state our main result, i.e.,
Theorem~\ref{thm}. In Section~\ref{sec:2nd_sc_blowups}, we investigate
the fundamental solution of the second structure connection of a
blowup. The goal is to expand the solution in a Laurent series at
$Q=0$, where $Q$ is the Novikov variable corresponding to the
exceptional divisor,
and to compute explicitly the coefficients of the leading order terms.
In Sections~\ref{sec:qexp_mon} and~\ref{sec:mirror}, we compute the
monodromy of the leading order coefficients in the $Q$-expansion. The
monodromy of the leading order coefficients allows us to determine the
reflection vectors corresponding to certain class of simple loops
which yields our main result. The logic in our proof is the
following. Let us look at the fundamental solution
$I^{(m)}(t,\lambda)$ of the second structure connection defined in
terms of the calibration --- see Section \ref{sec:pv}. This fundamental
solution is a~Laurent series whose coefficients are genus $0$,
$1$-point descendent GW invariants. Since the line bundle corresponding to the exceptional divisor $E$ is not ample, the
GW invariants are in general Laurent series in~$Q$. Our first
observation is that if we rescale appropriately the fundamental
solution $I^{(m)}(t,\lambda)$, then we will obtain a power series in
$Q$. Moreover, we can extract the leading order terms of the Taylor
series expansion at $Q=0$ up to order $Q^n$ where
${n=\operatorname{dim}(X)}$. This is done in Section
\ref{sec:2nd_sc_blowups} (see Propositions \ref{prop:Q-exp_e},
\ref{prop:Q-exp_phia} and \ref{prop:Q-exp_phi1}) by using~a
generalization of Gathman's vanishing result. The latter is proved in
Section \ref{sec:blowups}, Proposition~\ref{prop:van}. The next step
is to analyze the singularities of the second structure connection,
i.e., the dependence of the canonical coordinates $u_j$ on $Q$. Again
using Gathman's vanishing result we prove (see Proposition
\ref{prop:can-Q_exp}) that the canonical coordinates split into two
groups such that $Qu_j$ is either sufficiently close to $0$ (there are
$N=\operatorname{dim}H^*(X)$ such coordinates) or sufficiently close
to $-(n-1)v_k$ $(1\leq k\leq n-1)$. Suppose now that we have a simple
loop $\gamma_k$ around $ -(n-1)v_k$ that contains the corresponding
canonical coordinate and let $\alpha=: Q^{-(n-1)e}\beta $ be the
corresponding reflection vector, where the dependence of $\alpha$ on
$Q$ follows from the divisor equation (see Section~\ref{sec:rv_qdep}). Let us decompose $\beta=\beta_b+\beta_e$ where~$\beta_b \in H^*(X)$ and $\beta_e\in \widetilde{H}^*(E)$ where $E$ is
the exceptional divisor. In Section~\ref{sec:qexp_mon}, by analyzing
the monodromy of the leading order terms in the expansion at $Q=0$ we
prove that $\beta_b=0$. There is a slight complication in proving the
vanishing of the top degree part of~$\beta_b$ because one of the
coefficients in the $Q$-expansion (see Proposition~\ref{prop:Q-exp_e},
the term involving $Q^n\phi_N$) is an infinite series so its monodromy
is not straightforward to compute. In Section~\ref{sec:mirror}, we
prove that this problematic coefficient is a Mellin--Barnes integral
and we compute its monodromy by standard techniques based on deforming
the contour. Finally, in order to compute $\beta_e$, we look again at
the leading order term of the $Q$-expansion and we see that the
corresponding coefficient is a~fundamental solution for the second
structure connection in quantum cohomology of~$\mathbb{P}^{n-2}$ (see
Sections \ref{sec:tw_periods} and \ref{sec:proj_periods}). We get that
$\beta_e$ must be a reflection vector in the quantum cohomology of
$\mathbb{P}^{n-2}$ but the latter were computed in our previous work
\cite{Milanov:p2} (see also \cite{Iritani:gamma_structure}).

%%%%%%%%%%%%%%%%%%%%%%%%%%%%%%%%%%%%%%%%%%%%%%%%%%%%%%%%%%%%%%%%%%%%%%%%

\section{Frobenius manifolds}\label{sec:frobenius_manifolds}

%%%%%%%%%%%%%%%%%%%%%%%%%%%%%%%%%%%%%%%%%%%%%%%%%%%%%%%%%%%%%%%%%%%%%%%%%

Following Dubrovin \cite{Du}, we recall the notion of a Frobenius
manifold. Then we proceed by defining the so-called {\em second
 structure connection} and {\em reflection vectors} of a semisimple
Frobenius manifold. Finally, we would like to recall the construction
of a Frobenius manifold in the settings of Gromov--Witten theory.

\subsection{First and second structure connections}
\label{sec:flat_connections}
Suppose that $M$ is a complex manifold and $\mathcal{T}_M$ is the
sheaf of holomorphic vector fields on $M$. The manifold $M$ is
equipped with the following structures:
\begin{enumerate}\itemsep=0pt
\item[(F1)]
 A non-degenerate symmetric bilinear pairing
 \[
 ( \cdot , \cdot )\colon\
 \mathcal{T}_M\otimes \mathcal{T}_M\to \mathcal{O}_M.
 \]
\item[(F2)]
 A Frobenius multiplication: commutative associative multiplication
 \begin{equation*}
 \cdot \bullet \cdot \colon\
 \mathcal{T}_M\otimes \mathcal{T}_M\to \mathcal{T}_M,
 \end{equation*}
 such that $(v_1\bullet w, v_2) = (v_1,w\bullet v_2)$ $\forall
 v_1,v_2,w\in \mathcal{T}_M$.
\item[(F3)]
 A unit vector field: global vector field $\one\in \mathcal{T}_M(M)$, such
 that,
 \begin{equation*}
 \one\bullet v =v,\qquad \nabla^{\rm L. C.}_v \one=0,\qquad \forall v\in \mathcal{T}_M,
 \end{equation*}
 where $\nabla^{\rm L. C.}$ is the Levi-Civita connection of the
 pairing $(\cdot,\cdot)$.
\item[(F4)]
 An Euler vector field: global vector field $E\in \mathcal{T}_M(M)$,
 such that, there exists a constant~${n\in \CC}$, called {\em conformal
 dimension}, and
 \begin{equation*}
 E(v_1,v_2)-([E,v_1],v_2)-(v_1,[E,v_2]) = (2-n) (v_1,v_2)
 \end{equation*}
 for all $v_1,v_2\in \T_M$.
\end{enumerate}
Note that the complex manifold $TM\times \CC^*$ has a structure of a
holomorphic vector bundle with base $M\times \CC^*$: the fiber over
$(t,z)\in M\times \CC^*$ is $T_tM\times \{z\}\cong T_tM$ which has a
natural structure of a vector space. Given the data (F1)--(F4), we
define the so called {\em Dubrovin's connection} on the vector bundle
$TM\times \mathbb{C}^*$
\begin{align*}
 &
 \nabla_v :=
 \nabla^{\rm L.C.}_v -z^{-1} v\bullet,\qquad v\in
 \mathcal{T}_M,
 \qquad
\nabla_{\partial/\partial z} := \frac{\partial}{\partial z} -
 z^{-1}\theta + z^{-2} E\bullet,
\end{align*}
where $z$ is the standard coordinate on
$\mathbb{C}^*=\mathbb{C}\setminus{\{0\}}$, where $v\bullet $ is an
endomorphism of $\mathcal{T}_M$ defined by the Frobenius
multiplication by the vector field $v$, and where $\theta\colon \mathcal{T}_M\to
\mathcal{T}_M$ is an $\mathcal{O}_M$-modules morphism defined by
\begin{equation*}
 \theta(v):=\nabla^{\rm L.C.}_v(E)-\left(1-\frac{n}{2}\right) v.
\end{equation*}

\begin{Definition}\label{def:frob-manifold}
 The data $((\cdot,\cdot), \bullet, \one, E)$, satisfying the properties
 $(F1)$--$(F4)$, is said to be a~{\em Frobenius structure} of conformal
 dimension $n$ if the corresponding Dubrovin connection is flat, that
 is, if $(t_1,\dots,t_N)$ are holomorphic local coordinates on $M$,
 then the set of $N+1$ differential operators
$
\nabla_{\partial/\partial t_i} \ (1\leq i\leq N)$,
$
\nabla_{\partial/\partial z}
$
pairwise commute.

\end{Definition}

Let us proceed with recalling the notion of second structure connection
and reflection vectors. We follow the exposition from \cite{Milanov:p2}.
We are going to work only with Frobenius manifolds satisfying the
following 4 additional conditions:
\begin{enumerate}\itemsep=0pt
\item[(i)] The tangent bundle $TM$ is trivial and it admits a
 trivialization given by a frame of global flat vector fields.
\item[(ii)]
Recall that the operator
\[
\operatorname{ad}_E\colon\ \T_M\to \T_M,\qquad v\mapsto [E,v]
\]
preserves the space of flat vector fields. We require that the
restriction of $\operatorname{ad}_E$ to the space of flat vector
fields is a diagonalizable operator with eigenvalues rational numbers
$\leq 1$.
\item[(iii)]
 The Frobenius manifold has a {\em calibration} for which the grading
 operator is a Hodge grading operator (see Sections
 \ref{sec:cfm} and \ref{sec:pv}).
\item[(iv)]
 The Frobenius manifold has a direct product decomposition
 $M=\CC\times B$ such that if we denote by $t_1\colon M\to
 \CC$ the projection along $B$, then ${\rm d}t_1$ is a flat 1-form and
 $\langle {\rm d}t_1,\mathbf{1}\rangle=1.$
\end{enumerate}
Conditions (i)--(iv) are satisfied for all Frobenius manifolds
constructed by quantum cohomology or by the primitive forms in
singularity theory.

Let us fix a base point $t^\circ\in M$ and a basis
$\{\phi_i\}_{i=1}^N$ of the reference tangent space
$H:=T_{t^\circ}M$. Furthermore, let $(t_1,\dots,t_N)$ be a local flat
coordinate system on an open neighborhood of $t^\circ$ such that
$\partial/\partial t_i|_{t^\circ}=\phi_i$ in
$H$. The flat vector fields $\partial/\partial t_i$
($1\leq i\leq N$)
extend to global flat vector fields on $M$ and provide a
trivialization of the tangent bundle $TM\cong M\times H$. This allows us to
identify the Frobenius multiplication $\bullet$ with a family of
associative commutative multiplications~$\bullet_t\colon H\otimes H\to H$
depending analytically on $t\in M$. Modifying our choice of
$\{\phi_i\}_{i=1}^N$ and $\{t_i\}_{i=1}^N$ if necessary we may arrange
that
\[
E=\sum_{i=1}^N ((1-d_i) t_i + r_i)\partial/\partial t_i,
\]
where $\partial/\partial t_1$ coincides with the unit vector field
$\mathbf{1}$ and the numbers
\[
0=d_1\leq d_2\leq \cdots \leq d_N=n
\]
are symmetric with respect to the middle of the interval $[0,n]$. The
number $n$ is known as the {\em conformal dimension} of $M$. The
operator $\theta\colon \T_M\to \T_M$ defined above preserves the subspace of
flat vector fields. It induces a linear
operator on $H$, known to be skew symmetric with respect to
the Frobenius pairing $(\ ,\ )$. Following Givental, we refer to
$\theta$ as the {\em Hodge grading operator}.

There are two flat connections that one can associate with the
Frobenius structure. The first one is the {\em Dubrovin
 connection} --- defined above. The Dubrovin connection in flat
coordinates takes the following form:
\begin{gather*}
\nabla_{\partial/\partial t_i} = \frac{\partial}{\partial t_i} -
z^{-1} \phi_i\bullet, \qquad
\nabla_{\partial/\partial z} = \frac{\partial}{\partial z} +z^{-1}
\theta -z^{-2} E\bullet,
\end{gather*}
where $z$ is the standard coordinate on $\CC^*=\CC-\{0\}$ and
for $v\in \Gamma(M,\T_M)$ we denote by $v\bullet\colon H\to H$ the linear
operator of Frobenius multiplication by $v$.

Our main interest is in the {\em second structure connection }
\begin{align}
 \label{2nd_str_conn:1}
\nabla^{(m)}_{\partial/\partial t_i} & =
\frac{\partial}{\partial t_i} + (\lambda-E\bullet_t)^{-1} (\phi_i
 \bullet_t) (\theta-m-1/2),
 \\
 \label{2nd_str_conn:2}
\nabla^{(m)}_{\partial/\partial\lambda} & =
\frac{\partial}{\partial \lambda}-(\lambda-E\bullet_t)^{-1}
 (\theta-m-1/2),
\end{align}
where $m\in \CC$ is a complex parameter.
This is a connection on the trivial bundle
\[
(M\times \CC)'\times H \to (M\times \CC)',
\]
where
\[
(M\times \CC)'=\{ (t,\lambda)\mid  \det (\lambda-E\bullet_t)\neq 0\}.
\]
The hypersurface $\det (\lambda-E\bullet_t)=0$ in $M\times \CC$ is
called the {\em discriminant}.

\subsection{Reflection vectors for calibrated Frobenius
 manifolds}\label{sec:pv}
We would like to construct a fundamental solution to the second
structure connection $\nabla^{(m)}$ for~$m$ sufficiently negative. As
we already explained in Section \ref{sec:cfm}, this would allow us to
embed the reflection vectors (see Definition \ref{def:rv}) of the
Frobenius manifold $M$ in $H$.

Suppose that $M$ is a calibrated Frobenius manifold with calibration
$S(t,z)$ for which the grading operator is a Hodge grading operator.
By definition (see \cite{Givental:qqh}), the calibration is an
operator series $S(t,z)=1+\sum_{k=1}^\infty S_k(t)z^{-k}$, $S_k(t)\in
\operatorname{End}(H)$ depending holomorphically on $t$ and~$z$ for
$t$ sufficiently close to the base point $t^\circ$ and $z\in \CC^*$,
such that, the Dubrovin's connection has a fundamental solution near
$z=\infty$ of the form
\[
S(t,z) z^{\theta} z^{-\rho},
\]
where $\rho\in \operatorname{End}(H)$ is a nilpotent operator
and the following symplectic condition holds
\[
S(t,z)S(t,-z)^T=1,
\]
where ${}^T$ denotes transposition with respect to the Frobenius
pairing. We say that $\theta$ is a Hodge grading operator if
$[\theta,\rho]=-\rho$.

Let us fix a reference point $(t^\circ,\lambda^\circ)\in (M\times
\CC)'$ such that $\lambda^\circ$ is a sufficiently
large positive real number. It is
easy to check that the following function is a solution to the second
structure connection $\nabla^{(m)}$
\begin{equation}\label{fundamental_period}
I^{(m)}(t,\lambda) = \sum_{k=0}^\infty (-1)^k S_k(t) \widetilde{I}^{(m+k)}(\lambda),
\end{equation}
where
\[%\label{calibrated_period}
\widetilde{I}^{(m)}(\lambda) = {\rm e}^{-\rho \partial_\lambda \partial_m}
\left(
\frac{\lambda^{\theta-m-\frac{1}{2}} }{ \Gamma\big(\theta-m+\frac{1}{2}\big) }
\right).
\]
Note that both $I^{(m)}(t,\lambda)$ and $\widetilde{I}^{(m)}(\lambda)$
take values in $\op{End}(H)$. From now on we restrict ${m\in \ZZ}$.
The second structure connection has a
Fuchsian singularity at infinity, therefore the series ${I^{(m)}(t,\lambda)}$ is
convergent for all $(t,\lambda)$ sufficiently close to
$(t^\circ,\lambda^\circ)$. Using the differential equations
\eqref{2nd_str_conn:1}--\eqref{2nd_str_conn:2}, we extend
$I^{(m)}$ to a multi-valued analytic function on $(M\times \CC)'$
taking values in $\op{End}(H)$. We define the following multi-valued
functions taking values in $H$:
\begin{equation}\label{H-pv}
I^{(m)}_a(t,\lambda):=I^{(m)}(t,\lambda) a,
\qquad a\in H,\quad m\in \ZZ.
\end{equation}
Clearly, for each fixed $a\in H$, the sequence $I^{(m)}_a(t,\lambda)$
($m\in \ZZ$) is a period vector in the sense of Definition~\ref{def:pv}. Moreover, if $m\in \ZZ$ is sufficiently negative, then
$I^{(m)}(t,\lambda)$ is an invertible operator. Therefore, all period
vectors of $M$ have the form \eqref{H-pv}. Using analytic
continuation we get a representation
\[%label{mon-repr}
\pi_1((M\times\CC)',(t^\circ,\lambda^\circ) )\to \operatorname{GL}(H)
\]
called the {\em monodromy representation} of the Frobenius
manifold. The image $W$ of the monodromy representation is called the
{\em monodromy group} or {\em stable monodromy group} (see Remark
\ref{re:monodromy}).

Under the semi-simplicity
assumption, we may choose a generic reference point $t^\circ$ on $M$, such that
the Frobenius multiplication $\bullet_{t^\circ}$ is semisimple and the
operator $E\bullet_{t^\circ}$ has $N$ pairwise different eigenvalues
$u_i^\circ$ ($1\leq i\leq N$). The fundamental group $\pi_1((M\times
\CC)', (t^\circ,\lambda^\circ))$ fits into the following exact sequence
\begin{equation}\label{fund-es}
\xymatrix{
\pi_1(F^\circ,\lambda^\circ) \ar[r]^-{i_*} &
\pi_1((M\times \CC)', (t^\circ,\lambda^\circ) ) \ar[r]^-{p_*} &
\pi_1(M,t^\circ)\ar[r] &
1,}
\end{equation}
where $p\colon (M\times \CC)'\to M$ is the projection on $M$,
$F^\circ=p^{-1}(t^\circ)=\CC\setminus{\{u_1^\circ,\dots,u_N^\circ\}}$
is the fiber over~$t^\circ$, and
$i\colon F^\circ\to (M\times \CC)'$ is the natural inclusion. For a proof we refer to
\cite[Proposition~5.6.4]{Shimada:ln} or \cite[Lemma 1.5\,(C)]{Nori-zariski_conj}.
Using the exact sequence \eqref{fund-es}, we get that the monodromy
group $W$ is generated by the monodromy transformations
representing the lifts of the generators of $\pi_1(M,t^\circ)$ in
$\pi_1((M\times \CC)', (t^\circ, \lambda^\circ))$ and the
generators of $\pi_1(F^\circ,\lambda^\circ)$.

The image of $\pi_1(F^\circ,\lambda^\circ)$ under the monodromy
representation is a reflection group that can be described as
follows. Let us introduce the bi-linear pairing
\begin{equation}\label{euler-pairing}
\langle a,b\rangle = \frac{1}{2\pi}
\big( a, {\rm e}^{\pi\ii \theta} {\rm e}^{\pi\ii \rho} b\big),\qquad
a,b\in H.
\end{equation}
Motivated by the applications to mirror symmetry, we will refer to $\langle\ ,\ \rangle$ as the {\em Euler pairing}. Its symmetrization
\begin{equation}\label{inters-pairing:numerical}
(a|b):= \langle a, b\rangle + \langle b, a\rangle,\qquad a,b\in H,
\end{equation}
also plays an important role in mirror symmetry and we will refer to it as the {\em intersection pairing}. It can be checked that the intersection pairing can be expressed in terms of the period vectors as follows:
\[%\label{inters-pairing}
(a|b):=\big(I^{(0)}_a(t,\lambda),(\lambda-E\bullet)I^{(0)}_b(t,\lambda)\big).
\]
Using the differential equations of the second structure connection, it is
easy to prove that the right-hand side of the above identity is independent of $t$ and $\lambda$. However, the fact that the constant must be $(a|b)$ requires some additional work (see \cite{MilSa}).

Suppose now that $\gamma$ is a simple loop in $F^\circ$, i.e., a loop
that starts at $\lambda^\circ$,
approaches one of the punctures $u_i^\circ$ along a path $\gamma'$
that ends at a point sufficiently close to $u_i^\circ$, goes around~$u_i^\circ$, and finally returns back to $\lambda^\circ$ along
$\gamma'$. By analyzing the second structure connection near
$\lambda=u_i$ it is easy to see that up to a sign there exists a unique $a\in H$
such that $(a|a)=2$ and the monodromy transformation of $a$ along
$\gamma$ is $-a$. The monodromy transformation representing $\gamma\in
\pi_1(F^\circ,\lambda^\circ)$ is the reflection defined by the
following formula:
\[%\label{wa}
w_a(x)=x-(a|x) a.
\]
Let us denote by $\mathcal{R}$ the set of all $a\in H$ as above determined by
all possible choices of simple loops in $F^\circ$. Under the
isomorphism \eqref{H-periods}, the set $\mathcal{R}$ coincides with
the set of reflection vectors of $M$.

\subsection{The anti-invariant solution}\label{sec:ai_sol}
We would like to construct the unique solution to the second structure
connection appearing in part (b) of Lemma~\ref{le:local_mon}. We refer to
it as the anti-invariant solution because the analytic continuation
around the discriminant changes its sign. Our construction is very
similar to \eqref{fundamental_period}, except that now instead of
the singularity of $\nabla^{(m)}$ at $\lambda=\infty$, we will
consider the singularity at~$\lambda=u_i(t)$ and instead of
fundamental solution of Dubrovin's connection near $z=\infty$ we will
make use of the formal asymptotic solution to Dubrovin's connection
near $z=0$.

Let us recall Givental's $R$-matrix (see \cite{Givental:qqh})
\[
R(t,z)=1+R_1(t) z+R_2(t) z^2+\cdots,\qquad
R_k(t)\in \operatorname{End}(H)
\]
defined for all semisimple $t\in M$ as the unique solution to the
following system of differential equations:
\begin{align*}
\frac{\partial R}{\partial t_a} (t,z) & = -
R(t,z) \frac{\partial \Psi}{\partial t_a} \Psi^{-1} +
 z^{-1}[\phi_a\bullet,R(t,z)],\\
\frac{\partial R}{\partial z} (t,z) & = -
z^{-1}\theta R(t,z) -
z^{-2} [E\bullet, R(t,z)],
\end{align*}
where $\phi_a\bullet$ and $E\bullet$ are the operators of Frobenius
multiplication respectively by the flat vector field
$\partial/\partial t_a$ and by the Euler vector field $E$ and $\Psi$ is
the $(N\times N)$-matrix with entries
\[
\Psi_{ai}:= \sqrt{\Delta_i} \frac{\partial t_a}{\partial u_i},
\qquad
1\leq a,i\leq N,
\]
where $u_1,\dots,u_N$ are the canonical coordinates in a neighborhood
of the base point $t^\circ$, that is, a local coordinate system, such that,
\[
\frac{\partial}{\partial u_i} \bullet \frac{\partial}{\partial u_j} =
\delta_{ij} \frac{\partial}{\partial u_j},\qquad
\left( \frac{\partial}{\partial u_i},\frac{\partial}{\partial
 u_j}\right)=
\frac{\delta_{ij}}{\Delta_{i}},
\]
where $\delta_{ij} $ is the Kronecker delta symbol and $\Delta_i\in
\O_{M,t^\circ}$ is a holomorphic function that has no zeroes in a
neighborhood of $t^\circ$. It is known that the canonical coordinates
coincide with the eigenvalues of the operator $E\bullet_t$.
Here $\operatorname{End}(H)$ is identified with the space of $(N\times
N)$-matrices via the basis $\phi_1,\dots,\phi_N$, that is, the
entries $A_{ab}$ of $A\in \operatorname{End}(H)$ are defined by
$A(\phi_b)=:\sum_a \phi_a A_{ab}$.
\begin{Remark}
The matrix $\Psi$ up to the normalization factors $\Delta_i$ is the
Jacobian matrix of the change from canonical to flat coordinates. The
above definition of the $R$-matrix differed from the original
definition in \cite{Givental:qqh} by conjugation by $\Psi$, that is,
$\Psi^{-1}R(t,z)\Psi$ is the $R$-matrix of Givental.
\end{Remark}
Suppose that $\alpha \in H$ is a reflection vector. Let us fix a generic
semisimple point $t\in M$, such that, the canonical coordinates
$u_1(t),\dots,u_N(t)$
are pairwise distinct. Let us fix a reference path from
$(t^\circ,\lambda^\circ)$ to a neighborhood of a point on the
discriminant $(t,u_i(t))$ for some $i$, such that, the period vector
\smash{$I^{(-m)}_\alpha(t,\lambda)$} transforms into \smash{$-I^{(-m)}_\alpha(t,\lambda)$} under
the analytic continuation in $\lambda$ along a closed loop around
$u_i(t)$. We claim that the period vector has the following expansion
at $\lambda=u_i(t)$:
\begin{equation}\label{expansion:u_i}
I^{(-m)}_\alpha(t,\lambda)=\sqrt{2\pi }
\sum_{k=0}^\infty (-1)^k
\frac{(\lambda-u_i)^{k+m-1/2}}{\Gamma(k+m+1/2)}
R_k(t) \Psi(t) e_i,
\end{equation}
where $e_i$ is the vector column with $1$ on the $i$th position and $0$
elsewhere, that is, $\Psi e_i$ is the column representing the vector field
$\sum_{a=1}^N \sqrt{\Delta_i} \frac{\partial t_a}{\partial u_i}
\phi_a = \sqrt{\Delta_i}\partial/\partial u_i$. Let us prove this
claim. Using the differential
equations for $R(t,z)$, it is easy to check that the right-hand side of the above
formula is a solution to the second structure connection. Therefore, the
right-hand side of~\eqref{expansion:u_i} and the reference path determine a vector
$\alpha\in H$ for which formula~\eqref{expansion:u_i} holds. Moreover,
\[
\big(I^{(0)}_\alpha(t,\lambda),(\lambda-E\bullet)
I^{(0)}_\alpha(t,\lambda)\big) =
\frac{2\pi}{\Gamma(1/2)^2} +O(\lambda-u_i)=
2+O(\lambda-u_i).
\]
Since the left-hand side is independent of $\lambda$ and $u_i$, the higher order
terms $O(\lambda-u_i)$ in the above formula must vanish. This proves
that $(\alpha|\alpha)=2$. Finally, since the analytic continuation around~$\lambda=u_i$ of the right-hand side of \eqref{expansion:u_i} changes the sign of
the right-hand side, we conclude that $\alpha$ must be a reflection vector and
that \eqref{expansion:u_i} is the expansion of the corresponding
period vector near the discriminant.

\subsection{Gromov--Witten theory}\label{sec:GW}
Let us recall some basics on Gromov--Witten (GW) theory. For
further details we refer to~\cite{Ma}. Let
$\operatorname{Eff}(X)\subset H_2(X,\ZZ)_{\rm t.f.}$ be the monoid of all homology
classes that can be represented in the form $\sum_i k_i [C_i]$, where
$k_i$ is a non-negative integer and $[C_i]$ is the fundamental class
of a~holomorphic curve $C_i\subset X$. The
main object in GW theory is the moduli space of stable maps~$\overline{\M}_{g,k}(X,\beta)$, where $g$, $k$ are non-negative integers
and $\beta\in \operatorname{Eff}(X)$. By definition, a stable map
consists of the following data $(\Sigma, z_1,\dots,z_k, f)$:
\begin{enumerate}\itemsep=0pt
\item[(1)]
 $\Sigma$ is a Riemann surface with at most nodal
 singular points.
\item[(2)]
 $z_1,\dots,z_k$ are {\em marked points}, that is, smooth
 pairwise-distinct points on $\Sigma$.
\item[(3)]
 $f\colon \Sigma \to X$ is a holomorphic map, such that, $f_*[\Sigma]=\beta$.
\item[(4)]
 The map is stable, i.e., the automorphism group of
 $(\Sigma,z_1,\dots,z_k,f)$ is finite.
\end{enumerate}
Two stable maps $(\Sigma,z_1,\dots,z_k,f)$ and
$(\Sigma',z'_1,\dots,z'_k,f')$ are called equivalent if there exists a~biholomorphism $\phi\colon\Sigma\to \Sigma'$, such that, $\phi(z_i)=z_i'$
and $f'\circ \phi = f$. The moduli space of equivalence classes of
stable maps is known to be a proper
Delign--Mumford stack with respect to the \'etale topology on the
category of schemes (see \cite{BeMa}). The corresponding coarse moduli space
$\overline{M}_{g,k}(X,\beta)$ has a structure of a projective variety,
which however could be very singular. We have the following diagram:
\[
\xymatrix{
 \overline{\M}_{g,k+1} (X,\beta) \ar[r]^-{{\rm ev}_{k+1}} \ar[d]_\pi & X \\
 \overline{\M}_{g,k} (X,\beta) \ar[r]^-{{\rm ev}_{i} } \ar[d]_{{\rm
 ft}} & X \qquad 1\leq i\leq k,\\
 \overline{\M}_{g,k}, &
}
\]
where ${\rm ev}_i(\Sigma,z_1,\dots,z_k,f) := f(z_i)$, $\pi$ is the map
forgetting the last marked point an contracting all unstable
components, and ${\rm ft}$ is the map forgetting the holomorphic map
$f$ and contracting all unstable components. The moduli space has natural
orbifold line bundles $L_i$ ($1\leq i\leq k$) whose fiber at a
point $(\Sigma,z_1,\dots,z_k,f)$ is the cotangent line
$T^*_{z_i}\Sigma$ equipped with the action of the automorphism group
of $(\Sigma,z_1,\dots,z_k,f)$. Let $\psi_i=c_1(L_i)$ be the first
Chern class. The most involved construction in GW theory is the
construction of the so called {\em virtual fundamental cycle}. The
construction has as an input the complex $(R\pi_*{\rm
 ev}_{k+1}^*TX)^\vee$ which gives rise to a~perfect obstruction
theory on $\overline{\M}_{g,k}(X,\beta)$ relative to
$\overline{\M}_{g,k}$ (see \cite{Be,BF_intnormalcone}) and yields a homology cycle in~$\overline{M}_{g,k}(X,\beta)$ of complex dimension
\[
3g-3 +k + n(1-g) + \langle c_1(TX),\beta\rangle,
\]
known as the virtual fundamental cycle. Gromov--Witten invariants are by
definition the following correlators:
\[
\big\langle a_1\psi^{l_1},\dots, a_k\psi^{l_k}\big\rangle_{g,k,\beta} =
\int_{[\overline{M}_{g,k}(X,\beta)]^{\rm virt}}
\operatorname{ev}_1^*(a_1)\cdots \operatorname{ev}_k^*(a_k)
\psi_1^{l_1}\cdots \psi_k^{l_k},
\]
where $a_1,\dots,a_k\in H^*(X;\CC)$ and $l_1,\dots,l_k$ are
non-negative integers.

Let us recall the so-called {\em string} and {\em divisor} equations. Suppose that either $\beta\neq 0$ or ${2g-2+k>0}$, then
\begin{align*}
\big\langle 1, a_1\psi^{l_1},\dots, a_k\psi^{l_k}\big\rangle_{g,k+1,\beta} =
\sum_{i=1}^k \big\langle
a_1\psi^{l_1},\dots, a_i\psi^{l_i-1}, \dots, a_k\psi^{l_k}\big\rangle_{g,k,\beta},
\end{align*}
and if $p\in H^2(X,\CC)$ is a divisor class, then
\begin{align*}
\big\langle p, a_1\psi^{l_1},\dots, a_k\psi^{l_k}\big\rangle_{g,k+1,\beta} ={} &
\left(\int_\beta p\right) \big\langle a_1\psi^{l_1},\dots, a_k\psi^{l_k}\big\rangle_{g,k,\beta} \\
&+
\sum_{i=1}^k \big\langle
a_1\psi^{l_1},\dots, p\cup a_i\psi^{l_i-1}, \dots, a_k\psi^{l_k}\big\rangle_{g,k,\beta},
\end{align*}
where if $l_i=0$, then we define $\psi_i^{l_i-1}:=0$. We will need
also the {\em genus}-0 {\em topological recursion relations}, that is,
if $k\geq 2$, then the following relation holds:
\begin{align*}
&\big\langle a \psi^{l+1}, b_1\psi^{m_1},\dots,
 b_k\psi^{m_k}
 \big\rangle_{0,k+1,\beta}
 \\
 &\qquad=
 \sum_{i, I,\beta'}\big\langle
 a \psi^l, \phi_i,
 b_{i_1} \psi^{m_{i_1}},\dots,
 b_{i_r} \psi^{m_{i_r}}\big\rangle_{0,2+r,\beta'}
\big\langle
\phi^i, b_{j_1} \psi^{m_{j_1}},\dots,
 b_{j_s} \psi^{m_{j_s}}\big\rangle_{0,1+s,\beta''},
\end{align*}
where the sum is over all $1\leq i\leq N$, all subsequences
$I=(i_1,\dots,i_r)$ of the sequence $(1,2,\dots,k)$ including the
empty one, and all homology
classes $\beta'\in \operatorname{Eff}(X)$, such that,
$\beta'':=\beta-\beta'\in \operatorname{Eff}(X)$. The sequence
$(j_1,\dots,j_s)$ is obtained from $(1,2,\dots,k)$ by removing the
subsequence $I$. In particular, $r+s=k$.

\subsection[Quantum cohomology of X]{Quantum cohomology of $\boldsymbol{X}$}\label{sec:qcoh}
Let us recall the notation $L_i$, $p_i:=c_1(L_i)$, and $q_i$ ($1\leq i\leq r$) from the introduction. If $\beta\in \operatorname{Eff}(X)$, then we put
\smash{$q^\beta = q_1^{\langle p_1,\beta\rangle}
\cdots q_r^{\langle p_{r},\beta\rangle}$}. The group ring
$\CC[\operatorname{Eff}(X)]$ is called the {\em Novikov ring} of $X$
and the variables $q_i$ are called {\em Novikov variables}. Note that
the Novikov variables determine an embedding of the Novikov ring into
the ring of formal power series $\CC[\![q_1,\dots,q_r]\!]$.
Let us fix a~homogeneous basis $\phi_i$ ($1\leq i\leq N$) of $H^*(X;\CC)$, such that, $\phi_1=1$ and $\phi_{i+1}=p_i$ for all $1\leq i\leq r$. Let
$t=(t_1,\dots,t_N)$ be the corresponding linear coordinates. The
quantum cup product $\bullet_{t,q}$ of~$X$ is a deformation of the
classical cup product defined by
\[
(\phi_a\bullet_{t,q} \phi_b,\phi_c):=
\langle \phi_a, \phi_b,\phi_c\rangle_{0,3}(t) =
\sum_{m=0}^\infty
\sum_{\beta\in \operatorname{Eff}(X) }
\frac{q^\beta}{m!}
\langle
\phi_a,\phi_b,\phi_c,t,\dots,t\rangle_{0,3+m,\beta}.
\]
Using string and divisor equation, we get that the structure constants of the
quantum cup product, i.e., the 3-point genus-0 correlators in the above
formula are independent of $t_1$ and are formal power series in the
following variables:
\[
q_1 {\rm e}^{t_2},\dots, q_r {\rm e}^{t_r}, t_{r+1},\dots, t_N.
\]
We are going to consider only manifolds $X$, such that, the quantum
cup product is analytic. More precisely, let us allow for the Novikov
variables to take values $0<|q_i|<1$ ($1\leq i\leq r$). Then we will
assume that there exists an $\epsilon>0$, such that, the structure
constants of the quantum cup product are convergent power series for
all $t$ satisfying
\begin{equation}\label{conv-dom}
\operatorname{Re}(t_i)<\log \epsilon, \quad 2\leq i\leq r+1,\qquad
|t_j|<\epsilon, \quad r+1<j\leq N.
\end{equation}
The inequalities \eqref{conv-dom} define an open subset $M\subset
H^*(X;\CC)$. The main fact about genus-0 GW invariants is that $M$
has a Frobenius structure, such that, the Frobenius pairing is the
Poincar\'e pairing, the Frobenius multiplication is the quantum cup
product, the unit $\one=\phi_1$, and the Euler vector field is
\[
E=\sum_{i=1}^N (1-d_i) t_i \frac{\partial}{\partial t_i} +
\sum_{j=2}^{r+1} \big(c_1(TX),\phi^j\big) \frac{\partial}{\partial t_j},
\]
where $d_i$ is the complex degree of $\phi_i$, that is, $\phi_i\in
H^{2d_i}(X;\CC)$ and $\phi^j$ ($1\leq j\leq N$) is the basis of~$H^*(X;\CC)$ dual to $\phi_i$ ($1\leq i\leq N$) with respect to the
Poincar\'e pairing. Let us point out that in case the quantum cup
product is semisimple we have $H^{\rm odd}(X;\CC)=0$. Otherwise, in
general~$M$ has to be given the structure of a super-manifold (see~\cite{Ma}). The conformal dimension of $M$ is~$n=\operatorname{dim}_\CC(X)$ and the Hodge grading operator takes the
form
\begin{equation}\label{qcoh_hgo}
\theta(\phi_i) =\left(\frac{n}{2} - d_i\right) \phi_i,\qquad 1\leq i\leq N.
\end{equation}
Finally, there is a standard choice for a calibration
$S(t,q,z)=1+\sum_{k=1}^\infty S_k(t,q) z^{-k}$, where
$S_k(t,q)\in \operatorname{End}(H^*(X;\CC))$ is defined by
\begin{equation}\label{calibration}
(S_k(t,q)\phi_i,\phi_j) =
\sum_{m=0}^\infty
\sum_{\beta\in \operatorname{Eff}(X)} \frac{q^\beta}{m!}
\big\langle \phi_i \psi^{k-1}, \phi_j, t,\dots,t\big\rangle_{0,2+m,\beta}.
\end{equation}
Suppose that the Frobenius manifold $M$ corresponding to quantum cohomology is
semisimple. Recalling the construction from Section \ref{sec:pv}, we get the notion
of a reflection vector.

\section{The geometry of blowups}\label{sec:blowups}

Let $\operatorname{Bl}(X)$ be the blowup of $X$ at a point ${\rm pt}\in X$,
$\pi\colon \operatorname{Bl}(X)\to X$ be the corresponding
blowup map, and $E:=\pi^{-1}({\rm pt})$ the
exceptional divisor. Put $e=c_1(\O(E))={\rm P.D.}(E)$.
We would like to recall some well known facts about
$\operatorname{Bl}(X)$ which will be used later on.

\subsection{Cohomology of the blowup}
\label{sec:coh_blowup}
Using a Mayer--Vietories sequence
argument, it is easy to prove the following two facts:
\begin{enumerate}\itemsep=0pt
 \item[(1)]
The pullback map $\pi^*\colon\xymatrix{H^*(X;\CC)\ar[r] &
 H^*(\operatorname{Bl}(X);\CC)}$ is injective, so we can view the
 cohomology $H^*(X;\CC)$ as a subvector space of
 $H^*(\operatorname{Bl}(X);\CC)$.
\item[(2)]
 We have a direct sum decomposition
 \[
 H^*(\operatorname{Bl}(X);\CC) = H^*(X;\CC)\bigoplus
 \widetilde{H}^*(E),
 \]
 where $\widetilde{H}^*(E)= \bigoplus_{i=1}^{n-1} \CC e^i$ is the
 reduced cohomology of $E$.
\end{enumerate}
The Poincar\'e pairing of $\operatorname{Bl}(X)$ can be computed as follows.
Let us choose a basis $\phi_i$ ($1\leq i\leq N$) of
$H^*(X;\CC)$, such that,
\begin{enumerate}\itemsep=0pt
\item[(i)]
 $\phi_1=1$ and $\phi_N={\rm P.D.}({\rm pt})$,
\item[(ii)]
 $\phi_{i+1} =p_i=c_1(L_i)$ ($1\leq i\leq r$), where $L_i$ ($1\leq
 i\leq r$) is a set of ample line bundles on $X$, such that, $p_i$ ($1\leq
 i\leq r$) form a $\ZZ$-basis of $H^2(X,\ZZ)_{\rm t.f.}$.
\end{enumerate}
\begin{Lemma}\label{le:pp}
 Let $(\ ,\ )^{\operatorname{Bl}(X)}$ and $(\ ,\ )^X$ be the Poincar\'e
 pairings on respectively $ \operatorname{Bl}(X)$ and $X$. Then we have
\begin{enumerate}\itemsep=0pt
\item[$(a)$] $(\phi_i,\phi_j )^{\operatorname{Bl}(X)} = (\phi_i,\phi_j )^X$ for all
$1\leq i,j\leq N$.
\item[$(b)$] $(\phi_i,e^k )^{\operatorname{Bl}(X)} =0$ for $1\leq i\leq N$ and $1\leq
k\leq n-1$.
\item[$(c)$] $e^n=(-1)^{n-1} \phi_N$ and \smash{$\big(e^k,e^{n-k} \big)^{\operatorname{Bl}(X)} =(-1)^{n-1}$}.
\end{enumerate}
\end{Lemma}
\begin{proof}
Parts (a) and (b) follow easily by the projection formula and Poincar\'e
duality. The second part of (c) is a consequence of the first part, so
we need only to prove that $e^n=(-1)^{n-1} \phi_N$. We have $e^n=c
\phi_N$ for dimension reasons. Note that
$E\cong \mathbb{P}^{n-1}$ and
$\O(E)|_E=\O_{\PP^{n-1}}(-1)$. Therefore, $e|_E=c_1(O(E)|_E) = -p$,
where $p=c_1(\O_{\PP^{n-1}}(1))$ is the standard hyperplane class of~$\PP^{n-1}$. We get
\[
c=\int_{[\operatorname{Bl}(X)]} e^n = \int_{[E]} e^{n-1} = \int_{[\PP^{n-1}]}
(-p)^{n-1} = (-1)^{n-1}.\tag*{\qed}
\]\renewcommand{\qed}{}
\end{proof}

The ring structure of $H^*(\operatorname{Bl}(X);\CC)$ with respect to
the cup product is also easy to compute. We have
\begin{enumerate}\itemsep=0pt
\item[(1)]
 $H^*(X;\CC)$ is a subring of $H^*(\operatorname{Bl}(X);\CC)$.
\item[(2)]
$\phi_1\cup e^k=e^k$ and
$\phi_i \cup e^k = 0$, $2\leq i\leq N$, $1\leq k\leq n-1$.
\item[(3)]
\[
e^k\cup e^l=\begin{cases}
e^{k+l} & \mbox{if } k+l<n,\\
(-1)^{n-1}\phi_N & \mbox{if } k+l=n,\\
0 & \mbox{if } k+l >n.
\end{cases}
\]
\end{enumerate}
Property (1) follows from the fact that pullback in cohomology is a
ring homomorphism. The formulas in (3) follow from
Lemma~\ref{le:pp}\,(c). Finally, (2) follows from (1), (3) and Lemma~\ref{le:pp}\,(b).

\subsection[K-ring of the blowup]{$\boldsymbol{K}$-ring of the blowup}
For some background on topological K-theory we refer to
\cite[Chapter~6]{FoFu2016}.
Let us compute the topological $K$-ring of $\operatorname{Bl}(X)$. We will be interested
only in manifolds $X$, such that, the corresponding quantum cohomology
is semisimple. Such $X$ are known to have cohomology classes of Hodge
type $(p,p)$ only. In particular, $K^{1}(X)\otimes \QQ=0$. To simplify the
exposition, let us assume that $K^1(X)=0$. In our arguments below we
will have to work with non-compact manifolds. However, in all cases
the non-compact manifolds are homotopy equivalent to finite
CW-complexes. We define the corresponding $K$-groups by taking the
$K$-groups of the corresponding finite CW-complexes.

\begin{Proposition}\label{prop:K}\quad
\begin{itemize}
\item[$(a)$] The $K$-theoretic pullback $\pi^*\colon K^0(X)\to K^0(\operatorname{Bl}(X))$
 is injective.
\item[$(b)$] We have
 \[
 K^0(\operatorname{Bl}(X)) = K^0(X) \oplus \bigoplus_{j=1}^{n-1} \ZZ
 \O_E^j,
 \]
 where $K^0(X)$ is viewed as a subring of $K^0(\operatorname{Bl}(X)) $ via
 the K-theoretic pullback $\pi^*$ and $\O_E:=\O-\O(-E)$ is the
 structure sheaf of the exceptional divisor.
\end{itemize}
\end{Proposition}
\begin{proof}
Let $U\!\subset \!X$ be a small open neighborhood of the center of the
blowup ${\rm pt}$ and ${V\!:=\!X\!\setminus\!{\{{\rm pt}\}}}$. Note that
$\{U,V\}$ is a covering of $X$. Put $\widetilde{U}=\pi^{-1}(U)$ and
$\widetilde{V}:=\pi^{-1}(V)$, then $\big\{\widetilde{U},\widetilde{V}\big\}$
is a~covering of $\operatorname{Bl}(X)$. Let us compare the reduced $K$-theoretic
Mayer--Vietories sequences of these two coverings. We have the
following commutative diagram:
$$
\xymatrix@C=1.6em{
 \widetilde{K}^{-1}(X)\ar[r]\ar[d] &
 \widetilde{K}^{-1}(V)\oplus \widetilde{K}^{-1}(U) \ar[r]\ar[d] &
 \widetilde{K}^{-1}(U\setminus{{\rm pt}}) \ar[r]\ar[d]_{\cong}&
\widetilde{K}^{0}(X)\ar[r]\ar[d] &
 \widetilde{K}^0(V)\oplus \widetilde{K}^0(U) \ar[d]
\\
 \widetilde{K}^{-1}(\operatorname{Bl}(X))\ar[r] &
 \widetilde{ K}^{-1}\big(\widetilde{V}\big)\oplus\widetilde{ K}^{-1}\big(\widetilde{U}\big) \ar[r] &
\widetilde{ K}^{-1}\big(\widetilde{U}\setminus{E}\big) \ar[r]&
\widetilde{K}^{0}(\operatorname{Bl}(X))\ar[r] &
 \widetilde{ K}^0\big(\widetilde{V}\big)\oplus\widetilde{ K}^0\big(\widetilde{U}\big), }
$$
where the vertical
arrows in the above diagram are induced by the K-theoretic pullback
$\pi^*$. Note that
\smash{$\widetilde{K}^{\rm ev}(U\setminus{{\rm pt}}) =
\widetilde{K}^0\big(\widetilde{U}\setminus{E}\big) =0$} because
\smash{$\widetilde{U}\setminus{E} \cong U\setminus{{\rm pt}} $} is
homotopic to $\mathbb{S}^{2n-1}$ -- the $(2n-1)$-dimensional sphere. Therefore, the horizontal arrows in the first and the last square of the above diagram are respectively injections and surjections.
Furthermore, $\widetilde{K}^{-1}(U)=\widetilde{K}^{0}(U)=0$ because $U$
is contractible and \smash{$\widetilde{K}^{-1}\big(\widetilde{U}\big)=0$} because
$\widetilde{U}$ is homotopy equivalent to $E\cong \PP^{n-1}$. We get
that the second vertical arrow is an isomorphism \smash{$\big(V\cong
\widetilde{V}\big)$} and hence, recalling the 5-lemma or by simple diagram chasing, we get
$\widetilde{K}^{-1}(\operatorname{Bl}(X)) \cong \widetilde{K}^{-1}(X)$. By assumption $\widetilde{K}^{-1}(X)=0$, so $\widetilde{K}^{-1}(\operatorname{Bl}(X)) =0$. A
straightforward diagram chasing shows that the 4th vertical arrow is
injective, i.e., we proved~(a).

Note that the above diagram yields the following short exact sequence:
\begin{equation}\label{short-es}
\xymatrix{
 0\ar[r] &
 \widetilde{K}^0(X)\ar[r]^-{\pi^*} &
 \widetilde{K}^0(\operatorname{Bl}(X)) \ar[r]^-{|_E} &
 \widetilde{K}^0\big(\PP^{n-1}\big)\ar[r] & 0, }
\end{equation}
where the map $|_E$ is the restriction to the exceptional divisor
$E\cong \PP^{n-1}$. The above exact sequence splits because \smash{$\widetilde{K}^0\big(\PP^{n-1}\big)\cong \ZZ^{n-1}$} is a free module. Note that
$\O_E|_E = \O_{\PP^{n-1}}-\O_{\PP^{n-1}}(1)$ is the generator of
\smash{$\widetilde{K}^0\big(\PP^{n-1}\big)$}, so part~(b) follows from the exactness of
\eqref{short-es}.
\end{proof}

Let us compute the K-theoretic product of the torsion free part $K^0(\operatorname{Bl}(X))_{\rm t.f.}$. Note that
$\pi_*(\O_{\operatorname{Bl}(X)})=\O_X$. Therefore, $\pi_*\pi^*(F)=F$ for
every $F\in K^0(X)$. Let us compute $\O_E\otimes \pi^*F$ for~$F\in \widetilde{K}^0(X)$. The
restriction of $\O_E\otimes \pi^*F$ to $E$ is 0. Recalling the exact sequence
\eqref{short-es}, we get $\O_E\otimes \pi^*F=\pi^*G$ for some $G\in
\widetilde{K}^0(X)$. Taking pushforward, we get
\[
G=\pi_*(\O_E\otimes \pi^*F) = \pi_*(\O_E)\otimes F =
\CC_{{\rm pt}}\otimes F = \operatorname{rk}(F)
\CC_{{\rm pt}} =0,
\]
where $\CC_{{\rm pt}}$ is the skyscraper sheaf on $X$ and in the
3rd equality we used the exact sequence
\[
\xymatrix{
0\ar[r] & \O(-E)\ar[r] &\O \ar[r] & j_*(\O_{\PP^{n-1}}) \ar[r] & 0,}
\]
where $j\colon\PP^{n-1}\to \operatorname{Bl}(X)$ is the embedding whose image
is the exceptional divisor. This sequence implies
$\O_E=j_*\O_{\PP^{n-1}}$ and hence $\pi_* \O_E=
(\pi\circ j)_*\O_{\PP^{n-1}} = \CC_{{\rm pt}}.$ We proved
that
\[
\O_E\otimes \pi^*F=0,\qquad \forall F\in \widetilde{K}^0(X).
\]
It remains only to compute $\O_E^n$. The restriction of $\O_E^n$ to
$E$ is $(1-\O_{\PP^{n-1}}(-1))^n = 0$. Therefore, $\O_E^n =\pi^*
F$. The Chern character $\operatorname{ch}(\O_E^n) =
(1-\exp(-c_1(\O(E))) )^n = e^n = (-1)^{n-1} \phi_N$, where we used
Lemma \ref{le:pp}\,(c). On the other hand, the Chern character of
the skyscraper sheaf can be computed easily with the Grothendieck--Riemann--Roch
formula. Namely, we have
\[
\operatorname{ch}(\iota^\circ_*(\CC) )\cup \operatorname{td}(X) =
\iota^\circ_* (\operatorname{ch}(\CC)\cup \operatorname{td}({\rm pt}))
=
\iota^\circ_* (1) = \operatorname{P.D.}({\rm pt}) = \phi_N,
\]
where $\iota^\circ\colon {\rm pt}\to X$ is the natural inclusion of the point
${\rm pt}$. Thus, %The above formula implies
$\operatorname{ch}(\CC_{{\rm pt}}) = \phi_N$. Comparing with the
formula for $\operatorname{ch}(\O_E^n) $, we get
\[
\O_E^n = (-1)^{n-1} \CC_{{\rm pt}}\qquad
\operatorname{mod} \ \operatorname{ker}(\operatorname{ch}).
\]
Finally, let us finish this section by quoting the formula for the
K-theoretic class of the tangent bundle (see \cite[Lemma 15.4]{Fu}):
\[%\label{tangent-b}
T\operatorname{Bl}(X) = TX-n-1 + n \O(-E)+\O(E).
\]

\subsection{Quantum cohomology of the blowup}\label{sec:qcoh_blowup}
Let us first compare the effective curve cones $\operatorname{Eff}(X)$
and $\operatorname{Eff}(\operatorname{Bl}(X)).$ We have an exact sequence
\[
\xymatrix{0\ar[r] &
 H_2\big(\PP^{n-1};\ZZ\big) \ar[r]^-{j_*} &
 H_2(\operatorname{Bl}(X);\ZZ) \ar[r]^-{\pi_*} &
 H_2(X;\ZZ)\ar[r] & 0,}
\]
where $j\colon \PP^{n-1}\to \operatorname{Bl}(X)$ is the natural closed
embedding of the exceptional divisor. The proof of the exactness is
similar to the proof of \eqref{short-es}. In particular, since the torsion
free part of the above sequence splits, we get
\[
H_2(\operatorname{Bl}(X);\ZZ)_{\rm t.f.} = H_2(X;\ZZ)_{\rm t.f.} \oplus
\ZZ \ell,
\]
where $\ell\in H_2(E;\ZZ)$ is the class of a line in the exceptional
divisor. The cone of effective curve classes
$\operatorname{Eff}(\operatorname{Bl}(X))\subset \operatorname{Eff}(X) \oplus
\ZZ \ell$. The Novikov variables of the blowup will be fixed to be
the Novikov variables of $X$ and an extra variable corresponding to
the line bundle $\O(E)$. In other words, for $\widetilde{\beta} = \beta+ d \ell \in
\operatorname{Eff}(\operatorname{Bl}(X))$, put
\[
q^{\widetilde{\beta}} =
q^\beta q_{r+1}^{\langle c_1(O(E)),\widetilde{\beta}\rangle}=
q_1^{\langle \phi_2,\beta\rangle}\cdots
q_r^{\langle \phi_{r+1},\beta\rangle}
q_{r+1}^{-d}.
\]
Note that $\O(E)$ is not an ample line bundle: for example,
$\ell\cdot E =-1<0$. Our choice of~$q_{r+1}$ makes the structure constants formal Laurent (not power) series in
$q_{r+1}$. Following Bayer (see~\cite{Ba}), we write $q_{r+1}=Q^{n-1}$
for some formal variable $Q$.
Let us recall the basis $\phi_i$ ($1\leq i\leq N$) of
$H^*(X;\CC)$. Put $\phi_{N+k} = e^k$ ($1\leq k\leq n-1$). Then
$\phi_i$ $\big(1\leq i\leq \widetilde{N}:=N+n-1\big)$ is a~basis of
$H^*(\operatorname{Bl}(X);\CC)$. Let~${t=(t_1,\dots,t_{\widetilde{N}})}$
be the corresponding linear coordinate system on
$H^*(\operatorname{Bl}(X);\CC)$.
The structure constants of the quantum cohomology of $\operatorname{Bl}(X)$
take the form
\[
(\phi_a\bullet_{t,q} \phi_b,\phi_c):=
\langle \phi_a, \phi_b,\phi_c\rangle_{0,3}(t) =
\sum_{m=0}^\infty
\sum_{\widetilde{\beta}=(\beta,d)}
\frac{q^\beta Q^{-d(n-1)}}{m!}
\langle
\phi_a,\phi_b,\phi_c,t,\dots,t\rangle_{0,3+m,\widetilde{\beta}}.
\]
\begin{Remark}
The quantum cup product of $H^*(X)$ depends only on
$(t_2,\dots,t_N)$. Suppose that these $N-1$ parameters are generic
such that the quantum cup product of $H^*(X)$ is semisimple. Then,
according to Bayer \cite{Ba} (see also Proposition \ref{prop:Ba}),
even if we restrict the remaining parameters to $0$, that is, set
$t_1=t_{N+1}=\cdots=t_{\widetilde{N}}=0$, then the quantum cup product
of the blowup is still semisimple. Therefore, for our purposes, it is
sufficient to work with $t\in H^*(\operatorname{Bl}(X))$, such that,
$t_1=t_{N+1}=\cdots=t_{\widetilde{N}}=0$.
\end{Remark}

\subsection[Twisted GW invariants of P\^\{n-1\}]{Twisted GW invariants of $\boldsymbol{\PP^{n-1}}$}\label{sec:proj-tw-GW}

It turns out that genus-0 GW invariants of $\operatorname{Bl}(X)$ whose degree $\widetilde{\beta}= d\ell$ with $d\neq 0$ can be identified with certain twisted GW invariants of $\PP^{n-1}$.
Suppose that $(C,z_1,\dots,z_k,f)$ is a stable map representing a point in $\overline{\mathcal{M}}_{0,k}(\operatorname{Bl}(X), d\ell)$. Let $\pi\colon \operatorname{Bl}(X)\to X$ be the blowup map. Since~${\pi_*\circ f_*[C]=0}$ and $\pi$ induces a biholomorphism between $\operatorname{Bl}(X)\setminus{E}$ and $X\setminus{\{\rm pt \}}$, we get that $f(C)$ is contained in $E$. Therefore, we have a canonical identification
\[
\overline{\mathcal{M}}_{0,k}(\operatorname{Bl}(X), d\ell) =
\overline{\mathcal{M}}_{0,k}(E, d),
\]
where $E\cong \PP^{n-1}$ is the exceptional divisor. Let us compare the virtual tangent spaces of the two moduli spaces at $(C,z_1,\dots,z_k,f)$. For the left-hand side, we have
\begin{gather*}
\T_{0,k,d\ell} =
H^1(C,\T_C(-z_1-\dots -z_k)) -
H^0(C,\T_C(-z_1-\dots -z_k)) \\
\hphantom{\T_{0,k,d\ell} =}{}
+
H^0(C,f^*T_{\operatorname{Bl}(X)}) -
H^1(C, f^*T_{\operatorname{Bl}(X)}),
\end{gather*}
while for the right-hand side we have
\begin{gather*}
\T_{0,k,d} =
H^1(C,\T_C(-z_1-\dots -z_k)) -
H^0(C,\T_C(-z_1-\dots -z_k)) \\
\hphantom{\T_{0,k,d} =}{} +
H^0(C,f^*T_{E}) -
H^1(C, f^*T_{E}),
\end{gather*}
where $\T_C$ is the tangent sheaf of $C$ and $\T_C(-z_1-\cdots -z_k)$ is the sub sheaf of $\T_C$ consisting of sections vanishing at $z_1,\dots, z_k$.
On the other hand, we have an exact sequence
\[
\xymatrix{0\ar[r] &
 T_E \ar[r] &
 T_{\operatorname{Bl}(X)}|_E \ar[r] &
 \mathcal{O}_E(-1) \ar[r] & 0,}
\]
where we used that $\mathcal{O}_E(-1)$ is the normal bundle to the exceptional divisor in $\operatorname{Bl}(X)$. Pulling back the exact sequence to $C$ via the stable map and taking the long exact sequence in cohomology, we get
\[
\xymatrix@R=0.8em{0\ar[r] &
 H^0(C,f^*T_E) \ar[r] &
 H^0(C,f^*T_{\operatorname{Bl}(X)}) \ar[r] &
 H^0(C,f^*\mathcal{O}_E(-1)) \ar[r] & \\
 {}\ar[r] &
H^1(C,f^*T_E) \ar[r] &
H^1(C,f^*T_{\operatorname{Bl}(X)}) \ar[r] &
H^1(C,f^*\mathcal{O}_E(-1)) \ar[r] & 0.}
\]
Note that $H^0(C,f^*\mathcal{O}_E(-1))=0$ because $C$ is a rational curve. Indeed, if $C'$ is an irreducible component of $C$ and $d'=f_*[C']$ is its contribution to the degree of $f$, then $C'\cong \PP^1$ and~$f^*\mathcal{O}_E(-1)|_{C'}= \O_{\PP^1}(-d') $. Therefore, $H^0(C', f^*\mathcal{O}_E(-1))=0$ and we get that the restrictions of the sections of $f^*\O_{E}(-1)$ to the irreducible components of $C$ are $0$ which implies that there are no non-zero global sections. Let us recall the Riemann--Roch formula for nodal curves (easily proved by induction on the number of nodes)
\[
\operatorname{dim} H^0(C,\mathcal{L}) -
\operatorname{dim} H^1(C,\mathcal{L}) = 1-g + \int_{[C]} c_1(\mathcal{L}),
\]
where $\mathcal{L}$ is a holomorphic line bundle on $C$ and $g$ is the genus of $C$. Applying the Riemann--Roch formula to $f^*\O_E(-1)$, we get that
\[
\operatorname{dim} H^1(C, f^*\O_E(-1)) =
-1-\int_{f_*[C]} c_1(\O_E(-1)) = d-1.
\]
The cohomology group $H^1(C, f^*\O_E(-1))$ is the fiber of a holomorphic vector bundle $\mathbb{N}_{0,k,d}$ on~$\overline{\M}_{0,k}(E,d)$ of rank $d-1$. The virtual tangent bundles are related by $\T_{0,k,d\ell} = \T_{0,k,d} -\mathbb{N}_{0,k,d}$. Recalling the construction of the virtual fundamental cycle \cite{BF_intnormalcone}, we get
\[
\big[\overline{\mathcal{M}}_{0,k}(\operatorname{Bl}(X), d\ell)\big]^{\rm virt} =
\big[\overline{\mathcal{M}}_{0,k}(E, d)\big]^{\rm virt} \cap e(\mathbb{N}_{0,k,d}).
\]
The above formula for the virtual fundamental class yields the
following formula:
\[
\langle \alpha_1 \psi^{m_1},\dots, \alpha_k \psi^{m_k}\rangle_{0,k,d\ell} =
\int_{
[\overline{\mathcal{M}}_{0,k}(E, d)]^{\rm virt} }
\prod_{i=1}^k \operatorname{ev}_i^*(\alpha_i|_E) \psi_i^{m_i} \cup e(\mathbb{N}_{0,k,d}).
\]
Later on we will need the 3-point GW invariants with $d=1$. Let us compute them.
If $d=1$, then $e(\mathbb{N}_{0,k,d})=1$ and the above formula implies that the GW invariants of the blowup coincide with the GW invariants of the exceptional divisor, that is,
\[
\langle \alpha_1 \psi^{m_1},\dots, \alpha_k \psi^{m_k}
\rangle_{0,k,\ell}^{\operatorname{Bl}(X)} =
\langle \alpha_1|_E \psi^{m_1},\dots, \alpha_k|_E \psi^{m_k}
\rangle_{0,k,1}^E,
\]
where we used the superscripts $\operatorname{Bl}(X)$ and $E$ in order to specify that the correlators are GW invariants of respectively $\operatorname{Bl}(X)$ and $E$. Note that if $p=c_1\O_E(1)$ is the hyperplane class, then~$e|_{E} =-p$. The quantum cohomology of $\PP^{n-1}$ is well known to be $\CC[p]/(p^n-Q)$. In particular, the 3-point correlators
\[
\big\langle p^i,p^j,p^k\big\rangle_{0,3,1} =
\begin{cases}
1 & \mbox{if } i+j+k=2n-1,\\
0 & \mbox{otherwise},
\end{cases}
\qquad \forall
0\leq i,j,k\leq n-1.
\]
Therefore,
\[%\label{3pt:deg_1}
\big\langle e^i,e^j,e^k\big\rangle_{0,3,\ell} =
\begin{cases}
-1 & \mbox{if } i+j+k=2n-1,\\
\hphantom{-}0 & \mbox{otherwise}.
\end{cases}
\]
Let us specialize $k=1$. Using the divisor equation (recall that $\int_\ell e=-1$), we get
\[%\label{2pt:deg_1}
\big\langle e^i,e^j\big\rangle_{0,2,\ell} =
\begin{cases}
1 & \mbox{if } i=j=n-1,\\
0 & \mbox{otherwise}.
\end{cases}
\]

\subsection{The vanishing theorem of Gathmann}

Gathmann discovered a very interesting vanishing criteria for the GW
invariants of the blowup (see \cite{Gat}). We need a slight generalization of his
result which can be stated as follows. Following Gathmann, we
assign a weight to each basis vector
\[
{\rm wt}(\phi_a) = \begin{cases}
 0 & \mbox{if } 1\leq a\leq N,\\
 a-N-1 & \mbox{if } N<a\leq N+n-1.
\end{cases}
\]
In other words, the exceptional class $e^k$ has weight $k-1$ for all
$1\leq k\leq n-1$ and in all other cases the weight is $0$.
\begin{Proposition}\label{prop:van}
Suppose that we have a GW invariant
\begin{equation}\label{van-cor}
\big\langle \phi_a \psi^k,
\phi_{b_1},\dots,\phi_{b_m},
e^{l_1},\dots, e^{l_s}
\big\rangle_{0,\beta+ d\ell, 1+m+s},
\end{equation}
where $1\leq a\leq \widetilde{N}$, $1\leq b_1,\dots,b_m\leq N$, and
$2\leq l_1,\dots,l_s\leq n-1$, satisfying the following $3$~conditions:
\begin{enumerate}\itemsep=0pt
\item[$(i)$]
 $\beta\neq 0$.
\item[$(ii)$]
 ${\rm wt}(\phi_a)+ \sum_{i=1}^s (l_i-1) >0$ or $d>0$.
\item[$(iii)$]
 ${\rm wt}(\phi_a)+ \sum_{i=1}^s (l_i-1) < (d+1)(n-1)-k.$
\end{enumerate}
Then the GW invariant \eqref{van-cor} must be $0$.
\end{Proposition}
\begin{proof}
The proof is done by induction on $k$. Gathmann's result is the case
when $k=0$. The inductive step uses the genus-0 topological recursion
relations (see Section~\ref{sec:GW}). Suppose that the proposition is
proved for $k$ and let us prove it for $k+1$. Using the TRRs, we write
the correlator~\eqref{van-cor} with $k$ replaced by $k+1$ in the
following form:
\[
\sum_{c=1}^{N+n-1}\ \sum\
\big\langle \phi_a\psi^k, \phi_c,\phi_{B'}, e^{L'}\big\rangle_{\beta'+d'\ell}
\big\langle \phi^c, \phi_{B''}, e^{L''}\big\rangle_{\beta''+d''\ell},
\]
where the second sum is over all possible splittings $B'\sqcup B''=\{b_1,\dots,b_m\}$,
$L'\sqcup L''=\{l_1,\dots,l_s\}$, $\beta'+\beta''=\beta$ and $d'+d''=d$. The notation is as
follows. We dropped the genus and the number of marked points from the correlator
notation because the genus is always $0$ and the number of marked
points is the same as the
number of insertions. The insertion of all
$\phi_{b'}$ with $b'\in B'$ is denoted by $\phi_{B'}$ and the
insertions of all $e^{l'}$ with $l'\in L'$ is denoted by
$e^{L'}$. Similar conventions apply for $\phi_{B''}$ and
$e^{L''}$ in the second correlator. The first correlator has $2+m'+s'$ insertions while
the second one $1+m''+s''$, where $m'$, $m''$, $s'$, and $s''$ are
respectively the number of elements
of respectively $B'$, $B''$, $L'$, and $L''$. We have to prove that if
the 3 conditions in the proposition
are satisfied where $k$ should be replaced by $k+1$, then the above
sum is $0$. We will refer to the correlator involving~$B'$ and~$L'$ as
the first correlator and to the correlator involving~$B''$ and~$L''$
as the second correlator. We will prove that for each term in the
above sum either the first or the second correlator vanishes. The
proof will be divided into 4 cases.

{\em Case $1$:} if $\beta'=0$ and the second correlator does not satisfy condition (ii), that is,
${\rm wt}(\phi^c)+\sum_{l''\in L''} (l''-1) \leq 0$ and $d''\leq 0$. Note that since $\beta''=\beta\neq 0$, the second correlator satisfies condition~(i). Since $\beta'=0$ we need to consider only $c$, such that,
$\phi_c|_{E}\neq 0$ and hence ${\phi_c\in \big\{ 1, e, \dots, e^{n-1}\big\}}$. Moreover, the weight of $\phi^c$ is $0$
so \smash{$\phi_c\in \big\{ 1, e^{n-1}\big\}$} and $\phi^c\in \{\phi_N, e\}$. Since $l''\geq 2$ for all $l''\in L''$ the set
$L''$ must be empty. The corresponding term in the sum in this case takes the form
\[
\big\langle \phi_a\psi^k, \phi_c,1,\dots,1, e^{l_1},\dots, e^{l_s}\big\rangle_{d'\ell}
\langle \phi^c, \phi_{B''}\rangle_{\beta +d''\ell},
\]
where the insertions from $\phi_{B'}$ all must be $1$ otherwise ${\phi_b}|_E=0$ and the correlator vanishes.
Using the dimension formula, we get
\[
\operatorname{deg}(\phi_a) + k + \operatorname{deg}(\phi_c) +\sum_{i=1}^s l_i =
(d'+1)(n-1) + s + m'.
\]
Note that $\phi_a$ must satisfy $\phi_a|_E\neq 0$, otherwise the correlator is 0. Therefore,
$\phi_a\in \big\{ 1,e,\dots,e^{n-1} \big\}$ which implies that $\operatorname{deg}(\phi_a)\leq \operatorname{wt}(\phi_a)+1$ with inequality only if $\phi_a=1$. We get
\begin{align*}
\begin{aligned}
(d'+1)(n-1) + m' &=
\operatorname{deg}(\phi_a) + k + \operatorname{deg}(\phi_c) +\sum_{i=1}^s (l_i-1)\\
&\leq
\operatorname{wt}(\phi_a)+1+k + \sum_{i=1}^s (l_i-1) + \operatorname{deg}(\phi_c).
\end{aligned}
\end{align*}
On the other hand, let us recall that the correlator \eqref{van-cor} (with $k+1$ instead of $k$) satisfies condition (iii), that is,
\[
\operatorname{wt}(\phi_a)+\sum_{i=1}^s (l_i-1)<(d+1)(n-1)-k-1.
\]
We get $(d'+1)(n-1) + m' < \operatorname{deg}(\phi_c) + (d+1)(n-1)$. Recall that there are two possibilities for~$\phi_c\colon \phi_c=1$ or $\phi_c=e^{n-1}$. In the first case, we get $0\leq m'<d''(n-1)$ and hence $d''>0$ contradicting our assumption that $d''\leq 0$. In the second case, we get $0\leq m'<(d''+1)(n-1)$. This implies that $d''>-1$ which together with $d''\leq 0$ implies that $d''=0$. However, since $\phi^c= e$, we get that the second correlator vanishes by the divisor equation. This completes the proof of our claim in Case 1.

{\em Case $2$:} $\beta'=0$ and the second correlator satisfies condition (ii). Since $\beta''=\beta\neq 0$, the second correlator satisfies condition (i) too, so it will vanish unless condition (iii) fails, that is,
\[
\operatorname{wt}(\phi^c)+\sum_{l''\in L''} (l''-1) \geq (d''+1)(n-1).
\]
On the other hand, similarly to Case~1, we must have $\phi_{b'}=1$ for all $b'\in B'$, so the dimension formula applied to the first correlator yields
\[
\operatorname{deg}(\phi_a)+k + \operatorname{deg}(\phi_c)+
\sum_{l'\in L'} (l'-1) = (d'+1) (n-1) +m'.
\]
Adding up the above inequality and identity, we get
\[
\operatorname{deg}(\phi_a)+k + \operatorname{deg}(\phi_c)+
\operatorname{wt}(\phi^c)+
\sum_{i=1}^s (l_i-1) \geq (d+1)(n-1) +n-1+m'.
\]
Again $\operatorname{deg}(\phi_a)\leq \operatorname{wt}(\phi_a)+1$, so
\[
m'+n-1 - \operatorname{deg}(\phi_c)-\operatorname{wt}(\phi^c)\leq
\operatorname{wt}(\phi_a)+1+k +\sum_{i=1}^s (l_i-1) -(d+1)(n-1).
\]
Recalling again condition (iii), we get that the right-hand side of the above inequality is $<0$, and hence
$m'+n-1 < \operatorname{deg}(\phi_c)+\operatorname{wt}(\phi^c). $
Similarly to Case~1, we may assume that $\phi_c|_E\neq 0$, that is, $\phi_c\in \big\{ 1, e, \dots, e^{n-1}\big\}$ which implies that $\operatorname{deg}(\phi_c)+\operatorname{wt}(\phi^c)\leq n-1$. This is a contradiction with
$
m'+n-1 < \operatorname{deg}(\phi_c)+\operatorname{wt}(\phi^c)$.

{\em Case $3$:} if $\beta'\neq 0$ and the first correlator does not satisfy condition (ii), that is,
$\operatorname{wt}(\phi_a)+\operatorname{wt}(\phi_c) +\sum_{l'\in L'} (l'-1) =0$ and $d'\leq 0$. Note that we must have $L'=\varnothing$ and either $d''>0$ or~${\sum_{i=1}^s (l_i-1)>0}$. Therefore, the second correlator satisfies condition (ii).

Suppose that $\beta''=0$ ($\Leftrightarrow$ condition (i) fails). We must have $\phi_{b''}|_E\neq 0$ for all $b''\in B''$ $\Rightarrow$ $\phi_{b''}=1$ for all $b''\in B''$. Recalling the dimension formula for the second correlator, we get
\[
\operatorname{deg}(\phi^c)+\sum_{i=1}^{s} (l_i-1) = (d''+1)(n-1) + m''-1.
\]
On the other hand, since $\operatorname{wt}(\phi_a)=0$ for the case under consideration, condition (iii) implies that $\sum_{i=1}^s (l_i-1) < (d+1)(n-1) -k-1$. Combining this estimate with the above equality, we get~$m''+k < \operatorname{deg}(\phi^c) + d'(n-1).$ If $m''>0$, then the second correlator has at least one insertion by~$1$~($\because B''\neq \varnothing$). Since the second correlator does not have descendants it will vanish unless~$d''=0$. However, if $\beta''=d''=0$ the second correlator is non-zero only if the number of insertion is $3$ because the moduli space is $\overline{\mathcal{M}}_{0,1+m''+s}\times \operatorname{Bl}(X)$, that is, $m''=s=1$. Moreover,~$\phi^c \cup e^{l_1}$ up to a constant must be $\phi_N$ hence $\phi^c = e^{n-l_1}$ and $\phi_c= e^{l_1}$. However, $l_1\geq 2$ by definition, so~$\operatorname{wt}(\phi_c)=l_1-1>0$ --- contradicting the assumption that the first correlator does not satisfy condition~(ii). We get $m''=0$ and the estimate that we did above yields $k<\operatorname{deg}(\phi^c) + d'(n-1).$ Note that $\operatorname{deg}(\phi^c)\leq n-1$. Therefore, $d'>-1$. Recall that we are assuming that $d'\leq 0$, so~$d'=0$. Recalling the divisor equation we get $\phi_c\neq e$. Since $\beta''=0$ the restriction $\phi^c|_E\neq 0$ hence $\phi^c\in \big\{1,e,\dots,e^{n-1}\big\}$. Moreover, $\phi^c\neq 1$ thanks to the string equation. We get $\phi_c=e^l$ for some $l\geq 2$ contradicting the assumption that $\operatorname{wt}(\phi_c)=0$.

Suppose now that $\beta''\neq 0$. Then the second correlator satisfies both conditions (i) and (ii). Therefore, condition (iii) must fail, that is,
\begin{gather}\label{weight-c}
\operatorname{wt}(\phi^c) + \sum_{i=1}^s (l_i-1) \geq (d''+1)(n-1).
\end{gather}
Using that $\operatorname{wt}(\phi^c)\leq n-2$ and $\sum_i (l_i-1)<(d+1)(n-1)-k-1$, we get
\[
(d''+1)(n-1) < n-2 + (d+1)(n-1)-k-1,
\]
which implies that $k+1<d'(n-1) +n-2.$ In particular, $d'>-1$ and since by assumption $d'\leq 0$ we get $d'=0$. If $\phi_c=e$, then using the divisor equation we get
\[
\big\langle \phi_a \psi^k, \phi_c, \phi_{B'}\big\rangle_{\beta'} =
\big\langle e\cup \phi_a \psi^{k-1}, \phi_{B'}\big\rangle_{\beta'}.
\]
Since $\operatorname{wt}(\phi_a)=0$ the cup product $e\cup \phi_a\neq 0$ only if $\phi_a=e$. This however implies that $e\cup \phi_a=e^2$ has positive weight and hence the correlator on the right-hand side of the above identity satisfies both conditions (i) and (ii). Condition (iii) must fail, so $1\geq n-1-(k-1) =n-k$, that is, $k\geq n-1$. On the other hand, recall that we already have the estimate $k+1< d'(n-1) + n-2=n-2$ which contradicts the inequality in the previous sentence. We get $\phi_c\neq e$ which together with~$\operatorname{wt}(\phi_c)=0$ implies that
$\phi^c \notin \big\{e^2,\dots, e^{n-1}\big\}$ and hence $\operatorname{wt}(\phi^c)=0$. Recalling \eqref{weight-c}, we get
\[
(d''+1)(n-1) \leq \sum_{i=1}^{s} (l_i-1) < (d+1)(n-1) -k-1.
\]
Since $d'=0$, we get $0<-k-1$ which is clearly a contradiction. This completes the proof of the vanishing claim in Case 3.

{\em Case $4$:} if $\beta'\neq 0$ and the first correlator satisfies condition (ii). Then condition (iii) for the first correlator must fail, that is,
\begin{equation}\label{cor1_lower_bound4}
\operatorname{wt}(\phi_a) + \operatorname{wt}(\phi_c) +\sum_{l'\in L'} (l'-1)
\geq (d'+1)(n-1) - k.
\end{equation}
We claim that the second correlator also satisfies conditions (i) and (ii). Indeed, suppose that~(i) is not satisfied, that is, $\beta''=0$. All insertions in $\phi_{B''}$ must be $1$. Recalling the dimension formula, we get
\[
\operatorname{deg}(\phi^c)+\sum_{l''\in L''} (l''-1) = (d''+1)(n-1) -1 + m''.
\]
Adding up the above identity and the inequality \eqref{cor1_lower_bound4}, we get
\begin{align*}
\begin{aligned}
\operatorname{wt}(\phi_a) + \operatorname{deg}(\phi^c)+ \operatorname{wt}(\phi_c) +
\sum_{i=1}^s (l_i-1) &\geq (d+2)(n-1) - k-1+m'' \\
&=
(d+1)(n-1) -k-1 + n-1+m'',
\end{aligned}
\end{align*}
which is equivalent to
\[
n-1+m''- \operatorname{deg}(\phi^c)- \operatorname{wt}(\phi_c)\leq
\operatorname{wt}(\phi_a) + \sum_{i=1}^s (l_i-1)+k+1 - (d+1)(n-1)<0,
\]
where for the last inequality we used that the correlator whose vanishing we want to prove satisfies condition (iii). We get $m''+n\leq \operatorname{deg}(\phi^c)+ \operatorname{wt}(\phi_c)$. On the other hand, since $\phi^c|_{E}\neq 0$, we have $\phi^c\in \big\{1,e,\dots, e^{n-1}\big\}$ which implies that $\operatorname{deg}(\phi^c)+ \operatorname{wt}(\phi_c)\leq n-1$ -- contradiction. This proves that $\beta''\neq 0$.

Suppose that the second correlator does not satisfy condition (ii). Then, $d''\leq 0$, $\operatorname{wt}(\phi^c)=0$, and $L''=\varnothing$. Since $L''=\varnothing$, the inequality \eqref{cor1_lower_bound4} takes the form
\begin{align*}
\operatorname{wt}(\phi_a) + \operatorname{wt}(\phi_c) +\sum_{i=1}^s (l_i-1)
\geq (d'+1)(n-1) - k.
\end{align*}
On the other hand, recalling again condition (iii) for the correlator whose vanishing we wish to prove, we get \[
\operatorname{wt}(\phi_a) +\sum_{i=1}^s (l_i-1) <(d+1)(n-1)-k-1.\] Combining with the above estimate, we get
\[
(d'+1)(n-1)-k < \operatorname{wt}(\phi_c) + (d+1)(n-1)-k-1,
\]
which becomes $(-d'')(n-1)< \operatorname{wt}(\phi_c) -1$. Since $\operatorname{wt}(\phi^c)=0$ and $d''\leq 0$ the above inequality is possible only if $\phi^c=e$. Then we get $d''\neq 0$ thanks to the divisor equation, that is, $d''\leq -1$ and hence $ \operatorname{wt}(\phi_c) -1> n-1$. This is a contradiction because the maximal possible value of $\operatorname{wt}(\phi_c)$ is $n-2$. This completes the proof of our claim that the second correlator satisfies conditions (i) and (ii).

Finally, in order for the second correlator to be non-zero, condition (iii) must fail. We get
\[
\operatorname{wt}(\phi^c) + \sum_{l''\in L''} (l''-1) \geq (d''+1)(n-1).
\]
Adding up the above inequality and \eqref{cor1_lower_bound4}, we get
\[
\operatorname{wt}(\phi_a) + \operatorname{wt}(\phi_c) + \operatorname{wt}(\phi^c)+
\sum_{i=1}^s (l_i-1) \geq (d+2)(n-1) -k = (d+1)(n-1)-k-1 + n.
\]
On the other hand, recalling the inequality
\[
\operatorname{wt}(\phi_a)+ \sum_{i=1}^s (l_i-1) < (d+1)(n-1)-k-1,
\]
 we~get
\[
n-\operatorname{wt}(\phi_c) - \operatorname{wt}(\phi^c) < 0.
\]
The inequality clearly does not hold if one of the weights is $0$. If both weights are non-zero, then we will have $\operatorname{wt}(\phi_c) + \operatorname{wt}(\phi^c)= n-2$ which again contradicts the above inequality. The conclusion is that either the first or the second correlator satisfies condition (iii) and hence one of the two correlators must vanish according to the inductive assumption.
\end{proof}

\section{Second structure connection and blowups}
\label{sec:2nd_sc_blowups}

Let us recall the notation already fixed in
Sections \ref{sec:coh_blowup} and \ref{sec:qcoh_blowup}. From now on, for a complex variety $Y$, we denote by $H(Y):=H^*(Y,\CC)$
and $\widetilde{H}(Y)$ respectively the cohomology and the reduced
cohomology of $Y$ with complex coefficients. Using the direct sum decomposition~$H(\operatorname{Bl}(X))=H(X)\oplus \widetilde{H}(E)$, we define the $H(X)$-component (resp.\ $\widetilde{H}(E)$-component) of a~vector $v\in H(\operatorname{Bl}(X))$ to be the projection of $v$ on $H(X)$ (resp.\ $\widetilde{H}(E)$).

We will view quantum cohomology of $\operatorname{Bl}(X)$ as a family of
Frobenius manifolds parametrized by the Novikov variables $q=(q_1,\dots,q_{r+1})\in (\CC^*)^{r+1}$ defined in Section \ref{sec:qcoh_blowup}. Recall that ${q_{r+1}=Q^{n-1}}$. We will be
interested in the Laurent series expansion of the second structure
connection of $\operatorname{Bl}(X)$ with respect to $Q$ at $Q=0$, while the
remaining parameters $q_1,\dots, q_r$ remain fixed. The main goal in
this section is to determine the leading order terms of this expansion.

\subsection{Period vectors for the blowup}
Let us denote
by $\widetilde{\rho}$ and $\rho$ the operators of classical cup product
multiplications by respectively~$c_1(T\operatorname{Bl}(X))$ and
$c_1(TX)$. Let $\widetilde{\theta}$ and $\theta$ be the grading
operators of the Frobenius structures underlying the quantum
cohomologies of respectively $\operatorname{Bl}(X)$ and $X$ (see
\eqref{qcoh_hgo}). In the lemma below, we will need the following
notation. Suppose that $A\colon\CC\to \op{End}(H)$ is
an $\op{End}(H)$-valued smooth function. Let $m$ be the standard
coordinate on $\CC$. We define {\em right} differential operator on
$\CC$, to be a formal expression of the form
\begin{equation}\label{right-dop}
L\big(m,\overleftarrow{\partial}_m\big)=
\sum_{k=0}^r B_k(m)
\overleftarrow{\partial}_m^k,
\end{equation}
where the coefficients $B_k(m)\in
\op{End}(H)$ depend smoothly on $m$. We define the action of
$L$ on~$A$~by
\[
A(m) L\big(m,\overleftarrow{\partial}_m\big) := \sum_{k=0}^r \partial_m^k (
 A(m) \circ B_k(m)  ),
\]
where $\circ$ is the composition operation in $\op{End}(V)$. Given two right
differential operators $L_1$ and~$L_2$, there exists a unique right
differential operator $L_1\circ L_2$, such that,
\[
\big(A(m) L_1\big(m,\overleftarrow{\partial}_m\big)\big)
L_2\big(m,\overleftarrow{\partial}_m\big) =:
A(m) (L_1\circ L_2)\big(m,\overleftarrow{\partial}_m\big).
\]
We say that $L_1\circ L_2$ is the composition of $L_1$ and $L_2$. One
can check that this operation is associative and therefore, the set of
all right differential operators of the form \eqref{right-dop} has
a~structure of an associative algebra acting from the right on the
space of smooth $\op{End}(H)$-valued functions on $\CC$.
\begin{Lemma}\label{le:cp-homog}\quad
\begin{itemize}\itemsep=0pt
\item[$(a)$] The following formula holds:
\[
\widetilde{I}^{(-m)}(\lambda) :=
e^{\widetilde{\rho}\partial_\lambda\partial_m}
\left(
\frac{\lambda^{\widetilde{\theta}+m-1/2}}{
\Gamma\big(\widetilde{\theta}+m+1/2\big)} \right) =
\left(
\frac{\lambda^{\widetilde{\theta}+m-1/2}}{
\Gamma\big(\widetilde{\theta}+m+1/2\big)} \right)
e^{\widetilde{\rho} \overleftarrow{\partial}_m},
\]
where the first identity is just a definition and $\overleftarrow{\partial}_m$
denotes the right action by a derivation with respect to $m$.
\item[$(b)$]
The following identity holds:
\[
\widetilde{I}^{(-m)}\big(Q^{-1}\lambda\big) =
\left(
\frac{\lambda^{\widetilde{\theta}+m-1/2}}{
\Gamma\big(\widetilde{\theta}+m+1/2\big)}
\right)
e^{Q\widetilde{\rho} \overleftarrow{\partial}_m}
Q^{-(\widetilde{\theta}+m-1/2)}
Q^{-\widetilde{\rho}}.
\]
\end{itemize}
\end{Lemma}
\begin{proof}
(a)
By definition
\begin{align*}
e^{\widetilde{\rho}\partial_\lambda\partial_m}
\left(
\frac{\lambda^{\widetilde{\theta}+m-1/2}}{
\Gamma\big(\widetilde{\theta}+m+1/2\big)} \right) &=
\sum_{k=0}^\infty
\frac{1}{k!} \widetilde{\rho}^k \partial_m^k
\left(
\frac{\lambda^{\widetilde{\theta}+m-k-1/2}}{
\Gamma\big(\widetilde{\theta}+m-k+1/2\big)} \right) \\
&=
\sum_{k=0}^\infty
\frac{1}{k!} \partial_m^k
\left(
\frac{\lambda^{\widetilde{\theta}+m-1/2}}{
\Gamma\big(\widetilde{\theta}+m+1/2\big)} \right)
 \widetilde{\rho}^k,
\end{align*}
where we used that
$\widetilde{\rho} \widetilde{\theta} =
\big(\widetilde{\theta} +1\big) \widetilde{\rho}
$. The above expression is by definition the right action of~$e^{\widetilde{\rho} \overleftarrow{\partial}_m} $ on
\smash{${\lambda^{\widetilde{\theta}+m-1/2}}/{
\Gamma \big(\widetilde{\theta}+m+1/2\big)} $}.

(b) Using the formula from part (a), we get that the identity that we
have to prove is equivalent to the following conjugation formulas:
\[
Q^{\widetilde{\theta}+m} Q^{-Q\widetilde{\rho}}
Q^{-(\widetilde{\theta}+m)} = Q^{-\widetilde{\rho}},
\]
and
\begin{equation}\label{conj_Q}
Q^{-(\widetilde{\theta}+m)}
e^{\widetilde{\rho} \overleftarrow{\partial}_m}
Q^{\widetilde{\theta}+m} =
Q^{-Q \widetilde{\rho}}
e^{Q \widetilde{\rho} \overleftarrow{\partial}_m}.
\end{equation}
The first identity follows easily from
$\big[\widetilde{\theta},\widetilde{\rho}\big]=-\widetilde{\rho} $. Let us
prove \eqref{conj_Q}. To begin with, note that this is an
identity between operators acting from the right. We will use the
following fact. Suppose that we have an associative algebra
$\mathcal{A}$ acting on a
vector space $V$ from the right, that is, $v\cdot (AB) = (v\cdot
A)\cdot B$ for all $A,B\in \mathcal{A}$ and $v\in V$. Then the
following formula holds:
\begin{equation}\label{conj_formula}
v\cdot \big(e^A B e^{-A}\big) = v\cdot
\big(e^{\operatorname{ad}_A}(B)\big),
\end{equation}
where $\operatorname{ad}_A(X)=AX-XA$. In our case, $\mathcal{A}$ is
the algebra of differential operators in $m$ with coefficients in
$\operatorname{End}(H^*(\operatorname{Bl}(X)))$, that is, as a vector space
$\mathcal{A}$ consists of elements of the form
\[
\sum_{k=0}^{k_0} c_k(m) \overleftarrow{\partial}_m^k,\qquad
c_k(m)\in \operatorname{End}(H^*(\operatorname{Bl}(X))),
\]
and the product in $\mathcal{A}$ is determined by the natural
composition of endomorphisms of $H^*(\operatorname{Bl}(X))$ and the
commutation relation \smash{$\big[m, \overleftarrow{\partial}_m\big]=1$}.
Let us apply the conjugation formula \eqref{conj_formula} to
\eqref{conj_Q}. The main difficulty is to prove that
\begin{equation}\label{ad^k}
\operatorname{ad}_{\widetilde{\theta}+m}^k\big(
\widetilde{\rho} \overleftarrow{\partial}_m\big) =
(-1)^{k} \widetilde{\rho} \overleftarrow{\partial}_m+
(-1)^{k-1} k \widetilde{\rho},
\end{equation}
for all $k\geq 0$. We argue by induction on $k$. For $k=0$, the
identity is true. Suppose that it is true for $k$. Then we get
\begin{align*}
\big[\widetilde{\theta}+m,
(-1)^k \widetilde{\rho} \overleftarrow{\partial}_m +
(-1)^{k-1} k \widetilde{\rho}\big] &=
(-1)^k\big(-\widetilde{\rho} \overleftarrow{\partial}_m +
\widetilde{\rho}\big) +(-1)^k k \widetilde{\rho}\\
&=
(-1)^{k+1} \widetilde{\rho} \overleftarrow{\partial}_m+
(-1)^{k} (k+1) \widetilde{\rho},
\end{align*}
where we used that
$\big[\widetilde{\theta},\widetilde{\rho}\big]=-\widetilde{\rho}$ and
$\big[m,\overleftarrow{\partial}_m\big]=1$.
Using the conjugation formula \eqref{conj_formula} and formula
\eqref{ad^k}, we get that the left-hand side of \eqref{conj_Q} is equal to
\[
\exp\left(
\sum_{k=0}^\infty \frac{1}{k!} (-\log Q)^k\big(
(-1)^{k} \widetilde{\rho} \overleftarrow{\partial}_m+
(-1)^{k-1} k \widetilde{\rho}\big)\right) =
e^{Q \widetilde{\rho} \overleftarrow{\partial}_m}
e^{-(\log Q) Q \widetilde{\rho}},
\]
which is the same as the right-hand side of \eqref{conj_Q}.
\end{proof}

Put $q=(q_1,\dots, q_r)$ and let $S(t,q,Q,z)$ be the calibration of
the blowup $\operatorname{Bl}(X)$. Let us recall the fundamental solution of
the second structure connection for the quantum cohomology of~$\operatorname{Bl}(X)$
\[
I^{(-m)}(t,q,Q,\lambda) = \sum_{k=0}^\infty (-1)^k
S_k(t,q,Q) \partial_\lambda^k
\widetilde{I}^{(-m)}(\lambda).
\]
Recalling Lemma \ref{le:cp-homog}, we get
\begin{align}
&
I^{(-m)}\big(t,q,Q,Q^{-1}\lambda\big)
Q^{\widetilde{\rho}}
Q^{\widetilde{\theta}+m-1/2} \nonumber \\
&\qquad=
\sum_{i,l=0}^\infty (-1)^l
Q^l S_l(t,q,Q) \frac{\partial^i_m}{i!} \left(
\frac{\lambda^{\widetilde{\theta}+m-l-1/2}}{
\Gamma\big(\widetilde{\theta}+m-l+1/2\big)} (Q \widetilde{\rho})^i
\right).\label{pv-homog}
\end{align}
Let us extend the Hodge grading operator $\theta$ of $X$ to
$H^*(\operatorname{Bl}(X))$ in such a way that
\smash{$\theta\big(e^k\big) = \tfrac{n}{2} e^k$}
for all $1\leq k\leq n-1$. Let $\Delta:=\widetilde{\theta}-\theta$,
then in the basis
\begin{equation}\label{basis:blup_coh}
\phi_i,\quad 1\leq i\leq N, \qquad
e^k,\quad 1\leq k\leq n-1
\end{equation}
of $H^*(\operatorname{Bl}(X))$, the operator $\Delta$ takes the form
\begin{align*}
\Delta(\phi_i) & = 0, \quad 1\leq i\leq N,\qquad
\Delta\big(e^k\big) = -k e^k, \quad 1\leq k\leq n-1.
\end{align*}
Let us point out that the basis of $H^*(\operatorname{Bl}(X))$ dual to the basis
\eqref{basis:blup_coh} with respect to the Poincar\'e pairing is given by
$\phi^i$ ($1\leq i\leq N$),
$e_k$ ($1\leq k\leq n-1$),
where $\phi^i$ ($1\leq i\leq N$) is a basis of $H^*(X)$
dual to $\phi_i$ ($1\leq i\leq N$) with respect to the Poincar\'e pairing
(on $X$) and $e_k:= (-1)^{n-1} e^{n-k}$.
\begin{Lemma}\label{le:calibr_Q-homog}
Suppose that $t=\sum_{b=2}^N t_b \phi_b\in H^*(X)$ and that $l\geq 1$. Then
\begin{align*}
Q^\Delta Q^{k+l} \big(S_l(t,q,Q) e^k\big) ={} &
\sum_{k''=1}^{n-1} \sum_{d=0}^\infty
\big\langle \psi^{l-1} e^k, e^{k''}\big\rangle_{0,2,d\ell} e_{k''} +
\big\langle \psi^{l-1} e^k, 1\big\rangle_{0,2,d\ell} Q^n (-1)^{n-1}\phi_N \\
& +O\big(Q\widetilde{H}(E) + Q^{n+1} H(X)\big),
\end{align*}
where the $O$-term denotes a power series in $Q$ with values in $H(\operatorname{Bl}(X))$ whose $\widetilde{H}(E)$-component involves only positive
powers of $Q$ and its $H(X)$-component involves only powers of $Q$ of degree~${\geq n+1}$.
\end{Lemma}
\begin{proof}
Recall that every $\widetilde{\beta}\in \operatorname{Eff}(\operatorname{Bl}(X))$
has the form $\widetilde{\beta}=\beta + d\ell$ for some
$\beta \in \operatorname{Eff}(X)$ and $d\in \ZZ$. Recalling the
definition of the calibration, we get
\begin{align*}
Q^\Delta Q^{k+l} \big(S_l(t,q,Q) e^k\big) ={}&
Q^\Delta Q^{k+l}
\sum_{\widetilde{\beta}\in
\operatorname{Eff}(\operatorname{Bl}(X)) }
\left(
\sum_{b=1}^N
\big\langle e^k \psi^{l-1},\phi_b\big\rangle_{0,2,\beta+d\ell}(t)
\phi^b\right.\\
&\left. +
\sum_{k''=1}^{n-1}
\big\langle e^k \psi^{l-1}, e^{k''}\big\rangle_{0,2,\beta+d\ell}(t)
e_{k''} \right) q^\beta Q^{-d (n-1)}.
\end{align*}
Let us examine first the correlator
\[
\big\langle e^k \psi^{l-1},\phi_b\big\rangle_{0,2,\beta+d\ell}(t) =
\sum_{r=0}^\infty \sum_{b_1,\dots,b_r=2}^N
\big\langle e^k \psi^{l-1},\phi_b,\phi_{b_1},\dots,\phi_{b_r}
\big\rangle_{0,2+r,\beta+d\ell}
t_{b_1}\cdots t_{b_r}.
\]
There are 3 cases.

{\em Case $1$:} if $\beta=0$. Then the correlator is a twisted GW
invariant of $E$ and since $\phi_b|_E=0$ for~$2\leq b\leq N$, we may
assume that $b=1$, that is, $\phi_b=\phi_1=1$. For similar reasons, we
may assume that $r=0$. The dimension of the
virtual fundamental cycle of
$\overline{\mathcal{M}}_{0,2}(\operatorname{Bl}(X),d\ell)$ is
$(d+1)(n-1)$. Therefore, $k+l-1=(d+1)(n-1)$ or equivalently
$k+l-d(n-1)=n$, that is, in this case the correlator coincides with
the term of order $Q^n$ on the right-hand side of the formula that we want to prove.

{\em Case $2$:} if $\beta\neq 0$ and the correlator does not satisfy
condition (ii) in Gathmann's vanishing theorem. Then $d\leq 0$ and
$k-1\leq 0$, that is, $k=1$. If $d\leq -1$, then
$k+l-d(n-1)\geq k+l+n-1\geq n+1$, so the correlator contributes to the
terms of order $O\big(Q^{n+1}H(X)\big)$. It remains to consider the case when $k=1$ and $d=0$. We will prove that if the correlator \smash{$\big\langle e\psi^{l-1},\phi_b\big\rangle_{0,2,\beta}(t)$} is non-zero, then $l\geq n$.

First, by using the divisor equation we may reduce to the cases when $2\leq b\leq N$. Indeed, if $b=1$ then by using the string equation
\smash{$\big\langle e\psi^{l-1},1\big\rangle_{0,2,\beta}(t)=
\big\langle e\psi^{l-2}\big\rangle_{0,1,\beta}(t)$}. Since $\beta\neq 0$, there exists a divisor class $p\in H^2(X)$, such that, $\int_\beta p\neq 0$. Recalling the divisor equation, we get~\smash{$\big\langle e \psi^{l-2},p\big\rangle_{0,2,\beta} =
\int_\beta p \big\langle e\psi^{l-2}\big\rangle_{0,1,\beta}$}, that is, the correlator for $b=1$ can be expressed in terms of correlators involving only $2\leq b\leq N$.

Suppose now that $2\leq b\leq N$. Recalling the divisor equation, we get
\[
\big\langle e\psi^{l-1},\phi_b\big\rangle_{0,2,\beta}(t) =
\big\langle e,\psi^{l},\phi_b\big\rangle_{0,3,\beta}(t),
\]
which according to the topological recursion relations (TRR) can be written as
\begin{gather*}
\sum_{\beta'+\beta''=\beta} \sum_{d\in \ZZ} \left(
\sum_{j=1}^{n-1}
\big\langle \psi^{l-1},e^j\big\rangle_{0,2,\beta'+d\ell}(t)
\langle e_j, e,\phi_b\rangle_{0,3,\beta''-d\ell}(t)\right. \\
\left.\qquad+
\sum_{a=1}^{N}
\big\langle \psi^{l-1},\phi_a\big\rangle_{0,2,\beta'+d\ell}(t)
\langle \phi^a, e,\phi_b\rangle_{0,3,\beta''-d\ell}(t)
\right).
\end{gather*}
Let us consider the two correlators that involve $\beta'$. If
$\beta'=0$, then in both correlators $d>0$ and in the second
correlator $\phi_a=1$. Since $t|_E=0$ we may assume that $t=0$. By the
dimension formula, we get $l-1+j=(d+1)(n-1)$ for the 1st correlator and
$l-1=(d+1)(n-1)$ for the second correlator. In both cases, since
$d\geq 1$ and $n\geq 2$, we have $l\geq n$. Suppose now that~$\beta'\neq 0$ and that the 1st (resp.\ second) correlator does not
satisfy condition (ii) in Gathmann's vanishing theorem, that is~$j=1$
and $d\leq 0$. Note that for the correlators involving $\beta''$ we
must have~$\beta''\neq 0$ because~${\phi_b}|_{E}=0$ and~$d\neq 0$ due
to the divisor equation for the divisor class $e$. We get $d<0$. Therefore, the correlator involving $\beta''$ satisfies both conditions (i) and (ii) in Gathmann's vanishing theorem. Therefore, condition (iii) must fail, that is, $n-j-1\geq (-d+1)(n-1)$ and $0\geq (-d+1)(n-1)$. Since $-d\geq 1$, both inequalities lead to a contradiction. It remains only the possibility that $\beta'\neq 0$ and that the correlators involving $\beta'$ satisfy condition (ii) in Gathmann's vanishing theorem. Then condition (iii) must fail, so
$j-1+l-1\geq (d+1)(n-1)$ and $l-1\geq (d+1)(n-1)$. If $d\geq 1$, then these inequalities will imply that $l\geq n+1$. If $d=0$, then $\beta''=0$, otherwise the correlator involving $\beta''$ will be $0$ by the divisor equation. But then the correlator becomes \smash{$\int_{\operatorname{Bl}(X)} e_j\cup e\cup \phi_b$} which is $0$ because $\phi_b\cup e=0$. Finally, if $d\leq -1$, then, since $\phi_b|_{E}=0$ we must have $\beta''\neq 0$, so the correlator involving $\beta''$ satisfies conditions (i) and (ii) in Gathmann's vanishing theorem. The 3rd condition must fail, that is, $n-j-1\geq (-d+1)(n-1)$ and $0\geq (-d+1)(n-1)$. Both inequalities are impossible and this completes the analysis in the second case.

{\em Case $3$:} if $\beta\neq 0$ and the correlator does satisfy
condition (ii). Then condition (iii) in Gathmann's
vanishing theorem does not hold, that is,
\[
k-1 \geq (d+1)(n-1) - l+1 = d(n-1) + n-l.
\]
This inequality is equivalent to $k+l-d(n-1)\geq n+1$. We get that the
correlators in this case contribute to the terms of order
$O\big(Q^{n+1}\big)$.

Similarly, let us examine the correlators
\[
\big\langle e^k \psi^{l-1}, e^{k''}\big\rangle_{0,2,\beta+d\ell}(t) =
\sum_{r=0}^\infty
\sum_{b_1,\dots,b_r=2}^N
\big\langle e^k \psi^{l-1}, e^{k''},\phi_{b_1},\dots,\phi_{b_r}
\big\rangle_{0,2+r,\beta+d\ell}
t_{b_1}\cdots t_{b_r}.
\]
Since $\Delta(e_{k''}) = -(n-k'') e_{k''}$, we get that we have to
prove that if the above correlator is non-zero, then
$k+l-d(n-1) -(n-k'')\geq 0$ and that if the equality holds, then $r=0$
and $\beta=0$. Again we will consider 3 cases.

{\em Case $1$}: if $\beta=0$. Just like above,
$r=0$ because we can identify the correlator with a~twisted GW
invariant of the exceptional divisor and the restriction of
$\phi_{b_i}$ to $E$ is 0. Using the dimension formula for the virtual
fundamental cycle, we get
\[
k+l-1+k'' = \operatorname{dim}\big[
\overline{\mathcal{M}}_{0,2}(\operatorname{Bl}(X),d\ell)\big]^{\textnormal{virt}}
=
n-1+(n-1) d.
\]
We get
\[
k+l-d(n-1) -(n-k'') = k+l+k''-n-d(n-1) = 0.
\]
{\em Cases $2$:} if $\beta\neq 0$ and condition (ii) does not hold. Then
$d\leq 0$ and $k-1+k''-1\leq 0$, that is, $k=k''=1$. If
$d\leq -1$, then
\[
k+l-d(n-1)-(n-k'')\geq k+l +n-1-n+k'' = 1+l\geq 2.
\]
The correlator contributes to the terms of order $O\big(Q^2\big)$. Suppose
that $d=0$. The insertion~${e^{k''}=e}$ can be removed via the divisor
equation, that is, the correlator in front of $t_{b_1}\cdots t_{b_r}$
takes the form
\[
\big\langle e^k \psi^{l-1}, e^{k''},\phi_{b_1},\dots,\phi_{b_r}
\big\rangle_{0,2+r,\beta} =
\big\langle e^2 \psi^{l-2},\phi_{b_1},\dots,\phi_{b_r}
\big\rangle_{0,1+r,\beta}.
\]
The above correlator does satisfy condition (ii) of Proposition~\ref{prop:van}. Therefore, in order to have a non-trivial
contribution, condition (iii) in Gathmann's vanishing theorem must
fail, that is, $1\geq n-1 -l+2$ or equivalently $l\geq n$. We get
\[
k+l-d(n-1) -(n-k'') = 1+l-(n-1) = 2+l-n\geq 2 >0,
\]
so the equality that we need to prove holds.

{\em Case $3$:} if $\beta\neq 0$ and condition (ii) holds. In other words,
conditions (i) and (ii) in Gathmann's vanishing theorem (see
Proposition~\ref{prop:van}) hold for the correlators
\[
\big\langle e^k \psi^{l-1}, e^{k''},\phi_{b_1},\dots,\phi_{b_r}
\big\rangle_{0,2+r,\beta+d\ell}.
\]
Again, in order to have a non-trivial contribution, condition (iii) must fail, so
\[
k-1 + k''-1 \geq (d+1)(n-1)-l+1=d(n-1) +n -l,
\]
or equivalently $k+l+k'' \geq 2+ n + d(n-1)$. We get
\[
k+l-d(n-1) -(n-k'') = k+l+k''-n -d (n-1)\geq 2 >0.
\]
This completes the proof of the lemma.
\end{proof}

Note that
\[
\widetilde{\rho}^i e^k = (c_1(TX)-(n-1) e)^i e^k = (-n+1)^i e^{k+i},
\]
and that
$\widetilde{\theta}\big(e^{k+i}\big) = \big(\tfrac{n}{2}-k-i\big)
e^{k+i}$. Therefore, using formula~\eqref{pv-homog} and Lemma~\ref{le:calibr_Q-homog}, we get the following proposition.
\begin{Proposition}\label{prop:Q-exp_e}
The following formula holds:
\begin{align*}
&
\big( Q^\Delta
I^{(-m)}\big(t,q,Q, Q^{-1}\lambda\big)
Q^{\widetilde{\rho}} Q^{\widetilde{\theta}+m-1/2}
Q^{-\Delta} \big) e^k \\
&=
\sum_{l,d=0}^\infty\!
\sum_{i=0}^{n-1-k}\!
\sum_{k''=1}^{n-1}\!
\big\langle
e^{k+i} \psi^{l-1}, e^{k''}\big\rangle_{0,2,d\ell} e_{k''}
(-\partial_\lambda)^l \frac{(-(n-1)\partial_m)^i}{i!}\!
\left(
\frac{ \lambda^{\frac{n-1}{2}+ m-k-i} }{
\Gamma\big(\frac{n-1}{2}+ m-k-i+1\big) }
\right) \\
&\phantom{=}{}+
\sum_{l,d=0}^\infty
\sum_{i=0}^{n-1-k}
\big\langle
e^{k+i} \psi^{l-1}, 1\big\rangle_{0,2,d\ell} (-1)^{n-1}Q^n \phi_N
(-\partial_\lambda)^l \frac{(-(n-1)\partial_m)^i}{i!}
\\
&\phantom{=+}{}\times\left(
\frac{ \lambda^{\frac{n-1}{2}+ m-k-i} }{
\Gamma\big(\frac{n-1}{2}+ m-k-i+1\big) }
\right) +O\big(Q\widetilde{H}(E)+Q^{n+1}H(X)\big),
\end{align*}
where $1\leq k\leq n-1$ and the notation involving $O$ is the same as in Lemma {\rm\ref{le:calibr_Q-homog}}.
\end{Proposition}

\begin{Proposition}\label{prop:Q-exp_phia}
If $2\leq a\leq N$, then the following formula holds:
\begin{align*}
\big(
Q^\Delta
I^{(-m)}\big(t,q,Q, Q^{-1}\lambda\big)
Q^{\widetilde{\rho}} Q^{\widetilde{\theta}+m-1/2}
Q^{-\Delta} \big) \phi_a =
\frac{
\lambda^{\theta+m-1/2}
}{
\Gamma(\theta+m+1/2)} \phi_a + O(Q).
\end{align*}
Moreover, in the above expansion, the $H(X)$-component of the coefficient in front of $Q^m$ for $0\leq m\leq n-1$ is a Laurent polynomial in $\lambda^{1/2}$
$($with coefficients in $H(X))$.
\end{Proposition}
\proof
Using formula \eqref{pv-homog}, we get that the left-hand side of the formula that we want to prove is equal to
\[
Q^\Delta
\sum_{i,l=0}^\infty
(-1)^l Q^l S_l(t,q,Q)
\frac{\partial_m^i}{i!}\left(
\frac{\lambda^{\widetilde{\theta}+m-l-1/2}}{
\Gamma(\widetilde{\theta}+m-l+1/2)}
(Q\widetilde{\rho})^i \right)
\phi_a.
\]
Since $\widetilde{\rho} \phi_a = \rho \phi_a$, the above formula takes the form
\begin{equation}\label{rho-component}
\sum_{l=0}^\infty
\sum_{i=0}^\infty
Q^{\Delta+l+i} S_l(t,q,Q) \rho^i\phi_a
(-\partial_\lambda)^l
\frac{\partial_m^i}{i!}\left(
\frac{\lambda^{(n-1)/2+m-i-\operatorname{deg}(\phi_a)}}{
\Gamma((n+1)/2+m-i-\operatorname{deg}(\phi_a))}
\right).
\end{equation}
Note that the term in the above double sum corresponding to $l=i=0$ coincides with the leading order term on the right-hand side of the formula that we have to prove. Therefore, recalling the definition of the calibration, we get that we have to prove the following two statements. First, if $l+i>0$, then the following expression
\begin{align*}%\label{desc_cor}
\sum_{b=1}^N\!
\big\langle \rho^i\phi_a \psi^{l-1},\phi_b\big\rangle_{0,2,\beta+d\ell} (t) \phi^b Q^{l+i-d(n-1)} +
\sum_{k=1}^{n-1}\!
\big\langle \rho^i\phi_a \psi^{l-1}, e^k\big\rangle_{0,2,\beta+d\ell} (t) e_k Q^{k+l+i-d(n-1)-n}
\end{align*}
has order at least $O(Q)$ for all $\beta+d\ell\in \operatorname{Eff}(\operatorname{Bl}(X))$. Second, there are only finitely many $d$ and~$l$, such that, in the first sum the coefficient in front of $Q^m$ for $0\leq m\leq n-1$ is non-zero. Let us consider the correlators in the first sum.

{\em Case $1$}: if $\beta=0$. Then the correlator is a twisted GW invariant of the exceptional divisor $E$ and since $\phi_a|_E=0$, we get that the correlator must be $0$.

{\em Case $2$}: if $\beta\neq 0$ and condition (ii) in Gathmann's
vanishing theorem does not hold. Then~${d\leq 0}$ and we get $l+i-d(n-1)\geq l+i>0$. In order for the correlator to contribute to the coefficient in front of $Q^m$ for some $0\leq m\leq n-1$, we must have $-1\leq d\leq 0$ and $0\leq l<n$. Clearly, there are only finitely many $d$ and $l$ satisfying these inequalities.

{\em Case $3$}: if $\beta\neq 0$ and condition (ii) holds. Then condition (iii) in Gathmann's vanishing theorem must fail, that is, $0\geq (d+1)(n-1)-l+1$ or equivalently $l-d(n-1)\geq n$. The power of $Q$ is $l+i-d(n-1)\geq n+i.$ Therefore, the correlators satisfying the conditions of this case contribute only to the coefficients in front of $Q^m$ with $m\geq n$.

The argument for the correlators in the second sum is similar.

{\em Case $1$}: if $\beta=0$. Then the correlator is a twisted GW invariant of the exceptional divisor $E$ and since $\phi_a|_E=0$, we get that the correlator must be $0$.

{\em Case $2$}: if $\beta\neq 0$ and condition (ii) in Gathmann's vanishing theorem does not hold. Then
$d\leq 0$ and $k=1$. The divisor equation implies that if $d=0$, then
the correlator vanishes. Therefore, $d\leq -1$. We get $k+l+i-d(n-1)-n=l+i-(d+1)(n-1)\geq l+i>0$.

{\em Case $3$}: if $\beta\neq 0$ and condition (ii) holds. Then condition (iii) in Gathmann's vanishing theorem must fail, that is, $k-1\geq (d+1)(n-1)-l+1$ or equivalently $k+l-d(n-1)\geq n+1$. The power of $Q$ is $k+l+i-d(n-1)-n\geq i+1\geq 1.$\qed

\begin{Proposition}\label{prop:Q-exp_phi1}
The following formula holds:
\begin{align}
&
\big(
Q^\Delta
I^{(-m)}\big(t,q,Q, Q^{-1}\lambda\big)
Q^{\widetilde{\rho}} Q^{\widetilde{\theta}+m-1/2}
Q^{-\Delta} \big) \phi_1 =
\frac{
\lambda^{\theta+m-1/2}
}{
\Gamma(\theta+m+1/2)} \phi_1\nonumber\\ &\qquad+
\sum_{d,l\geq 0}
\sum_{i=0}^{n-1}
\sum_{k=1}^{n-1}
\big\langle (-(n-1)e)^i \psi^{l-1}, e^k\big\rangle_{0,2,d\ell} e_k (-\partial_\lambda)^l
\frac{\partial_m^i}{i!} \left(
\frac{\lambda^{(n-1)/2+m-i}}{
\Gamma((n+1)/2+m-i) }\right)\!\! \label{Q:exp_phi1_l2}\\
&\qquad+
\sum_{l=1}^\infty
\sum_{\beta\in \operatorname{Eff}(X)}
\big\langle \psi^{l-1},e\big\rangle_{0,2,\beta}(t) e_1
q^\beta Q^{l-n+1}
(-\partial_\lambda)^l \left(
\frac{\lambda^{(n-1)/2+m}}{
\Gamma((n+1)/2+m) }\right)
+ O(Q),\label{Q:exp_phi1_l3}
\end{align}
where $e_k=(-1)^{n-1} e^{n-k}$ and the correlator
\[
\big\langle (-(n-1)e)^i \psi^{l-1}, e^k\big\rangle_{0,2,d\ell}=
\big\langle (-(n-1)e)^i \psi^{l}, e^k,1\big\rangle_{0,3,d\ell}
\]
can be defined also for $l=0.$
\end{Proposition}
\begin{proof}
Note that if $i>0$,then $\widetilde{\rho}^i=\rho^i+(-(n-1)e)^i.$
Just like in the proof of Proposition \ref{prop:Q-exp_phia}, we get that the left-hand side of the identity that we would like to prove is equal to the sum of~\eqref{rho-component} with $a=1$ and
\begin{equation}\label{e-component}
\sum_{l=0}^\infty
\sum_{i=1}^\infty
Q^{\Delta+l+i} S_l(t,q,Q) (-(n-1)e)^i
(-\partial_\lambda)^l
\frac{\partial_m^i}{i!}\left(
\frac{\lambda^{(n-1)/2+m-i}}{
\Gamma((n+1)/2+m-i)}
\right).
\end{equation}
Let us discuss first the contribution of \eqref{rho-component}. The same argument as in the proof of Proposition~\ref{prop:Q-exp_phia} yields that if $i>0$, then the corresponding terms in the sum have order at least $O(Q)$. If~$i=0$ and $l=0$, then the corresponding term in the sum becomes $\lambda^{\theta+m-1/2}/\Gamma(\theta+m+1/2)$ which is precisely the first term on the right-hand side of the formula that we would like to prove.
Finally, we are left with the case $i=0$ and $l\geq 1$. By definition, $Q^{\Delta+l}S_l(t,q,Q)\phi_1$ is
\begin{gather}\label{phi1_desc}
\sum_{\beta+d\ell }\!
\sum_{b=1}^N\!
\big\langle \psi^{l-1},\phi_b\big\rangle_{0,2,\beta+d\ell} (t) \phi^b Q^{l-d(n-1)} +
\!\sum_{k=1}^{n-1}\!
\big\langle \psi^{l-1}, e^k\big\rangle_{0,2,\beta+d\ell} (t) e_k q^\beta Q^{k+l-d(n-1)-n},\!\!\!
\end{gather}
where the first sum is over all effective curve classes $\beta+d\ell\in \operatorname{Eff}(\operatorname{Bl}(X))$. Let us consider the following 3 cases for the correlators in~\eqref{phi1_desc}.

{\em Case $1$:} if $\beta=0$. we may assume that $t=0$ because $t|_E=0$. The sum over $b$ is non-zero only if $\phi_b=1$, that is, $b=1$. Recalling the formula for the dimension of the moduli space, we get
\[
l-1= -1+n+d(n-1)\quad\Rightarrow\quad
l-d(n-1) = n.
\]
Therefore, in this case the contribution has order $O(Q^n)$. The sum over $k$ in \eqref{phi1_desc} (when~$\beta=0$) is independent of $Q$ because by matching the degree of the correlator insertion with the dimension of the virtual fundamental cycle we get $k+l-1=-1+n+d(n-1)$. Therefore, the contribution to the sum \eqref{rho-component} with $a=1$ of the terms with $i=0$, $l\geq 1$, and degree $\beta=0$ is
\[
\sum_{l=1}^\infty
\sum_{d=0}^\infty
\big\langle \psi^{l-1}, e^k\big\rangle_{0,2,d\ell}
e_k
(-\partial_\lambda)^l
\left(
\frac{\lambda^{(n-1)/2+m}}{
\Gamma((n+1)/2+m)}
\right).
\]
Note that the above sum coincides with the $i=0$ component of \eqref{Q:exp_phi1_l2}.

{\em Case $2$:} if $\beta\neq 0$ and condition (ii) in Gathmann's vanishing theorem does not hold. Since ${d\leq 0}$ and $l\geq 1$ the sum over $b$ has order at least $O(Q)$. For the sum over $k$, only for $k=1$ condition~(ii) does not hold and if $d\leq -1$ then the term has order at least $O(Q)$. Therefore, only the terms with $k=1$ and $d=0$ satisfy the conditions of this case and do not have order~$O(Q)$. The corresponding contribution to the sum \eqref{rho-component} with $a=1$ becomes
\[
\sum_{l\geq 1}
\big\langle \psi^{l-1}, e\big\rangle_{0,2,\beta} (t) e_1 q^\beta Q^{1+l-n}
(-\partial_\lambda)^l
\left(
\frac{\lambda^{(n-1)/2+m}}{
\Gamma((n+1)/2+m)}
\right).
\]
Note that the above sum coincides with the sum in \eqref{Q:exp_phi1_l3}.

{\em Case $3$:} if $\beta\neq 0$ and condition (ii) holds. Then condition (iii) does not hold. For the correlators in the sum over $b$ we get $0\geq (d+1)(n-1) -l+1 =d(n-1) +n-l$. This inequality implies that the sum over $b$ has order $O(Q^n)$. Similarly, for the correlators in the sum over~$k$, we~get
\[
k-1\geq (d+1)(n-1)-l+1 = d(n-1)+n-l\quad
\Rightarrow \quad
k+l-d(n-1) -n \geq 1.
\]
In other words, the sum over $k$ has order at least $O(Q)$.

This completes the analysis of the contributions from the sum \eqref{rho-component} with $a=1$. It remains to analyze the contributions from the sum \eqref{e-component}. This is done in a similar way. To begin with, note that the sum of the terms with $l=0$ is equal to
\[
\sum_{i=1}^n Q^{\Delta+i}(-(n-1)e)^i
(-\partial_\lambda)^l
\frac{\partial_m^i}{i!}
\left(
\frac{\lambda^{(n-1)/2+m-i}}{
\Gamma((n+1)/2+m-i)}
\right).
\]
Note that only the term with $i=n$ depends on $Q$, that is, it has order $O(Q^n)$. Therefore, up to terms of order $O(Q)$ the above sum coincides with the sum of the terms in \eqref{Q:exp_phi1_l2} with $l=0$ and~$i\geq 1$. Suppose that $l>0$. By definition, $Q^{\Delta+l+i}S_l(t,q,Q)e^i$ is equal to
\begin{gather*}%\label{e_desc}
\sum_{\beta+d\ell }
\sum_{b=1}^N
\big\langle e^i\psi^{l-1},\phi_b\big\rangle_{0,2,\beta+d\ell} (t) \phi^b Q^{l+i-d(n-1)} \\
\qquad{} +
\sum_{k=1}^{n-1}
\big\langle e^i \psi^{l-1}, e^k\big\rangle_{0,2,\beta+d\ell} (t) e_k q^\beta Q^{k+l+i-d(n-1)-n}.
\end{gather*}
Let us consider the following 3 cases for the correlators in the above sum.

{\em Case $1$:} if $\beta=0$. Again since $t|_E=0$, we may assume that $t=0$. Let us consider first the correlators in the sum over $b$. Since $\phi_b|_{E}=0$ for $b>1$, the only non-trivial contribution will come from the term with $b=1$. The dimension constraint now yields
$i+l-1 = -1 + n + d(n-1)$ and hence $l+i-d(n-1)=n$. Therefore, the contribution to \eqref{e-component} has order $O(Q^n)$. Let us consider now the correlators in the sum over $k$. The dimension constraint takes the form~$i+l-1+k=-1+n+d(n-1)$ and hence $k+l+i-d(n-1)-n=0$. Therefore, the contribution of these terms to the sum~\eqref{e-component}~is
\[
\sum_{l=1}^\infty
\sum_{i,k=1}^{n-1}
\sum_{d=0}^\infty
\big\langle (-(n-1)e)^i \psi^{l-1}, e^k
\big\rangle_{0,2,d\ell} e_k
(-\partial_\lambda)^l
\frac{\partial_m^i}{i!}\left(
\frac{\lambda^{(n-1)/2+m-i}}{
\Gamma((n+1)/2+m-i)}
\right).
\]
The above sum coincides with the sum of the terms in \eqref{Q:exp_phi1_l2} with $l\geq 1$ and $i\geq 1$.

Note that at this point all terms in the formula that we would like to prove are already matched with contributions from \eqref{rho-component} and \eqref{e-component}. It remains only to check that in the remaining two cases the contributions have order $O(Q)$.

{\em Case $2$:} if $\beta\neq 0$ and condition (ii) does not hold. For the sum over $b$, since $d\leq 0$ and~$l\geq 1$, the powers of $Q$ are positive. For the sum over $k$, in addition to $d\leq 0$, we also have ${i-1+k-1=0}$, that is, $i=k=1$. If $d\leq -1$, then the power of $Q$ is positive. Suppose that $d=0$. Using the divisor equation we get \smash{$\big\langle e\psi^{l-1}, e\big\rangle_{0,2,\beta}(t) = \big\langle e^2\psi^{l-2}\big\rangle_{0,1,\beta}(t)$}. The latter satisfies both conditions (i) and~(ii) of Gathmann's vanishing theorem. In order for the correlator to be non-zero, condition~(iii) must fail $1\geq (d+1)(n-1)-l+2=d(n-1)+n-l+1$. The power of $Q$ becomes
\[
k+l+i-d(n-1)-n = 2+l-d(n-1)-n \geq 2.
\]
{\em Case $3$:} if $\beta\neq 0$ and condition (ii) holds. Then condition (iii) must fail. For the correlators in the sum over $b$, we get
$i-1\geq (d+1)(n-1)-l+1 =d(n-1) + n-l$.
Therefore, the power of $Q$ is~$l+i-d(n-1)\geq n+1$. For the correlators in the sum over $k$, we get
$i-1+k-1\geq d(n-1) + n-l$, that is, the power of $Q$ is $k+l+i-d(n-1)-n\geq 2$.
\end{proof}

\subsection{Quantum cohomology of the blowup}
Let us recall the result of Bayer \cite{Ba}. Suppose that $t\in
\widetilde{H}^*(X)\subset H^*(\operatorname{Bl}(X))$, that is,
$t_1=t_{N+1}=\cdots=t_{N+n-1}=0$. Let us denote by
$\widetilde{\Omega}_i(t,q,Q)$ ($1\leq i\leq N+n-1$) the linear operator in~$H^*(\operatorname{Bl}(X))$ defined by quantum multiplication
$\phi_i\bullet_{t,q,Q}$ for $1\leq i\leq N$ and by quantum
multiplication by~$e^k\bullet_{t,q,Q}$ for $i=N+k$, $1\leq k\leq n-1$.
Slightly
abusing the notation let us denote by the same letters
$\widetilde{\Omega}_i$ the matrices of the corresponding linear
operators with respect to the basis $\phi_i$ ($1\leq
i\leq N+n-1$), where recall that $\phi_{N+k}:=e^k$ ($1\leq k\leq n-1$). Note that the matrix of
$\Delta$ is diagonal with diagonal entries $0,\dots, 0, -1, -2,\dots,
-n+1 $ ($0$ appears $N$ times). The main observation of Bayer (see \cite[Section 3.4]{Ba}) can be stated as follows.
\begin{Proposition}\label{prop:Ba}
The matrices of the linear operators $\widetilde{\Omega}_i$ $(1\leq
i\leq N+n-1)$ with respect to the basis $Q^{-\Delta} \phi_i$ $(1\leq
i\leq N+n-1)$ have the
following Laurent series expansions at $Q=0$:
\begin{equation}\label{qm:phii}
Q^{\Delta} \widetilde{\Omega}_i(t,q,Q) Q^{-\Delta} =
\begin{bmatrix}
\Omega_i(t,q) +O\big(Q^{n-1}\big) & O(Q^n) \\
O(1) & \delta_{i,1} \operatorname{Id}_{n-1} + O(Q)
\end{bmatrix}
,\qquad 1\leq i\leq N,
\end{equation}
where $\operatorname{Id}_{n-1}$ is the identity matrix of size
$(n-1)\times (n-1)$ and
\begin{equation}\label{qm:e^k}
Q^\Delta \widetilde{\Omega}_{N+a}(t,q,Q) Q^{-\Delta} =
Q^{-a}
\begin{bmatrix}
O(Q^n)& O(Q^n) \\
O(1) & \epsilon^a + O\big(Q^2\big)
\end{bmatrix}
,\qquad 1\leq a\leq n-1,
\end{equation}
where $\Omega_i(t,q)$ is the matrix of the linear operator in $H^*(X)$
defined by quantum multiplication by $\phi_i\bullet_{t,q}$ with
respect to the basis $\phi_i$ $(1\leq i\leq N)$ and $\epsilon$ is the
following $(n-1)\times (n-1)$-matrix:
\[
\epsilon =
\begin{bmatrix}
0 & 0 & \cdots & 0 & (-1)^n \\
1 & 0 & \cdots & & 0 \\
 \vdots & \vdots & \ddots & & \vdots \\
0 & 0 & \cdots & 0 & 0 \\
0 & 0 & \cdots & 1 & 0
\end{bmatrix}.
\]
\end{Proposition}
\begin{proof}
The proof is based on Gathmann's vanishing theorem and it is very
similar to the proof of Lemma \ref{le:calibr_Q-homog}.
Since the proofs of \eqref{qm:phii} and \eqref{qm:e^k} are similar, let
us prove only \eqref{qm:e^k}. We~have
\begin{align*}
Q^\Delta \big(e^a\bullet \big(Q^{-\Delta} e^k\big) \big)={}&
\sum_{\widetilde{\beta}=\beta+d\ell}
\left(
\sum_{j=1}^N
\big\langle e^a, e^k, \phi^j\big\rangle_{0,3,\beta+ d\ell}(t)
\phi_j \right.\\
&\left.+
\sum_{l=1}^{n-1} \big\langle e^a, e^k, e_l\big\rangle_{0,3,\beta+d\ell} (t)
e^l Q^{-l}
\right) q^\beta Q^{k-d(n-1)}.
\end{align*}
Let us examine the correlators in the sum over $j$, that is,
\[
\big\langle e^a, e^k, \phi^j\big\rangle_{0,3,\beta+ d\ell}(t) Q^{k -d(n-1)}.
\]
There are 3 cases.

{\em Case $1$:} if $\beta=0$. The correlator is a twisted GW invariant
of the exceptional divisor $E$. The restriction $\phi^j|_E$ is
non-zero only if $\phi^j=1$, that is, $j=N$. Recalling the string equation, we get that
the correlator is non-zero only if $\phi^j=1$ and $d=0$. Therefore,
the contribution takes the form
\[
\int_{\operatorname{Bl}(X)} e^a\cup e^k \cup 1\ \phi_N Q^{k} =
(-1)^{n-1} Q^{n-a} \delta_{a+k,n}\ \phi_N.
\]

{\em Case $2$:} if $\beta\neq 0$ and condition (ii) in Gathmann's
vanishing theorem does not hold. Here we have in mind the correlator
\smash{$\big\langle e^a, e^k, \phi^j\big\rangle_{0,3,\beta+ d\ell}(t)$}. Note that the
weight of this correlator is~${a-1+k-1}$. If condition (ii) does not
hold, then $a-1+k-1\leq 0$ and $d\leq 0$. Since $a,k\geq 1$, this case
is possible only if $a=k=1$. Moreover, if $d=0$, then the correlator
vanishes by the divisor equation. Therefore, we may assume that $d\leq
-1$. The power of $Q$ becomes $k-d(n-1)\geq 1 + n-1=n$, that is, the
contribution in this case has order $O(Q^n)=O\big(Q^{n+1-a}\big)$.

{\em Case $3$:} if $\beta\neq 0$ and condition (ii) holds. According to
Gathmann's vanishing theorem, condition (iii) does not hold, that is,
\[
a-1+k-1\geq (d+1)(n-1)=d(n-1)+n-1\quad \Rightarrow\quad
k-d(n-1)\geq n+1-a.
\]
We get that the contribution in this case has order $O\big(Q^{n+1-a}\big)$.

Combining the results of the 3 cases, we get that the sum over $j$ has
the form
\[
Q^{-a} \big(
(-1)^{n-1} \delta_{a+k,n} Q^n\phi_N + O\big(Q^{n+1}\big)
\big).
\]
Let us examine the correlators in the sum over $l$. Just like above,
there are 3 cases.

{\em Case $1$:} if $\beta=0$. The correlator \smash{$\big\langle e^a, e^k,
e_l\big\rangle_{0,3,d\ell}(t)$} can be computed explicitly. Indeed, such a~correlator is a twisted GW invariant of the exceptional divisor, so it
is independent of~$t\in H^*(X)$, that is, we may substitute
$t=0$. Moreover, since $d\ell$ must be an effective curve class in
$E$, we have $d\geq 0$. Recall that $e_l= (-1)^{n-1} e^{n-l}$ and note that the
dimension of the virtual fundamental cycle of
$\overline{\mathcal{M}}_{0,3}(\operatorname{Bl}(X), d\ell) $ is $d
(n-1)+n$. Therefore, $a+k-l = d(n-1) $. We conclude that $d=0$ or
$d=1$, that is,
\[
\big\langle e^a, e^k, e_l \big\rangle_{0,3,d\ell}(t)=
\begin{cases}
(-1)^n & \mbox{if $d=1$ and $l=a+k-n+1$},\\
1 & \mbox{if $d=0$ and $l=a+k$},\\
0 & \mbox{otherwise.}
\end{cases}
\]
The contribution to the sum over $l$ becomes
\begin{equation}\label{(2,2)-contribution}
\begin{cases}
(-1)^n e^{a+k-n+1} Q^{-a} & \mbox{if $a+k>n-1$}, \\
e^{a+k} Q^{-a}& \mbox{if $a+k\leq n-1$}.
\end{cases}
\end{equation}
Note that the matrix $\epsilon^a$ has entries
\[
\epsilon^a_{ij}=
\begin{cases}
(-1)^n & \mbox{if $j=i+ n-1-a$}, \\
1 & \mbox{if $i=j+a$},\\
0 & \mbox{otherwise}.
\end{cases}
\]
Comparing with formula \eqref{(2,2)-contribution}, we get that the
contribution in this case to formula \eqref{qm:e^k} coincides with the
matrix $Q^{-a} \epsilon^a$.

{\em Case $2$:} if $\beta\neq 0$ and condition (ii) does not hold.
Then $d\leq 0$ and $a-1+k-1+n-l-1\leq 0$. Since~${a,k,n-l\geq 1}$, this
case is possible only if $a=k=1$ and $l=n-1$. Since $\beta\neq 0$, the
 divisor equation implies that $d\neq 0$, that is, $d\leq -1$. In
 other words, if condition (ii) in Gathmann's
 vanishing theorem does not hold, then the power of $Q$, must be
 $k-l-d(n-1)\geq 1-(n-1) + n-1 =1=-a+2$.

{\em Case $3$:} if $\beta\neq 0$ and condition (ii) holds.
Then condition (iii) must fail, that is, $a-1+k-1+n-l-1\geq
(d+1)(n-1)=d(n-1) +n-1$, or equivalently
 $k-l-d(n-1)\geq -a+2$.

Combining the results of these 3 cases, we get that the contribution
of the sum over $l$ matches the (2,2)-block of the matrix on the right-hand side
in formula \eqref{qm:e^k} with the factor $Q^{-a}$ inserted.

In order to complete the argument, we have to repeat the
above discussion by replacing $e^k$ with $\phi_i$ ($1\leq i\leq
N$), that is, we have to determine the contribution to the right-hand side of~\eqref{qm:e^k} of the following expression:
\begin{gather*}
Q^\Delta \big(e^a\bullet \big(Q^{-\Delta} \phi_i\big) \big)\\
\qquad=
\sum_{\widetilde{\beta}=\beta+d\ell}
\left(
\sum_{j=1}^N
\big\langle e^a, \phi_i, \phi^j\big\rangle_{0,3,\beta+ d\ell}(t)
\phi_j +
\sum_{l=1}^{n-1} \big\langle e^a, \phi_i, e_l\big\rangle_{0,3,\beta+d\ell} (t)
e^l Q^{-l}
\right) q^\beta Q^{-d(n-1)}.
\end{gather*}
First, let us determine the contribution of the correlators in the sum
over $j$.

{\em Case $1$:} if $\beta=0$. The correlator could be non-zero only if
$\phi^j=1$ and $d=0$. In the latter case, since $\int_{\operatorname{Bl}(X)}
e^a\cup \phi_i\cup 1 =0$, we get that the correlator still
vanishes. There is no contribution in this case.

{\em Case $2$:} if $\beta\neq 0$ and condition (ii) does not hold. Then
$d\leq 0$ and the weight $a-1\leq 0$, that is, $a=1$. Due to divisor
equation, $d\neq 0$, so $d\leq -1$ and $-d(n-1)\geq n-1=n-a$. We
get that the contribution in this case has order $O\big(Q^{n-1}\big)$.

{\em Case $3$:} if $\beta\neq 0$ and condition (ii) does hold. Then
condition (iii) does not hold, so
${a-1\geq (d+1)(n-1)}$ and $-d(n-1)\geq n-a$. We get that the
contribution in this case is still of order $O(Q^{n-a})$.

Combining the results of the 3 cases, we get that the order of the
elements in the (1,1)-block of the matrix
$Q^\Delta \widetilde{\Omega}_{N+a} Q^{-\Delta}$ is $O(Q^{n-a})$, that
is, the same as in formula \eqref{qm:e^k}.

Finally, it remains to determine the contribution of the correlators
in the sum over $l$.

{\em Case $1$:} if $\beta=0$. The correlator could be non-zero only if
$\phi_i=1$, that is, $i=1$ and $d=0$. We get that the contribution in
this case is $\delta_{i,1} e^a Q^{-a}$.

{\em Case $2$:} if $\beta\neq 0$ and condition (ii) does not hold. Then
$d\leq 0$ and the weight $a-1+n-l-1\leq 0$, that is, $a=n-l=1$. Due to divisor
equation, $d\neq 0$, so $d\leq -1$ and $-l-d(n-1)\geq
-l+n-1=0=-a+1$. The contribution in this case has order $O\big(Q^{-a+1}\big)$.

{\em Case $3$:} if $\beta\neq 0$ and condition (ii) does hold. Then
condition (iii) must fails. We get
\[
a-1+n-l-1\geq (d+1)(n-1)\quad \Rightarrow\quad
-l-d(n-1) \geq -a+1.
\]
The contribution in this case also has order $O\big(Q^{-a+1}\big)$.

Combining the results in the 3 cases we get that the elements in the
(2,1)-block of the matrix~${Q^\Delta \widetilde{\Omega}_{N+a}
Q^{-\Delta}}$ have the form $Q^{-a} (E_{a,1} +O(Q))$, where $E_{a,1}$
denotes the matrix whose~$(a,1)$-entry is $1$ and the remaining
entries are $0$.
\end{proof}

\section{The exceptional component of a reflection vector}
\label{sec:qexp_mon}

Suppose that $\alpha \in  H^* (\operatorname{Bl}(X))$ is a reflection vector. Let us decompose $\alpha=\alpha_e+\alpha_b$, where ${\alpha_e \in  \widetilde{H}^* (E)}$ and $\alpha_b\in H^*(X)$. We will refer to $\alpha_e$ and $\alpha_b$ as respectively the {\em exceptional} and the {\em base} components of $\alpha$. Using Proposition \ref{prop:Q-exp_e}, we would like to classify the exceptional components of the reflection vectors.

\subsection{Dependence on the Novikov variables}
\label{sec:rv_qdep}

Since the quantum cohomology is a Frobenius manifold depending on the parameters $q:=(q_1,\dots,q_r)$ and $q_{r+1}:=Q^{n-1}$, the reflection vectors depend on $q_i$ too. We claim that if $\alpha$ is a~reflection vector, then
\begin{gather}\label{refl:q_dep}
\alpha=q_1^{-p_1}\cdots q_r^{-p_r} q_{r+1}^{- e}\beta,
\end{gather}
where $\beta\in H^*(\operatorname{Bl}(X))$ is independent of
$q_i$ ($1\leq i\leq r+1$). To proof this fact, we will make use of the divisor equation. Suppose that the basis of divisor classes is part of the basis $\{\phi_i\}_{1\leq i\leq N+n-1}$, such that, $p_i=\phi_{i+1}$ for $1\leq i\leq r$ and $p_{r+1}= e = \phi_{N+1}$. Let $\tau_i$ ($1\leq i\leq r+1$) be the linear coordinates corresponding to the divisor classes $p_i$, that is, $\tau_i:=t_{i+1}$ for $1\leq i\leq r$ and $\tau_{r+1}=t_{N+1}$. Using the divisor equation we get that the calibration satisfies the following differential equations:
\begin{align*}
&z\frac{\partial}{\partial \tau_i}S(t,q,Q,z)=
p_i\bullet S(t,q,Q,z),\\
&zq_i\frac{\partial}{\partial q_i}S(t,q,Q,z)=
z\frac{\partial}{\partial \tau_i}S(t,q,Q,z)-S(t,q,Q,z)p_i\cup.
\end{align*}
Therefore,
\[
S(t,q,Q,z)=T(t,q,Q,z)e^{\sum^{r+1}_{i=1}\tau_i p_i\cup /z},
\]
 where
for fixed $z$ the operator series $T(t,q,Q,z)$ is a function on the variables
\begin{equation}\label{divisor_variables}
t_1, q_1e^{t_2},\dots, q_r e^{t_{r+1}}, t_{r+2}, \dots, t_N,
q_{r+1} e^{t_{N+1}}, t_{N+2},\dots, t_{N+n-1}.
\end{equation}
As we already pointed out before (see Section~\ref{sec:qcoh}), due to the divisor equation, the operators of quantum multiplication $\phi_i\bullet_{t,q,Q}$ are represented by matrices whose entries are functions in the variables~\eqref{sec:qcoh} too. Since the canonical coordinates $u_i(t,q,Q)$ are eigenvalues of $E\bullet_{t,q,Q}$, it follows that they have the same property. Moreover, using the chain rule, we get that the partial derivatives \smash{$\frac{\partial u_j}{\partial t_a}$} are also functions in~\eqref{divisor_variables}. On the other hand, if $\alpha$ is a reflection vector, then the Laurent series expansion of \smash{$I^{(-m)}_\alpha(t,q,Q,\lambda)$} at a point $\lambda=u_i(t,q,Q)$ has coefficients that are rational functions in the canonical coordinates $u_j(t,q,Q)$ and their partial derivatives
\smash{$\frac{\partial u_j}{\partial t_a}$} (see Section~\ref{sec:ai_sol}). Therefore,
\begin{equation}\label{refl_period:de}
\left(
\frac{\partial}{\partial \tau_i} - q_i\frac{\partial}{\partial q_i}
\right) I^{(-m)}_\alpha(t,q,Q,\lambda) =0.
\end{equation}
By definition,
\begin{align*}
I^{(-m)}(t,q,Q,\lambda)&=
S\big(t,q,Q,-\partial_\lambda^{-1}\big) \widetilde I^{-m}(\lambda)=T\big(t,q,Q,-\partial_\lambda^{-1}\big)
e^{-\sum^{r+1}_{i=1}\tau_i p_i\cup\partial_\lambda} \widetilde I^{(-m)}(\lambda)\\
&=T\big(t,q,Q,-\partial_\lambda^{-1}\big)
\widetilde I^{-m}(\lambda) e^{-\sum^{r+1}_{i=1}\tau_i p_i},
\end{align*}
where for the last equality we used the following relation (see also the proof of Lemma \ref{le:cp-homog}\,(a)):
\[
-p\cup \partial_\lambda
\frac{\lambda^{\theta+\alpha-1}}{\Gamma(\theta+\alpha)} =
\frac{\lambda^{\theta+\alpha-1}}{\Gamma(\theta+\alpha)} (-p).
\]
Since $I_\alpha^{(-m)}(t,q,Q,\lambda) =I^{(-m)}(t,q,Q,\lambda) \alpha$, from equation \eqref{refl_period:de} we get
\[
q_i\frac{\partial \alpha}{\partial q_i} + p_i\cup \alpha =0,\qquad
\forall 1\leq i\leq r+1.
\]
Our claim that the reflection vector has the form \eqref{refl:q_dep} follows.

\subsection{Canonical coordinates}\label{sec:can_coords}

We would like to determine the dependence of the canonical coordinates $u_i(t,q,Q)$ $(1\leq i\leq N+n-1$) on $Q$, where the parameter $t\in \widetilde{H}(X)$, that is, $t_1=t_{N+1}=\cdots= t_{N+n-1}=0$. Using the identity $u_i=\widetilde{E}(u_i)$, we get
\[%\label{homog_ui}
u_i(t,q,Q)=
\sum_{a=2}^N (1-\op{deg}\phi_a)t_a
\frac{\partial u_i}{\partial t_a}(t,q,Q)+
\sum_{j=1}^r \rho_j
\frac{\partial u_i}{\partial \tau_j}(t,q,Q)-
(n-1)\frac{\partial u_i}{\partial t_{N+1}}(t,q,Q),
\]
where $\rho_j$ are the coefficients in the decomposition $c_1(TX)=\sum_{j=1}^N \rho_j p_j$ and $\tau_j=t_{j+1}$.
The above formula allows us to reduce the problem to investigating the dependence on $Q$ of the partial derivatives \smash{$\tfrac{\partial u_i}{\partial t_j}$} ($1\leq i\leq N+n-1$, $1\leq j\leq N+1$). The advantage now is that the eigenvalues of the operator \smash{$\widetilde{\Omega}_j(t,q,Q) =\phi_j\bullet_{t,q,Q}$} of quantum multiplication by $\phi_j$ are precisely~\smash{$\tfrac{\partial u_i}{\partial t_j}$} ($1\leq i\leq N+n-1$).
\begin{Lemma}\label{le:eig_Q}
Suppose that $U(Q)$ is a square matrix of size $k\times k$ whose entries are functions holomorphic at $Q=0$.
\begin{itemize}\itemsep=0pt
\item[$(a)$] There exists an integer $b>0$, such that, every eigenvalue of
$U(Q)$ has an expansion of the form $\lambda_0+\sum_{i=1}^\infty \lambda_i Q^{i/b}$.
\item[$(b)$] If $\lambda_0$ is an eigenvalue of $U(0)$ of multiplicity $1$, then $U(Q)$ has a unique eigenvalue of multiplicity one of the form $\lambda_0+\sum_{i=1}^\infty \lambda_i Q^i$.
 \end{itemize}
\end{Lemma}
\begin{proof}
The eigenvalues are roots of the characteristic polynomial $\operatorname{det}(\lambda-U(Q))$. This is a~monic polynomial in $\lambda$ of degree $k$ with coefficients in $C\{Q\}$ --- the ring of convergent power series in~$Q$. Therefore, in order to prove~(a), it is sufficient to prove the following statement. Let~${f(Q,\lambda)\in C\{Q\}[\lambda]}$ be a monic polynomial. Then the roots of $f(Q,\lambda)$ have the expansion stated in the lemma. Let us decompose
\[
f(0,\lambda)=(\lambda -w_1)^{b_1}\cdots (\lambda -w_s)^{b_s},
\]
where $w_i\neq w_j$ for $i\neq j$. Recalling Hensel's lemma (see \cite[Chapter 2, Section 2]{GR}), we get that~${f(Q,\lambda) = f_1(Q,\lambda)\cdots f_r(Q,\lambda)}$, where $f_i(Q,\lambda)\in C\{Q\}[\lambda]$ is a monic polynomials of degree~$b_i$, such that, $f_i(0,\lambda)=(\lambda -w_i)^{b_i}$. Note that if $b_i=1$ for some $i$, then the unique zero of $f_i(Q,\lambda)=0$ is a holomorphic at $Q=0$ and its value at $Q=0$ is $w_i$. Therefore, part~(b) is an elementary consequence of Hensel's lemma. If $s>1$, then the lemma follows from the inductive assumption. Suppose that $s=1$, that is,
\[
f(Q,\lambda)=\lambda^k+a_1(Q)\lambda^{k-1}+\cdots + a_k(Q),
\]
where $a_i(0)=0$. We may assume that the sub-leading coefficient $a_1(Q)=0$. Indeed, using the substitution $\lambda\mapsto \lambda - a_1(Q)/k$ we can transform the polynomial to one for which the sub-leading coefficient is~$0$. The roots of the two polynomials are related by a shift of $a_1(Q)/k$, so it is sufficient to prove our claim for one of them. Let $\operatorname{ord}(a_i)$ be the order of vanishing of $a_i$ at~${Q=0}$. If $a_i(Q)=0$, then we define the order of vanishing to be $+\infty$. Put
\smash{$\nu:=\operatorname{min}_{1\leq i\leq k}
\frac{\operatorname{ord}(a_i)}{i}$}.
Substituting $\lambda=Q^\nu \mu$ in the equation $f(Q,\lambda)=0$ and dividing by $Q^{\nu k}$, we get
\[
\mu^k+\sum_{i=2}^k a_i(Q)Q^{-\nu i} \mu^{k-i}=0.
\]
Since $\operatorname{ord}(a_i) \geq \nu i$ with equality for at leats one~$i$, we get that the left-hand side of the above equation is a monic polynomial $g\big(Q^{1/b},\mu\big)$ in $\CC\big\{Q^{1/b}\big\} [\mu]$ for some integer $b>0$. Note that~$g(0,\mu)$ has at least two different zeroes because its sub-leading coefficient is $0$. Therefore, just like above, we can use Hensel's lemma to reduce the proof to a~case in which the inductive assumption can be applied. This completes the proof.
\end{proof}

\begin{Proposition}\label{prop:can-Q_exp}
Let $u_j(t,q,Q)$ $(1\leq j\leq N+n-1)$ be the canonical coordinates of the quantum cohomology of $\operatorname{Bl}(X)$, where the parameter $t\in \widetilde{H}(X)$. After renumbering, the canonical coordinates split into two groups
\[
u_j(t,q,Q)\in \CC\{Q\},\qquad 1\leq j\leq N,
\]
and
\[
u_j(t,q,Q)=-(n-1)v_k Q^{-1}+O(1),\qquad j=N+k,\qquad 1\leq k\leq n-1,
\]
where $v_k$ $(1\leq k\leq n-1)$ are the solutions of the equation $\lambda^{n-1}=(-1)^n$.
\end{Proposition}
\begin{proof}
Let us apply the above lemma to the matrix of the linear operator
\begin{equation}\label{aux_euler}
\sum_{a=2}^N (1-\operatorname{deg}\phi_a)
t_a \widetilde{\Omega}_a(t,q,Q) +
\sum_{j=1}^r \rho_j \widetilde{\Omega}_{j+1}(t,q,Q)+
Q\widetilde{\Omega}_{N+1}(t,q,Q)
\end{equation}
with respect to the basis $Q^{-\Delta }\phi_i$ ($1\leq i\leq N+n-1$). Recalling Proposition \ref{prop:Ba}, we get that the entries of the matrix of the operator \eqref{aux_euler} are holomorphic at $Q=0$ and that its specialization to $Q=0$ has the form
\[
\begin{bmatrix}
E\bullet_{t,q} & 0 \\
* & \epsilon
\end{bmatrix}.
\]
The eigenvalues of the above matrix are the canonical coordinates $u_i^X(t,q)$ $(1\leq i\leq N)$ of the quantum cohomology of $X$ and the solutions $v_k$ ($1\leq k\leq n-1$) of the equation $\lambda^{n-1}=(-1)^n$. Note that for a generic choice of $t$ the eigenvalues are pairwise distinct. On the other hand, the canonical vector fields \smash{$\frac{\partial}{\partial u_j}$} ($1\leq j\leq N+n-1$) form an eigenbasis for the operator \eqref{aux_euler}. Let us enumerate the canonical coordinates in such a way that the eigenvalues corresponding to \smash{$\frac{\partial}{\partial u_j}$} for~$1\leq j\leq N$ and $j=N+k$ with $1\leq k\leq n-1$ are respectively $u_j^X(t,q)+O(Q)$ and $v_k+O(Q)$. Recall that the eigenvalues of the operators \smash{$\widetilde{\Omega}_a(t,q,Q)$} are \smash{$\tfrac{\partial u_j}{\partial t_a}(t,q,Q)$} ($1\leq j\leq N+n-1$). Recalling Lemma \ref{le:eig_Q}\,(b), we get that the functions
\[
E(u_j) + Q\frac{\partial u_j}{\partial t_{N+1}}, \qquad 1\leq j\leq N+n-1
\]
are holomorphic at $Q=0$, where
$E:=
\sum_{a=2}^N (1-\operatorname{deg}(\phi_a)) t_a \partial/\partial {t_a} +
\sum_{j=1}^r \rho_j \partial/\partial t_{j+1}$. Moreover, the restriction to $Q=0$ satisfies
\[
\left.
\left(
E(u_j) + Q\frac{\partial u_j}{\partial t_{N+1}}
\right)
\right|_{Q=0} =
\begin{cases}
u_j^X(t,q) & \mbox{if } 1\leq j\leq N,\\
v_k & \mbox{if } j=N+k.
\end{cases}
\]
On the other hand, note that $E(u_j)$ are the eigenvalues of the matrix
\[
\sum_{a=2}^N (1-\operatorname{deg}\phi_a)
t_a \widetilde{\Omega}_a(t,q,Q) +
\sum_{j=1}^r \rho_j \widetilde{\Omega}_{j+1}(t,q,Q)
\]
and that the restriction of the above matrix at $Q=0$ is
\[
\begin{bmatrix}
E\bullet_{t,q} & 0 \\
* & 0
\end{bmatrix}.
\]
Recalling Lemma \ref{le:eig_Q}, we get that $N$ of the eigenvalues $E(u_j)$ ($1\leq j\leq N+n-1$) are holomorphic at $Q=0$ and have the form $u_i^X(t,q)+O(Q)$ ($1\leq i\leq N$), while the remaining $n-1$ ones have order~$O(Q^\alpha)$ for some rational number $\alpha>0$. Similarly, by applying Lemma \ref{le:eig_Q} to the matrix~\smash{$Q\widetilde{\Omega}_{N+1}$}, we get that its eigenvalues \smash{$Q\frac{\partial u_j}{\partial t_{N+1}}$} split into two groups. The first group consist of $n-1$ functions holomorphic at $Q=0$ with an expansion of the form $v_k + O(Q)$, while the second group consist of $N$ functions that have an expansion in possibly fractional powers of~$Q$ of order $O\big(Q^\beta\big)$ for some $\beta>0$. Let $(t,q)$ be generic, such that, the canonical coordinates~${u_i^X(t,q)}$~($1\leq i\leq N$) are pairwise distinct and non-zero. Then for every $1\leq j\leq N+n-1$, the sum \smash{$E(u_j)+ Q\frac{\partial u_j}{\partial t_{N+1}}\neq 0$}. Therefore, the two numbers $E(u_j)$ and \smash{$Q\frac{\partial u_j}{\partial t_{N+1}}$} can not be vanishing at $Q=0$, that is, either $E(u_j)$ is holomorphic at $Q=0$ of the form $u_i^X(t,q)+O(Q)$ or \smash{$Q\frac{\partial u_j}{\partial t_{N+1}}$} is holomorphic at $Q=0$ of the form $v_k+O(Q)$. In the first case, since $E(u_j)$ is holomorphic at $Q=0$ and the sum \smash{$E(u_j)+ Q\frac{\partial u_j}{\partial t_{N+1}}$} is also holomorphic at $Q=0$, we get that~\smash{$Q\frac{\partial u_j}{\partial t_{N+1}}$} is holomorphic at $Q=0$. Similarly, the holomorphicity of \smash{$Q\frac{\partial u_j}{\partial t_{N+1}}$} implies that~$E(u_j)$ is holomorphic. Therefore, $E(u_j)$ and \smash{$Q\frac{\partial u_j}{\partial t_{N+1}}$} are holomorphic at $Q=0$ for all $j$. In particular, the numbers $\alpha$ and $\beta$ must be integral. Note that since $E(u_j)$ and \smash{$Q\frac{\partial u_j}{\partial t_{N+1}}$} can not vanish simultaneously at $Q=0$, we get that for every $1\leq j\leq N+n-1$ either
\begin{align*}
E(u_j) & = u_i^X(t,q) + O(Q), \qquad
Q\frac{\partial u_j}{\partial t_{N+1}} = O(Q)
\end{align*}
for some $i$ or
\begin{align*}
E(u_j) & = O(Q),\qquad
Q\frac{\partial u_j}{\partial t_{N+1}}= v_k+ O(Q)
\end{align*}
for some $k$. In the first case, we will get that
\[
u_j(t,q,Q)= E(u_j)-(n-1) \frac{\partial u_j}{\partial t_{N+1}} \in \CC\{Q\},
\]
while in the second case
\[
u_j(t,q,Q)= E(u_j)-(n-1) \frac{\partial u_j}{\partial t_{N+1}}= -(n-1) v_k Q^{-1} + O(1).\tag*{\qed}
\] \renewcommand{\qed}{}
\end{proof}

\subsection[Twisted periods of P\^\{n-1\}]{Twisted periods of $\boldsymbol{\PP^{n-1}}$}
\label{sec:tw_periods}
Let us recall the reduced cohomology $\widetilde{H}(E)$ of the
exceptional divisor. It has a basis given by $e^i$ ($1\leq i\leq
n-1$). The Poincar\'e pairing on $H(\operatorname{Bl}(X))$ induces a non-degenerate pairing on $\widetilde{H}(E)$:
\[
\big(e^i,e^j\big) = (-1)^{n-1} \delta_{i+j,n},\qquad 1\leq i,j\leq n-1.
\]
The {\em twisted periods} will be multi-valued analytic functions with values in $\widetilde{H}(E)$.
Let us define the following linear operators on $\widetilde{H}(E)$:
\begin{align*}
\leftexp{tw}{\theta} \big(e^i\big) & := \left(\frac{n}{2} -i\right) e^i,\qquad
\leftexp{tw}{\rho} \big(e^i\big) :=
 \begin{cases}
 -(n-1) e^{i+1} & \mbox{if } 1\leq i<n-1, \\
 0 & \mbox{if } i=n-1.
 \end{cases}
\end{align*}
Let us define first the calibrated twisted periods:
\[
\leftexp{ tw}{\widetilde{I}}^{(-m)}_\beta(\lambda) =
e^{\leftexp{ tw}{\rho} \partial_\lambda \partial_m}
\left(
 \frac{
 \lambda^{ \leftexp{tw}{\theta}+m-1/2}}{
 \Gamma\big(\leftexp{tw}{\theta}+m+1/2\big)}
 \right)
 \beta,\qquad \beta\in \widetilde{H}(E),
\]
 and the twisted calibration
 \[
 \leftexp{tw}{S}(Q,z) =
\sum_{k=0}^\infty
\leftexp{tw}{S}_k(Q) z^{-k}
 \in \operatorname{End}\big(\widetilde{H}(E)\big)\big[\big[z^{-1}\big]\big],
 \]
 where $\leftexp{ tw}{S}_0(Q)=1$ and
 \[
 \big(\leftexp{tw}{S}_k(Q) e^i,e^j\big) =
 \sum_{d=0}^\infty
 \big\langle e^i\psi^{k-1},e^j\big\rangle_{0,2,d\ell} Q^{-d(n-1)},\qquad
 1\leq i,j\leq n-1.
 \]
Note that in the above sum only one value of $d$ contributes, because the degree of the cohomology class in the correlator, that is, $k-1+i+j$ must be equal to the dimension of the virtual fundamental cycle of $\overline{\mathcal{M}}_{0,2}(\operatorname{Bl}(X),d\ell) $ which is $(n-1)(d+1)$ and we get $d(n-1) = k+i+j-n$.
 The twisted periods are defined by
 \[
 \leftexp{tw}{I}^{(-m)}_\beta(Q,\lambda) :=
 \sum_{l=0}^\infty
\leftexp{tw}{S}_l(Q) (-\partial_\lambda)^l\
 \leftexp{tw}{\widetilde{I}}^{(-m)}_\beta(\lambda),\qquad
 \beta\in \widetilde{H}(E).
 \]
The twisted periods satisfy a system of ODEs with respect to $Q$ and
$\lambda$. Let us derive these differential equations.
\begin{Lemma}\label{le:de-tw-calibr-period}
We have
\[
\big(\lambda-\leftexp{tw}{\rho}\big)\partial_\lambda
\leftexp{tw}{\widetilde{I}_\beta^{(-m)}}(\lambda)=
\left(
\leftexp{tw}{\theta}+m-\frac{1}{2}\right)
\leftexp{tw}{\widetilde{I}_\beta^{(-m)}}(\lambda).
\]
\end{Lemma}
The proof is straightforward and it is left as an exercise.
\begin{Lemma}\label{le:homog-tw-caibr}
We have
\[
Q\partial_Q \leftexp{tw}{S}_l +
\leftexp{tw}{\theta} \leftexp{tw}{S}_l -
\leftexp{tw}{S}_l \leftexp{tw}{\theta} = -l \leftexp{tw}{S}_l.
\]
\end{Lemma}
\begin{proof}
Let us apply the operator on the left-hand side to $e^i$ and compute the pairing
with $e^j$ for an arbitrary $1\leq i,j\leq n-1$. We get
\[
Q\partial_Q \big(\leftexp{tw}{S}_l e^i,e^j\big)+
\big(\leftexp{tw}{\theta} \leftexp{tw}{S}_l e^i,e^j\big)-
\big(\leftexp{tw}{S}_l \leftexp{tw}{\theta} e^i, e^j\big).
\]
Since $\leftexp{tw}{\theta} e^i= \left(\frac{n}{2}-i\right) e^i$ and
$\leftexp{tw}{\theta} $ is skew-symmetric with respect to the pairing
$(\ ,\ )$, the above expression becomes
\[
Q\partial_Q \big(\leftexp{tw}{S}_l e^i,e^j\big) -(n-i-j) \big(\leftexp{tw}{S}_l e^i,e^j\big).
\]
We saw above that the expression $\big(\leftexp{tw}{S}_l e^i,e^j\big)$ is
proportional to $Q^{-d(n-1)}$ where $d(n-1)=l+i+j-n$. Therefore, the
above expression becomes $-l \big(\leftexp{tw}{S}_l e^i,e^j\big)$.
\end{proof}

In order to state the next result we need to introduce the linear
operator
\[
e\bullet_{tw}\colon\ \widetilde{H}(E)\to \widetilde{H}(E),\qquad
e^i\mapsto e\bullet_{tw} e^i,
\]
where the quantum product is defined by
\[
\big(e\bullet_{tw} e^i,e^j\big) =
\sum_{d=0}^\infty
\big\langle e, e^i, e^j\big\rangle_{0,3,d\ell}
Q^{-d(n-1)}.
\]
For dimensional reasons, that is, $1+i+j= n+d(n-1)$, we get that the
contributions to the quantum product could be non-trivial only in
degree $d=0$ and $d=1$. Recalling our computations from Section~\ref{sec:proj-tw-GW}, we get the following formulas:
\[
e\bullet_{tw} e^i =
\begin{cases}
e^{i+1} & \mbox{if } 1\leq i\leq n-2, \\
(-1)^n Q^{-(n-1)} e & \mbox{if } i=n-1.
\end{cases}
\]
In other words, the matrix of $e\bullet_{tw}$ with respect to the
basis $e,e^2,\dots,e^{n-1}$ is
\[
e\bullet_{tw} =
\begin{bmatrix}
0 & 0 & \cdots & 0 & (-1)^n Q^{-(n-1)} \\
1 & 0 & \cdots & 0 & 0 \\
\vdots & \vdots &\vdots & & \vdots \\
0 & 0 & \cdots & 1 & 0
\end{bmatrix}.
\]
\begin{Lemma}\label{le:tw-period-div_eq}
We have
\[
Q\partial_Q
\leftexp{tw}{S}_l = (n-1) e\bullet_{tw}
\leftexp{tw}{S}_{l-1}+
\leftexp{tw}{S}_{l-1} \
\leftexp{tw}{\rho}, \qquad \forall l\geq 1.
\]
\end{Lemma}
\begin{proof}
The lemma is an easy consequence of the divisor equation and the
topological recursion relations for the GW invariants of the blowup $\operatorname{Bl}(X)$. We have, by the divisor equation,
\[
\big\langle e, e^i \psi^l, e^j\big\rangle_{0,3,d\ell} =
-d \big\langle e^i \psi^l, e^j\big\rangle_{0,2,d\ell} +
\big\langle e\cup e^{i} \psi^{l-1}, e^j\big\rangle_{0,3,d\ell}.
\]
The left-hand side, according to the topological recursion relations is equal to
\[
\sum_{d'+d''=d}
\sum_{k=1}^{n-1}
\big\langle e^i \psi^{l-1}, e_k\big\rangle_{0,2,d'\ell}
\big\langle e^k, e,e^j\big\rangle_{0,3,d''\ell}.
\]
Multiplying the above identity by $(n-1) Q^{-d(n-1)}$ and summing over
all $d\geq 0$, we get
\[
\big(S_l e^i, (n-1)e\bullet_{tw} e^j\big) = Q\partial_Q \big(S_{l+1} e^i,e^j\big)+
\big(S_l (n-1)e\cup e^i, e^j\big).
\]
Note that the above expression is 0 for $i=n-1$ because
$e^n=(-1)^{n-1}\phi_N$ is a cohomology class on~$\operatorname{Bl}(X)$ whose
restriction to the exceptional divisor $E$ is $0$. Therefore, we may
replace~${(n-1)e\cup e^i}$ with $-\leftexp{tw}{\rho}(e^i)$. The lemma
follows. \end{proof}

Using Lemmas \ref{le:de-tw-calibr-period}, \ref{le:homog-tw-caibr} and~\ref{le:tw-period-div_eq}, we get that the twisted periods satisfy
the following system of differential equations
\begin{align}
\label{tw-de_1}
&(\lambda+(n-1) e\bullet_{tw}) \partial_\lambda \,
\leftexp{tw}{I}^{(-m)}_\alpha(Q,\lambda) =
\left(
\leftexp{tw}{\theta}+m-\frac{1}{2}\right)\,
\leftexp{tw}{I}^{(-m)}_\alpha(Q,\lambda), \\
\label{tw-de_2}
&Q\partial_Q \,
\leftexp{tw}{I}^{(-m)}_\alpha(Q,\lambda) =
-(n-1)e\bullet_{tw}\partial_\lambda\,
\leftexp{tw}{I}^{(-m)}_\alpha(Q,\lambda),
\end{align}
where $\alpha= Q^{\leftexp{tw}{\rho}} \beta$ with $\beta\in
\widetilde{H}(E)$ independent of $Q$ and $\lambda$. Note that the
determinant
\[
\operatorname{det}(\lambda+(n-1) e\bullet_{tw}) =
\lambda^{n-1} + \big( (n-1)Q^{-1}\big) ^{n-1}.
\]
We get that the twisted periods are multivalued analytic functions on
the complement of the hypersurface in $\CC^*\times\CC$ defined by the
equation $(Q\lambda)^{n-1}+(n-1)^{n-1}=0$.

\subsection[Periods of P\^\{n-2\}]{Periods of $\boldsymbol{\PP^{n-2}}$}
\label{sec:proj_periods}

We would like to compute the monodromy of the system of differential
equations \eqref{tw-de_1}--\eqref{tw-de_2}. We will do this by
identifying the twisted periods with the periods of $\PP^{n-2}$. To
begin with, let us recall the definition of the periods of
$\PP^{n-2}$. We have $H^*\big(\PP^{n-2}\big) = \CC[p]/p^{n-1}$, where
$p=c_1(\O(1))$ is the hyperplane class. We have an isomorphism of
vector spaces
\[
\widetilde{H}(E)\cong H\big(\PP^{n-2}\big),\qquad e^i \mapsto p^{i-1}.
\]
Note that under this isomorphism $\leftexp{tw}{\theta}$ coincides with
the grading operator $\theta_{\PP^{n-2}}$ and $\leftexp{tw}{\rho}$
coincides with $-c_1\big(T\PP^{n-2}\big)\cup$. Therefore, the calibrated
periods in the twisted GW theory of $\PP^{n-1}$ and the GW theory of
$\PP^{n-2}$ are related by \smash{$e^{\leftexp{tw}{\theta} \pi\ii }
\leftexp{tw}{ \widetilde{I}_\beta^{(-m)} } (\lambda) =
\widetilde{I}^{(-m)}_{\sigma(\beta)}(\lambda)$},
where $\sigma(\beta) := e^{\pi\ii \theta } \beta$, where~$\theta$ is
the grading operator of $\PP^{n-2}$.

Let us compare the $S$-matrices. In the GW theory of $\PP^{n-2}$, we
have
\[
S(q,z)^{-1} 1 = 1 + \sum_{d=1}^\infty
\frac{q^d}{\prod_{m=1}^d (p- m z)^{n-1}},
\]
where $q$ is the Novikov variable corresponding to $\O(1)$.
Using the divisor equation $(-z q\partial_q + p\cup) S(q,z)^{-1} =
S(q,z)^{-1} p\bullet$, where $p\bullet$ is the operator of quantum
multiplication by $p$, we get
\begin{equation}\label{S-proj}
S(q,z)^{-1} p^i = p^i + \sum_{d=1}^\infty
\frac{q^d (p-d z)^i }{\prod_{m=1}^d (p-m z)^{n-1}},
\qquad 0\leq i\leq n-2.
\end{equation}
On the other hand, the twisted $S$-matrix $\leftexp{tw}{S}(Q,z)$ can
be computed from the $S$-matrix of the blowup
$\operatorname{Bl}(\PP^n)$ of $\PP^n$ at one point
which is known explicitly. Namely, let us recall that~$\operatorname{Bl}(\PP^n)$ is the submanifold of $\PP^{n-1}\times
\PP^n$ defined by the quadratic equations $x_i y_j=x_j y_i$
($0\leq i,j\leq n-1$), where~${x=[x_0,\dots, x_{n-1}]}$ and
$y=[y_0,\dots,y_{n}]$ are the homogeneous coordinate systems on
respectively $\PP^{n-1}$ and $\PP^{n}$. We have two projection maps
$\pi_1\colon \operatorname{Bl}(\PP^n)\to \PP^{n-1}$ and
${\pi_2\colon \operatorname{Bl}(\PP^n)\to \PP^{n}}$. Note that $\pi_2$ is the
projection of the blowup --- the exceptional divisor $E$ is the fiber
over $[0,0,\dots,0,1]\in \PP^n$. Let $L_1$ and $L_2$ be the pullbacks
of the hyperplane bundles $\O(1)$ on respectively $\PP^{n-1}$ and
$\PP^{n}$. Let us denote by $\leftexp{bl}{S}(q_1,q_2,z)$ the S-matrix
in the GW theory of $\operatorname{Bl}(\PP^n)$, where $q_1$ and $q_2$
are the Novikov variables corresponding to the line bundles $L_1$ and
$L_2$. Then we have
\[
\leftexp{bl}{S}(q_1,q_2,z)^{-1} 1 =
\sum_{d_1,d_2\geq 0}
\frac{q_1^{d_1} q_2^{d_2} \prod_{m=-\infty}^0 (p_2-p_1 -mz) }{
\prod_{m=1}^{d_1} (p_1-mz)^n
\prod_{m=1}^{d_2} (p_2-mz)
\prod_{m=-\infty}^{d_2-d_1} (p_2-p_1 -mz)},
\]
where $q_1$ and $q_2$ are the Novikov variables. The degree class in
$\operatorname{Bl}(\PP^n) $ corresponding to a~pair~${(d_1,d_2)}$ is
$d_1 e_1+ d_2 e_2$, where $e_1$ is the class of a line in $E$ and
$e_2=\pi_2^{-1}($line in $\PP^n$ avoiding $[0,0,\dots,0,1])$. It can
be checked that the cohomology ring of the blowup is
\[
H(\operatorname{Bl}(\PP^n))=\CC[p_1,p_2]/
\langle p_2(p_2-p_1)=0, p_1^n=0\rangle
\]
and that $\O(E)=L_2L_1^{-1}$, that is, the Poincar\'e dual of the
exceptional divisor $E$ is $e=p_2-p_1$. In order to compute the
twisted S-matrix $\leftexp{tw}{S}$, we have to restrict
$\leftexp{bl}{S}$ to $q_2=0$ and substitute $q_1=Q^{-(n-1)}$. We get
\begin{equation}\label{blS}
\leftexp{bl}{S}\big(Q^{-(n-1)},0,z\big)^{-1} 1 = 1+
\sum_{d=1}^\infty
\frac{Q^{-d(n-1)} \prod_{m=-d+1}^0 (p_2-p_1 -mz) }{
\prod_{m=1}^{d} (p_1-mz)^n }.
\end{equation}
Note that the numerator is proportional to $p_2-p_1$. Using the
relation $p_2(p_2-p_1)=0$, we get that $p_1-mz$ can be replaced by
$p_1-p_2-mz= -e-mz$. The above formula takes the form
\begin{equation}\label{blS2}
\leftexp{bl}{S}\big(Q^{-(n-1)},0,z\big)^{-1} 1 = 1+
\sum_{d=1}^\infty
\frac{(-1)^{dn} Q^{-d(n-1)} e }{
(e+dz)^n\prod_{m=1}^{d-1} (e+mz)^{n-1} }.
\end{equation}
Using the above formula and the divisor equation
\[
\left( -\frac{1}{n-1} zQ\partial_Q+ e\cup\right)
\leftexp{bl}{S}\big(Q^{-(n-1)},0,z\big)^{-1} =
\leftexp{bl}{S}\big(Q^{-(n-1)},0,z\big)^{-1} e\bullet,
\]
whose proof is the same as the proof of Lemma~\ref{le:tw-period-div_eq}, we get
\begin{equation}\label{S-tw_proj}
\leftexp{tw}{S}(Q,z)^{-1} e^i =
e^i +
\sum_{d=1}^\infty
\frac{(-1)^{dn} Q^{-d(n-1)} e }{
(e+dz)^{n-i}\prod_{m=1}^{d-1} (e+mz)^{n-1} },\qquad 1\leq i\leq n-1,
\end{equation}
where the right-hand side should be expanded into a power series in $z^{-1}$ and
$e$ should be identified with the linear operator
\[
e\cup_{tw}\colon\ \widetilde{H}(E)\to \widetilde{H}(E),\qquad
e\cup_{tw} e^i :=
\begin{cases}
e^{i+1} & \mbox{if } 1\leq i\leq n-2,\\
0 & \mbox{if } i=n-1.
\end{cases}
\]
Comparing formulas \eqref{S-proj} and \eqref{S-tw_proj}, we get that
if we put $q=(-1)^n Q^{-(n-1)}$, then the matrices of $S(q,z)$ and
$\leftexp{tw}{S}(Q,-z)$ with respect to respectively the bases
$1,p,\dots,p^{n-2}$ and $e,e^2,\dots,e^{n-1}$ coincide. Now we are in
position to prove the following key formula.
\begin{Proposition}\label{prop:tw=proj}
Under the isomorphism
$\widetilde{H}(E)\cong H\big(\PP^{n-2}\big)$ the following identity holds:
\[%\label{tw_period=proj_period}
\leftexp{tw}{I}^{(-m)}_\beta(Q,\lambda) =
e^{-\pi\ii\theta}\,
I^{(-m)}_{\sigma(\beta)}\bigl(-Q^{-(n-1)},\lambda\bigr),
\]
where $\sigma=e^{\pi\ii\theta}$ and $\theta$ is the grading operator
of $\PP^{n-2}$.
\end{Proposition}
\begin{proof}
By definition,
\[
\leftexp{tw}{I}^{(-m)}_\beta(Q,\lambda) =
\sum_{l\in \ZZ} \sum_{i=1}^{n-1}
\operatorname{Res} dz z^{l-1} (-\partial_\lambda)^l
\big(
\leftexp{tw}{\widetilde{I}}^{(-m)}_\beta(\lambda),
\leftexp{tw}{S}(Q,-z)^{-1} e^i
\big) e_i,
\]
where the residue is defined formally as the coefficient in front of
$dz/z$. Under the isomorphism $\widetilde{H}(E)\cong H(\PP^{n-2})$ the
period
\begin{gather*}
\leftexp{tw}{\widetilde{I}}^{(-m)}_\beta(\lambda) =
e^{-\pi \ii \theta} \widetilde{I}^{(-m)}_{\sigma(\beta)}(\lambda),
\qquad
\leftexp{tw}{S}(Q,-z)^{-1} e^i = S\big((-1)^n Q^{-(n-1)}, z\big)^{-1} p^{i-1},
\end{gather*}
and $e_i= (-1)^{n-1} p^{n-1-i}$. Note that the Poincar\'e pairing on
$H\big(\PP^{n-2}\big)$ differs from the pairing on~\smash{$\widetilde{H}(E)$} by the
sign $(-1)^{n-1}$. The above formula for the period takes the form
\begin{align*}
\leftexp{tw}{I}^{(-m)}_\beta(Q,\lambda) ={}&
\sum_{l\in \ZZ} \sum_{i=1}^{n-1}
\operatorname{Res} dz z^{l-1} (-\partial_\lambda)^l\\
&\times
\big(
 e^{-\pi \ii \theta} \widetilde{I}^{(-m)}_{\sigma(\beta)}(\lambda),
S\big((-1)^n Q^{-(n-1)}, z\big)^{-1} p^{i-1}
\big) p^{n-1-i}.
\end{align*}
Since $e^{\pi \ii \theta} p =-p e^{\pi\ii\theta}$, using formula
\eqref{S-proj}, we get
\[
e^{\pi\ii\theta} S(q,z)^{-1} p^{i-1} =
e^{\pi\ii (\frac{n}{2}-i )}
S\big((-1)^{n-1}q,-z\big)^{-1} p^{i-1}.
\]
The formula for the period takes the form
\begin{align*}
\leftexp{tw}{I}^{(-m)}_\beta(Q,\lambda) ={}&
\sum_{l\in \ZZ} \sum_{i=1}^{n-1}
\operatorname{Res} dz z^{l-1} (-\partial_\lambda)^l\\&\times
\big(
\widetilde{I}^{(-m)}_{\sigma(\beta)}(\lambda),
S\big(-Q^{-(n-1)}, -z\big)^{-1} p^{i-1}
\big) \sigma^{-1}\big(p^{n-1-i}\big),
\end{align*}
where we used that
$
\sigma^{-1}\big(p^{n-1-i}\big)=
e^{\pi\ii (\frac{n}{2}-i )} p^{n-1-i}.
$
Clearly, the right-hand side of the above identity coincides with
\smash{$\sigma^{-1}\big(I^{(-m)}_{\sigma(\beta)} \bigl(-Q^{-(n-1)},\lambda\bigr)\big)$}. The
lemma follows.
\end{proof}

\subsection[Monodromy of the twisted periods of P\^\{n-1\}]{Monodromy of the twisted periods of $\boldsymbol{\PP^{n-1}}$}

Let us describe the monodromy group of the system of differential
equations \eqref{tw-de_1}--\eqref{tw-de_2}, that is, the monodromy of
the twisted periods of $\PP^{n-1}$. According to Proposition
\ref{prop:tw=proj}, it is sufficient to recall the monodromy group for
the periods of $\PP^{n-2}$. Let us first fix $q=1$ and
$\lambda^\circ\in \RR_{>0}$ sufficiently large --- any
$\lambda^\circ>n-1$ works. The value of the period
$I^{(-m)}(q,\lambda)$ depends on the choice of a path from
$(1,\lambda^\circ)$ to $(q,\lambda)$ avoiding the discriminant
\[
\{(q,\lambda)\mid  \operatorname{det}(\lambda-(n-1)p\bullet) = 0\}.
\]
For fixed $q$, the equation of the discriminant has $n-1$ solutions
\[
u_k(q):=(n-1)\eta^{-2k} q^{1/(n-1)}, \qquad 0\leq k\leq n-2,
\]
where $\eta=e^{\pi \ii/(n-1)}$. Let us focus first on the monodromy of
the twisted periods for $q=1$. The fundamental group
\[
\pi_1(\CC\setminus{\{u_0(1),\dots,u_{n-2}(1)\}},\lambda^\circ)
\]
is a free group generated by the simple loops $\gamma_k^\circ$
corresponding to the paths $C_k^\circ$ from $\lambda^\circ$ to
$u_k(1)$ defined as follows. $C_k^\circ$ consists of two pieces. First,
an arc on the
circle with center $0$ and radius $\lambda^\circ$ starting at
$\lambda^\circ$ and rotating clockwise on angle $2\pi\ii k/(n-1)$. The
second piece is the straight line segment from $\lambda^\circ
\eta^{-2k}$ to $u_k(1)= (n-1) \eta^{-2k}$. It turns out that the
reflection vector corresponding to the simple loop $\gamma_k^\circ$ is
precisely $\Psi(\O(k))$, where $\Psi$ is the Iritani's map for the
integral structure of the quantum cohomology of $\PP^{n-2}$ (see
formula \eqref{psi}), that is,
\[
\Psi(\O(k)) = (2\pi)^{\frac{3-n}{2}} \Gamma(1+p)^{n-1} e^{2\pi\ii k p}.
\]
If $q\in \CC^*:=\CC\setminus{\{0\}}$ is arbitrary, then we construct
the path $\big(q,\lambda^\circ q^{1/(n-1)}\big)$ by letting $q$ vary
continuously along some reference path. This path allows us to
determine the value of \smash{$I^{(-m)}_\alpha(q,\lambda)$} at
$\lambda=\lambda^\circ q^{1/(n-1)}$ which we declare to be the base
point of $\CC\setminus{\{u_0(q),\dots,u_{n-2}(q)\}}$. Let
$\gamma_k(q)$ be the simple loop obtained from $\gamma_k^\circ$ by
rescaling $\lambda\in \gamma_k^\circ\mapsto \lambda q^{1/(n-1)}$. The
reflection vectors corresponding to $\gamma_k(q)$ are precisely
\[
\Psi_q(\O(k)) =
(2\pi)^{\frac{3-n}{2}} \Gamma(1+p)^{n-1} q^{-p} e^{2\pi \ii k p}.
\]
The proof of the above facts in the case of $\PP^2$ (that is $n=4$)
can be found in \cite{Milanov:p2}. In general, the argument is
straightforward to generalize. Now let us apply the above construction
and Proposition \ref{prop:tw=proj} in order to describe the monodromy
of the twisted periods of $\PP^{n-1}$. We have~${q=-Q^{-(n-1)}}$. Let us
assume that $Q\in \RR_{>0}$ is a real number. We pick a reference path
from~$1$ to $q$ consisting of the interval \smash{$\big[1,Q^{-(n-1)}\big]$} and the arc in
the upper half-plane from~${Q^{-(n-1)}}$ to $q = - Q^{-(n-1)}.$
Note that with such a choice of the reference path $q^{1/(n-1)}=\eta Q^{-1}$. Therefore, $\gamma_k(q)$ becomes a simple loop around $u_k(q)=(n-1)
\eta^{-2k+1} Q^{-1}$ which are precisely the singularities of the
differential equation \eqref{tw-de_1}. We get the following corollary.
\begin{Corollary}\label{cor:tw_refl}
If $\beta\in \widetilde{H}(E)$ is such that the analytic continuation
of \smash{$\leftexp{tw}{I}^{(-m)}_\beta(Q,\lambda)$} along $\gamma_k(q)$ is
\smash{$\leftexp{tw}{I}^{(-m)}_{-\beta}(Q,\lambda)$}, then $\beta$ must be
proportional to
\[
\Psi(\O_E(-k+1)) = (2\pi)^{\frac{1-n}{2}}
\Gamma(\operatorname{Bl}(X)) Q^{-e(n-1)}
(2\pi\ii)^{\operatorname{deg}}
\operatorname{ch}(\O_E(-k+1)),
\]
where $\O_E(-k+1):=\O(-(k-1) E)-\O(-kE)$.
\end{Corollary}

\begin{proof}
According to the above discussion and Proposition \ref{prop:tw=proj},
under the isomorphism $\widetilde{H}(E)\cong H\big(\PP^{n-2}\big)$,
$\sigma(\beta)$ must be proportional to $\Psi_q(\O(k))$, that is,
$\beta$ is proportional to
\[
e^{-\pi\ii \theta} \
(2\pi)^{\frac{3-n}{2}} \Gamma(1+p)^{n-1} q^{-p} e^{2\pi k p}=
(2\pi)^{(3-n)/2} \ii^{2-n}
\Gamma(1-p)^{n-1} q^p e^{-2\pi\ii k p}.
\]
Note that $q^p = e^{\pi\ii p} Q^{-p(n-1)}$. Therefore, under the
isomorphism $\widetilde{H}(E)\cong H\big(\PP^{n-2}\big)$, the above
expression becomes
\[
(2\pi)^{(3-n)/2} \ii^{2-n}
\Gamma(1-e)^{n-1} Q^{-e(n-1)} e^{(-2k+1)\pi\ii e} e.
\]
We have to check that the above expression is proportional to the
image of the Iritani map for~$\operatorname{Bl}(X)$ of the exceptional object
$\O((-k+1)E)-\O(-kE).$ We have
\begin{gather*}
\Gamma(\operatorname{Bl}(X))= \Gamma(X) \Gamma(1-e)^n \Gamma(1+e),\\
\Gamma(1-e)\Gamma(1+e) = \frac{2\pi\ii e}{e^{\pi\ii e}-e^{-\pi\ii
 e}} =
\frac{2\pi\ii e}{e^{2\pi\ii e}-1} e^{\pi\ii e},
\end{gather*}
and
\[
(2\pi\ii)^{\operatorname{deg}}
\operatorname{ch}(\O((-k+1)E)-\O(-kE))=
e^{-2\pi\ii k e }\big(e^{2\pi\ii e}-1\big).
\]
Since $\Gamma(X)\cup e= e$, the image of the Iritani map becomes
\[
(2\pi)^{(3-n)/2} \ii \Gamma(1-e)^{n-1} Q^{-e(n-1)}
e^{(-2k+1)\pi\ii e} e.
\]
The claim of the lemma follows.
\end{proof}

\subsection{Isomonodromic analytic continuation}\label{sec:iso_ac}

Let $D(u,r)$ be the open disk in $\CC$ with center $u$ and radius $r$. Put $\mathbb{D}_r:= D(0,r)$. Let $\epsilon>0$ be a~real number, $V\subset \CC$ an open subset, and $u_i\colon\mathbb{D}_\epsilon \to V$ ($1\leq i\leq m$) be $m$ holomorphic functions, such that, there exists a positive real number $\delta>0$ satisfying
\begin{enumerate}\itemsep=0pt
\item[(i)]
The $m$ disks $D(u_i(0),\delta)$ ($1\leq i\leq m$) are pairwise disjoint and contained in $V$.
\item[(ii)]
We have $u_i(Q)\in D(u_i(0),\delta)$ for all $Q\in \mathbb{D}_\epsilon$.
\end{enumerate}
Suppose that $I$ is a multi-valued analytic function on $\mathbb{D}_\epsilon\times V\setminus{\Sigma}$ with values in a finite dimensional vector space $H$, where
\[
\Sigma:=\{
(Q,\lambda)\in \mathbb{D}_\epsilon\times V\mid
\lambda=u_i(Q)\ \mbox{for some $i$}
\}.
\]
Let us fix $\lambda^\circ\in V$, such that, $\mathbb{D}_\epsilon\times \{\lambda^\circ\}$ is disjoint from $\Sigma$. Then $I$ is analytic at $(Q,\lambda)=(0,\lambda^\circ)$ and $I$ extends analytically along any path in
$\mathbb{D}_\epsilon\times V\setminus{\Sigma}$ starting at $(0,\lambda^\circ)$. In particular, we can extend uniquely $I(Q,\lambda)$ for all $Q\in \mathbb{D}_\epsilon$ and $\lambda$ sufficiently close to $\lambda^\circ$.
Let us expand~${I(Q,\lambda)=\sum_{d=0}^\infty I_d(\lambda) Q^d}$, where each coefficient $I_d$ is an $H$-valued analytic function at $\lambda=\lambda^\circ$. Clearly, $I_d(\lambda)$ extends analytically along any path in $V\setminus{\{u_1(0),\dots,u_m(0)\}}$.
\begin{Lemma}\label{le:iso_ac}
Suppose that $\gamma$ is a closed loop based at $\lambda^\circ$ in
\[
V\setminus{D(u_1(0),\delta)\sqcup\cdots \sqcup D(u_m(0),\delta)},
\]
such that, for every fixed $Q\neq 0$, the analytic extension of $I(Q,\lambda)$ along the path $\{Q\}\times \gamma$ transforms $I(Q,\lambda)$ into $A(I(Q,\lambda))$, where $\lambda$ is sufficiently close to $\lambda^\circ$ and $A\in \operatorname{GL}(H)$ is a~linear operator. If the operator $A$ is independent of $Q$, then the analytic continuation along $\gamma$ transforms the coefficient $I_d(\lambda)$ into $A(I_d(\lambda))$.
\end{Lemma}
\begin{proof}
Since \smash{$I_d(\lambda)=\frac{1}{d!}\frac{\partial^d I}{\partial Q^d}(0,\lambda)$} by replacing the function $I(Q,\lambda)$ with its partial derivative \smash{$\frac{1}{d!}\frac{\partial^d I}{\partial Q^d}(Q,\lambda)$}, we can reduce the general case to the case when $d=0$.

Let us cover the path $\gamma$ with small {\em closed} disks $D_j$ ($1\leq j\leq N$), such that,
\begin{enumerate}\itemsep=0pt
\item[(i)]
$D_j$ is disjoint from $D(u_i(0),\delta)$ for all $i$.
\item[(ii)]
$D_j\cap D_{j+1}\neq \varnothing$.
\item[(iii)]
$D_N=D_1$.
\end{enumerate}
In other words, the union of the disks $D_j$ give a fattening of the
path $\gamma$. Let $I(Q,\lambda_j)$ $\forall \lambda_j\in D_j$ be the
analytic extension of $I(Q,\lambda)$ along $\gamma$. Let us fix an
arbitrary $\epsilon'>0$. There exists a small~${\rho_j>0}$, such that,
$I(Q,\lambda_j)$ is a uniformly continuous function in $(Q,\lambda)\in
\mathbb{D}_{\rho_j}\times D_j$. Therefore, there exists $0<\delta_j'<\rho_j$, such that,
\[
\|I(Q,\lambda_j)-I(0,\lambda_j)\|<\epsilon',\qquad \forall
(Q,\lambda_j)\in \mathbb{D}_{\delta_j'}\times D_j,
\]
where $\|\ \|$ is any norm on $H$ --- for example fix an
isomorphism $H\cong \RR^{\operatorname{dim}(H)}$ and choose the
standard Euclidean metric.
Since there are only finitely many disks $D_j$, we can choose $\rho$
and~$\delta'$ that work for all $j$ simultaneously, that is,
$\rho_j=\rho$ and $\delta_j'=\delta'$. Using the triangle inequality,
we~get
\begin{align*}
\| I(0,\lambda_N)-AI(0,\lambda_1)\| \leq{}&
\| I(0,\lambda_N)-I(Q,\lambda_N)\| +
\| I(Q,\lambda_N)-AI(Q,\lambda_1)\|\\
& +
\|A\| \| I(Q,\lambda_1)-I(0,\lambda_1)\|.
\end{align*}
Note that the middle term on the right-hand side of the inequality is $0$ by definition. Choosing~${|Q|<\delta'}$, we get that the right-hand side of the above inequality is bounded by $\epsilon'(1+\|A\|)$. Since $\epsilon'$ can be chosen arbitrary small, we get $I(0,\lambda_N)=A(I(0,\lambda_1))$ which is exactly what we had to prove.
\end{proof}

We will need a result which is slightly more general then Lemma \ref{le:iso_ac}. Namely, suppose that~$I$ is a multi-valued analytic function on $\mathbb{D}_\epsilon^*\times V\setminus{\Sigma}$, where $\mathbb{D}_\epsilon^*:=\mathbb{D}_\epsilon\setminus{\{0\}}$ is the punctured disk.
\begin{Definition}\label{def:log_sing}
The singularity of $I(Q,\lambda)$ at $Q=0$ is said to be {\em at most logarithmic} if the following two conditions hold:
\begin{enumerate}\itemsep=0pt
\item[(i)] The function has an expansion of the form
\[
I(Q,\lambda) =\sum_{s=0}^n \sum_{d=0}^\infty I_{s,d}(\lambda) Q^d (\log Q)^s,
\]
where $\lambda$ is sufficiently close to $\lambda^\circ$ and $Q\in \mathbb{D}_\epsilon\setminus{(-\epsilon,0]}$.
\item[(ii)] The functions $I_s(Q,\lambda):=\sum_{d=0}^\infty I_{s,d}(\lambda) Q^d$ ($0\leq s\leq n$) extend analytically along any path in~$\mathbb{D}_\epsilon\times V\setminus{\Sigma}$.
\end{enumerate}
\end{Definition}
\begin{Remark}
Condition (ii) might be redundant but we could not prove it in this generality. For our purposes, both conditions are easy to verify, because $I(Q,\lambda)$ will be a solution to an ODE in $Q$ that has a Fuchsian singularity at $Q=0$.
\end{Remark}
\begin{Proposition}\label{prop:iso_ac}
Suppose that $I$ is a multi-valued analytic function on $\mathbb{D}_\epsilon^*\times V\setminus{\Sigma}$ and that it has at most a logarithmic singularity at $Q=0$. Furthermore, suppose that $\gamma$ is a closed loop based at $\lambda^\circ$ in
\[
V\setminus{D(u_1(0),\delta)\sqcup\cdots \sqcup D(u_m(0),\delta)},
\]
such that, for every fixed $Q\in \mathbb{D}_\epsilon\setminus{(-\epsilon,0]}$, the analytic extension of $I(Q,\lambda)$ along the path~${\{Q\}\times \gamma}$ transforms $I(Q,\lambda)$ into $A(I(Q,\lambda))$, where $\lambda$ is sufficiently close to $\lambda^\circ$ and $A\in \operatorname{GL}(H)$ is a~linear operator. If the operator $A$ is independent of $Q$, then the analytic continuation along $\gamma$ transforms the coefficient $I_{s,d}(\lambda)$ into $A(I_{s,d}(\lambda))$.
\end{Proposition}
\begin{proof}
Let $I_s(Q,\lambda)$ be as in condition (ii) in Definition \ref{def:log_sing}. We have $I(Q,\lambda) = \sum_{s=0}^n I_s(Q,\lambda) \allowbreak\times(\log Q)^s$. The analytic continuation along $\gamma$ yields $A(I(Q,\lambda)) = \sum_{s=0}^n \widetilde{I}_s(Q,\lambda) (\log Q)^s$, where $\widetilde{I}_s(Q,\lambda)$ is the analytic extension of $I_s(Q,\lambda)$ along $\gamma$. It is easy to prove by letting $Q\to 0$ that such an identity is possible only if the coefficients in front of the powers of $\log Q$ are equal, that is, $A(I_s(Q,\lambda)) = \widetilde{I}_s(Q,\lambda)$. It remains only to recall Lemma \ref{le:iso_ac}.
\end{proof}

\subsection{Vanishing of the base component}
Let $\alpha=Q^{-(n-1)e}\beta$, where $\beta \in H^*(\operatorname{Bl}(X))$ is a vector independent of $Q$ and $t$. Let $\beta=\beta_e+\beta_b$. We would like to extract the leading order terms in the power series expansion at $Q=0$ of
\begin{equation}\label{blowup_period}
Q^{\Delta+m+(n-1)/2} I^{(-m)}\left(t,q,Q,Q^{-1}\lambda\right)\alpha,
\end{equation}
where $m>0$ is a sufficiently large integer, that is, we choose $m$ so
big that the operator $\widetilde{\theta}+m+1/2$ has only positive
eigenvalues.
Moreover, we would like to determine the structure of the following terms in the expansion up to order $Q^n$. Note that
\[
Q^{\Delta} Q^{-\widetilde\theta-m+\frac{1}{2}}Q^{-\widetilde\rho}\alpha =
Q^{-m-(n-1)/2} \big(
\beta_e + Q^{\operatorname{deg}} (Q^{-\rho} \beta_b)
\big).
\]
Therefore, we have
\begin{align*}
Q^{\Delta+m+(n-1)/2} I^{(-m)}\big(t,q,Q,Q^{-1}\lambda\big)\alpha={}&
\big(Q^\Delta I^{(-m)}\big(t,q,Q,Q^{-1}\lambda\big)Q^{\widetilde\rho}Q^{\widetilde\theta+m-\frac{1}{2}}Q^{-\Delta}\big)\\
&\times
\big(
\beta_e + Q^{\operatorname{deg}} (Q^{-\rho} \beta_b)
\big).
\end{align*}
Let us look at the contribution of $\beta_e$ to \eqref{blowup_period}, that is, the expression
\begin{equation}\label{blowup_period_e}
\big(Q^\Delta I^{(-m)}\big(t,q,Q,Q^{-1}\lambda\big)Q^{\widetilde\rho}Q^{\widetilde\theta+m-\frac{1}{2}}Q^{-\Delta}\big)
\beta_e.
\end{equation}
According to Proposition \ref{prop:Q-exp_e}, the leading order term of the $\widetilde{H}(E)$-component of \eqref{blowup_period_e} is at degree 0 and it is precisely \smash{$\leftexp{tw}{I}_{\beta_e}^{(-m)}(1,\lambda)$}, that is, the twisted period of $\PP^{n-1}$ at $Q=1$. The leading order term of the $H(X)$-component of \eqref{blowup_period_e} is at degree $n$ and the corresponding coefficient in front of $Q^n$ is given by
\begin{equation}\label{base_blowup_period_be:lot}
\sum_{l,d=0}^\infty
(-\partial_\lambda)^l \big\langle
\psi^{l-1} \widetilde{I}^{(-m)}_{\beta_e}(\lambda), 1\big\rangle_{0,2,d\ell} \
(-1)^{n-1}\phi_N,
\end{equation}
where
\[
\widetilde{I}^{(-m)}_{\beta_e}(\lambda)=
e^{-(n-1)e \partial_\lambda\partial_m} \left(
\frac{\lambda^{\widetilde{\theta}+m-1/2} }{
\Gamma(\widetilde{\theta}+m+1/2)}\right)\beta_e
\]
is the calibrated period of $\operatorname{Bl}(X)$ and for $l=0$ the correlator should be understood via the string equation as \smash{$\big\langle\psi^l \widetilde{I}^{(-m)}_{\beta_e}(\lambda), 1,1\big\rangle_{0,3,d\ell}$}.

Let us look at the contribution to \eqref{blowup_period} corresponding to $\beta_b$, that is, the expression
\begin{equation}\label{blowup_period_b}
\big(Q^\Delta I^{(-m)}\big(t,q,Q,Q^{-1}\lambda\big)Q^{\widetilde\rho}Q^{\widetilde\theta+m-\frac{1}{2}}Q^{-\Delta}\big)
 Q^{\operatorname{deg}} (Q^{-\rho} \beta_b).
\end{equation}
Let us decompose $\beta_b=\sum_{a=1}^N \beta_{b,a} \phi_a$. The $H(X)$-component of \eqref{blowup_period_b} is a power series in~$Q$ whose coefficients are polynomials in $\log Q$ whose coefficients are in $H(X)$. According to Propositions~\eqref{prop:Q-exp_phi1} and
\eqref{prop:Q-exp_phia}, the coefficient in front of $Q^M (\log Q)^0$ with $0\leq M\leq n$ has the form
\begin{gather}\label{base_blowup_period:degree_m}
\sum_{a\colon \operatorname{deg}(\phi_a)=M}
\frac{\lambda^{\theta+m-1/2}}{\Gamma(\theta+m+1/2)} \beta_{b,a}\phi_a +
\sum_{a'\colon \operatorname{deg}(\phi_{a'})<M} \beta_{b,a'} f_{M,a'}(\lambda),
\end{gather}
where $f_{M,a'}(\lambda)$ is the $H(X)$-component of the coefficient in front of $Q^{M-\operatorname{deg}(\phi_{a'})}$ in the expansion at $Q=0$ of
\[
\big(Q^\Delta I^{(-m)}\left(t,q,Q,Q^{-1}\lambda\right)Q^{\widetilde\rho}Q^{\widetilde\theta+m-\frac{1}{2}}Q^{-\Delta}\big)
 \phi_{a'}.
\]
Let us summarize our analysis.
\begin{Proposition}\label{prop:blowup_period_Qexp}
Let $\alpha=Q^{-(n-1)e}\beta$, where $\beta\in H(\operatorname{Bl}(X))$ is independent of $t$ and $Q$. Then
\begin{itemize}\itemsep=0pt
\item[$(a)$] The $H(X)$-component of \eqref{blowup_period} expands as a power series in $Q$ whose coefficients are polynomials in $\log Q$. The coefficient in front of $(\log Q)^0 Q^M$ for $0\leq M\leq n-1$ is given by~\eqref{base_blowup_period:degree_m}, while for $M=n$ it is given by the sum of \eqref{base_blowup_period:degree_m} $($with $M=n)$ and \eqref{base_blowup_period_be:lot}.
\item[$(b)$] If $\beta_{b,1}=0$, then the $\widetilde{H}(E)$-component of \eqref{blowup_period} expands as a power series in $Q$. The corresponding leading order term is \smash{$\leftexp{tw}{I}_{\beta_e}^{(-m)}(1,\lambda)$}.
\end{itemize}
\end{Proposition}

Let us discuss now the analytic properties of the series \eqref{base_blowup_period_be:lot}. It is convenient to introduce the following series:
\[
\Phi_\beta(Q,\lambda) :=
\sum_{l,d=0}^\infty
(-\partial_\lambda)^l \big\langle
\psi^{l-1} \widetilde{I}^{(-m)}(\lambda)Q^{-(n-1)e}\beta, 1
\big\rangle_{0,2,d\ell}
 Q^{-d(n-1)},\qquad \beta\in\widetilde{H}(E).
\]
Note that \eqref{base_blowup_period_be:lot} coincides with $\Phi_{\beta_e}(1,\lambda) (-1)^{n-1}\phi_N$. Recalling the definition of the period vector~\smash{$I^{(-m)}_{\alpha}(t,q,Q,\lambda)$}, we get
\[
\Phi_\beta (Q,\lambda)=\big(I^{(-m)}_{Q^{-(n-1)e}\beta}(t,0,Q,\lambda),1\big),
\qquad
\forall \beta \in \widetilde{H}(E).
\]
\begin{Proposition}\label{prop:mon_Phi}
Let $Q$ be a positive real number and $\gamma_k(q)$ with $q=-Q^{-(n-1)}$ be the same simple loop as in Corollary~{\rm\ref{cor:tw_refl}}. If
\[
\beta=\Psi(\O_E(-k+1))=
(2\pi)^{(1-n)/2}
\Gamma(1-e)^{n-1}
Q^{-(n-1) e}
e^{(-2k+1)\pi\ii e}
2\pi \ii e,
\]
then the analytic continuation of $\Phi_\beta(Q,\lambda)$ along $\gamma_k(q)$ is $-\Phi_\beta(Q,\lambda)$.
\end{Proposition}
The proof of Proposition \ref{prop:mon_Phi} will be given in Section
\ref{sec:mirror}.
\subsection{Proof of Theorem \ref{thm}}
Now we are in position to prove the main result of this paper. Let us fix $t\in \widetilde{H}(X)$ and the Novikov variables $q=(q_1,\dots,q_r)$ of $X$ to be generic, such that, the quantum cohomology of $X$ is semisimple and the conclusions of Proposition \ref{prop:can-Q_exp} hold. Let us pick a real number $R>0$, such that, $u_j(t,q,0)\in \mathbb{D}_R$ for all $1\leq j\leq N$, where recall that $\mathbb{D}_R$ denotes the circle with center~$0$ and radius $R$. Let us choose a real number $\epsilon>0$ so small that the quantum cup product of the blowup~$\operatorname{Bl}(X)$ at $(t,q,Q)$ is convergent for all $|Q|<\epsilon$, $u_j(t,q,Q)\in \mathbb{D}_R$ for all $|Q|<\epsilon$ and~$1\leq j\leq N$, and~$R<(n-1)\epsilon^{-1}$. We would like to use the results from Section \eqref{sec:iso_ac} in the following settings: the domain $V:= \{\lambda\mid  |\lambda|>R\epsilon \}$, $m:=n-1$, and the $(n-1)$ holomorphic functions (denoted by~$u_i$ in Section \ref{sec:iso_ac}) will be given by $Qu_{N+k}(t,q,Q)$, $1\leq k\leq n-1$. Here we are using Proposition~\ref{prop:can-Q_exp} to conclude that $Q u_{N+k}(t,q,Q)=-(n-1)v_k + O(Q)$ is analytic at~${Q=0}$. Let us choose $\delta>0$, such that, the disks $D(-(n-1)v_k,2\delta)$ ($1\leq k\leq n-1$) are pairwise disjoint. If necessary, we decrease $\epsilon$ even further so that condition (ii) given in the beginning of Section \ref{sec:iso_ac} is satisfied. Note that condition (i) is satisfied according to our choice of $\delta$. Before we continue further let us fix the solutions $v_k$ of $\lambda^{n-1}=(-1)^n$ to be given by $v_k=-\eta^{-2k+1}$, where~${\eta:=e^{\pi\ii/(n-1)}}$. Then $-(n-1)v_k =(n-1) \eta^{-2k+1}$. Finally, for a reference point $\lambda^\circ\in V$ we pick any positive real~${\lambda^\circ>(n-1)>R\epsilon}$.

Let us define the loop $\gamma_k$ in \smash{$V\setminus{D\big((n-1)\eta^{-1},\delta\big)\sqcup\cdots \sqcup D\big((n-1) \eta^{-2n+3},\delta\big)}$} to be the simple loop around $(n-1)\eta^{-2k+1}$ based at $\lambda^\circ$ corresponding to the path from $\lambda^\circ$ to $(n-1)\eta^{-2k+1}$ consisting of the following two pieces: an arc along the circle $|\lambda|=\lambda^\circ$ obtained by rotating from~$\lambda^\circ$ clock-wise on angle $(2k-1)\pi/(n-1)$ and the second piece is the straight segment from~$\lambda^\circ \eta^{-2k+1}$ to $(n-1)\eta^{-2k+1}$.

Suppose now that $Q\in \mathbb{D}_\epsilon$ is a positive real number. Note that by re-scaling the path~$\gamma_k$, we obtain a path $\gamma_k\cdot Q^{-1}$ which is a simple loop around $(n-1)\eta^{-2k+1}Q^{-1}$. The simple loop $\gamma_k$ goes around $(n-1)\eta^{-2k+1}$ along a circle with center $(n-1)\eta^{-2k+1}$ and radius $r$, where $\delta<r<2\delta$. We claim that by decreasing $\epsilon$ if necessary, we can arrange that the circle with center $(n-1)\eta^{-2k+1} Q^{-1}$ and radius $r Q^{-1}$ contains the canonical coordinate $u_{N+k}(t,q,Q)$. Indeed, we have
\[
\big|u_{N+k}(t,q,Q)-(n-1) \eta^{-2k+1} Q^{-1}\big|=
\big|Qu_{N+k}(t,q,Q)-(n-1) \eta^{-2k+1}\big| Q^{-1}
\]
and since \smash{$\big|Qu_{N+k}(t,q,Q)-(n-1) \eta^{-2k+1}\big|$} has order $O(Q)$, by choosing $\epsilon$ small enough we can arrange that \smash{$\big|Qu_{N+k}(t,q,Q)-(n-1) \eta^{-2k+1}\big|<r$} for all $|Q|<\epsilon$. In other words, the re-scaled loop $\gamma_k \cdot Q^{-1}$ is a simple loop around the canonical coordinate $u_{N+k}(t,q,Q)$. Let us denote by~${\alpha\in H(\operatorname{Bl}(X))}$ the reflection vector corresponding to the simple loop $\gamma_k \cdot Q^{-1}$. Let us recall Proposition \ref{prop:iso_ac} for the series \eqref{blowup_period}, that is,
\[%\label{blowup_period_reflection}
I(Q,\lambda):= Q^{\Delta+m+(n-1)/2} I^{(-m)}\big(t,q,Q,Q^{-1}\lambda\big)\alpha.
\]
The singularities of $I(Q,\lambda)$ are precisely at $Q^{-1}\lambda = u_j(t,q,Q)$ for $1\leq j\leq N+k$, that is, $\lambda=Q u_j(t,q,Q)$. Note that by definition of $R$, the first $N$ singularities $Qu_j(t,q,Q)$ (${1\leq j\leq N}$) are in $\mathbb{D}_{R\epsilon}$. Therefore, $I(Q,\lambda)$ is a multi-valued analytic function in $(Q,\lambda)\in \mathbb{D}_\epsilon \times V\setminus{\Sigma}$. Although we are not going to give a complete proof, let us outline how to prove that $I(Q,\lambda)$ has at most logarithmic singularity at $Q=0$ (see Definition \ref{def:log_sing}). Recall the divisor equation~\eqref{refl_period:de} with~${i=r+1}$. Note that \smash{$q_{r+1}\partial_{q_{r+1}}=\tfrac{1}{n-1} Q\partial_Q$}. Combining the divisor equation and the differential equation of the second structure connection with respect to $\tau_{r+1}=t_{N+1}$, it is easy to prove that for every $\lambda\in V\setminus{D\big((n-1)\eta^{-1},\delta\big)\sqcup\dots\sqcup D\big((n-1)\eta^{-2n+3},\delta\big)}$ the function $I(Q,\lambda)$ is a solution to a differential equation that has a Fuchsian singularity at $Q=0$. Now the conclusion follows from the theory of Fuchsian singularities.

The analytic continuation of $I(Q,\lambda)$ along $\gamma_k$ transforms $I(Q,\lambda)$ into $-I(Q,\lambda)$ because when~$\lambda$ changes along $\gamma_k$, $Q^{-1}\lambda$ changes along $\gamma_k\cdot Q^{-1}$ which is the simple loop used to define the reflection vector $\alpha$. Let us look at the expansion of $I(Q,\lambda)$ at $Q=0$ in the powers of $Q$ and~$\log Q$. To begin with, we know that $\alpha= Q^{-(n-1)e}\beta$ where $\beta\in H(\operatorname{Bl}(X))$ is independent of $t$ and $Q$ (it could depend on~$q$). Let us decompose $\beta=\beta_e+\beta_b$, where $\beta_e\in \widetilde{H}(E)$ and $\beta_b\in H(X)$. Put $\beta_b=:\sum_{i=1}^N \beta_{b,i} \phi_i$. We claim that $\beta_b=0$. According to Proposition \ref{prop:blowup_period_Qexp}\,(a), the coefficient in front of $Q^0(\log Q)^0$ in the expansion of the $H(X)$-component of $I(Q,\lambda)$ is
\[
\frac{\lambda^{\theta+m-1/2}}{\Gamma(\theta+m+1/2)} \beta_{b,1}\phi_1 =
\frac{\lambda^{m+(n-1)/2}}{\Gamma(m+(n+1)/2)} \beta_{b,1}\phi_1.
\]
According to Proposition \ref{prop:iso_ac}, the analytic continuation along $\gamma_k$ of the above expression should change the sign. However, the function $\lambda^{m+(n-1)/2}$ is invariant under the analytic continuation along $\gamma_k$. Therefore, the only possibility is that $\beta_{b,1}=0$. Suppose that $M$ is the smallest number, such that, $\beta_{b,a}\neq 0$ for some $\phi_a$ of degree $M$. If $M\leq n-1$, then since $\beta_{b,a'}=0$ for all $a'$, such that, $\operatorname{deg}(\phi_{a'})<M$, Proposition \ref{prop:blowup_period_Qexp}\,(a) yields that the coefficient in front of $Q^M (\log Q)^0$ in the expansion of the $H(X)$-component of $I(Q,\lambda)$ is
\[
\sum_{a: \operatorname{deg}(\phi_a)=M}
\frac{\lambda^{\theta+m-1/2}}{\Gamma(\theta+m+1/2)} \beta_{b,a}\phi_a.
\]
Just like before, the above expression is invariant under the analytic continuation along $\gamma_k$, while Proposition~\ref{prop:iso_ac} implies that the analytic continuation must change the sign. The conclusion is again that $\beta_{b,a}=0$ for all $a$ for which $\phi_a$ has degree~$M$. We get that all $\beta_{b,a}=0$ except possibly for~$\beta_{b,N}$. Let us postpone the analysis of $\beta_{b,N}$ and consider $\beta_{e}$ first. Recalling Proposition~\ref{prop:blowup_period_Qexp}\,(b), we get that the coefficient in front of $Q^0$ in the expansion of the $\widetilde{H}(E)$-component of $I(Q,\lambda)$ is the twisted period
\smash{$\leftexp{tw}{I}^{(-m)}_{\beta_e}(1,\lambda)$}.
Therefore, the analytic continuation of
\smash{$\leftexp{tw}{I}^{(-m)}_{\beta_e}(1,\lambda)$} along~$\gamma_k$ must be
\smash{$\leftexp{tw}{I}^{(-m)}_{-\beta_e}(1,\lambda)$}. Recalling Corollary \ref{cor:tw_refl}, we get that $\beta_e$ must be proportional to
\[
\Psi(\O_E(-k+1))=
(2\pi)^{\frac{1-n}{2}} \Gamma(\operatorname{Bl}(X)) (2\pi\ii)^{\rm deg}
\operatorname{ch}(\O_E(-k+1)).
\]
Let us prove that $\beta_{b,N}=0$.
According to Proposition \ref{prop:blowup_period_Qexp}, the coefficient in front of $Q^n (\log Q)^0$ in the expansion of the $H(X)$-component of $I(Q,\lambda)$ is
\[
\left( \frac{\lambda^{m-(n+1)/2}}{\Gamma(m+(1-n)/2)} \beta_{b,N} +
\Phi_{\beta_e}(1,\lambda) (-1)^{n-1} \right) \phi_N.
\]
Let us analytically continue the above expression along $\gamma_k$. Just like above, the analytic continuation should change the sign. However, recalling Proposition \ref{prop:mon_Phi}, we get
\[
\left( \frac{\lambda^{m-(n+1)/2}}{\Gamma(m+(1-n)/2)} \beta_{b,N} -
\Phi_{\beta_e}(1,\lambda) (-1)^{n-1} \right) \phi_N.
\]
Therefore, $\beta_{b,N}=0$ and this completes the proof of our claim that $\beta_b=0$. Moreover, we proved that $\beta=\beta_e$ is proportional to $\Psi(\O_E(-k+1))$. In order to conclude that the proportionality coefficient is $\pm 1$, we need only to check that the Euler pairing
${\langle \Psi(\O_E(-k+1)), \Psi(\O_E(-k+1))\rangle =1}$. For simplicity, let us consider only the case when $k=1$. In fact, the general case follows easily by using analytic continuation with respect to $q$ around $q=0$: the clock-wise analytic continuation transforms $\Psi_q(\O_E(-k+1))$ to $\Psi_q(\O_E(-k+1)\otimes \O(-E))=\Psi_q(\O_E(-k))$. We have
\[
\Psi_q(\O_E)=
(2\pi)^{(1-n)/2}
\Gamma(1-e)^{n-1}
q^{e} (2\pi \ii e),
\]
where $q=-Q^{-(n-1)}$ and the branch of $\log q$ is fixed in such a way that $q^{e}= e^{\pi\ii e} Q^{-(n-1)e}$. Recalling formula \eqref{euler-pairing}, after a straightforward computation, we get $\langle\Psi_q(\O_E),\Psi_q(\O_E)\rangle=1$.

\section{Mirror model for the twisted periods}\label{sec:mirror}

The goal in this section is to prove Proposition
\ref{prop:mon_Phi}. The idea is to prove that the Laplace transform of
$\Phi_\beta(Q,\lambda)$ with respect to $\lambda$ can be identified
with an appropriate oscillatory integral whose integration cycle is
swept out by a family of vanishing cycles. Once this is done, the
statement of the proposition follows easily by an elementary local
computation. It is more convenient to construct an oscillatory integral when $q:=-Q^{-(n-1)}$ is a positive real number. Therefore, let us reformulate the statement of Proposition \ref{prop:mon_Phi} by analytically continuing~$\Phi_\beta(Q,\lambda)$ with respect to $Q$ along an arc in the counter-clockwise direction connecting the rays $\RR_{>0}$ and~$\eta \RR_{>0}$, where~$\eta:=e^{\pi\ii/(n-1)}$. Note that the value of $\log Q$ will change to $\log |Q| + \tfrac{\pi\ii}{n-1}$. In other words, we will assume that $Q= \eta q^{-1/(n-1)}$ where $q\in \RR_{>0}$ is a positive real number. Note that $Q^{-(n-1) e}= e^{-\pi\ii e} q^{e}$ and that the formula for $\Phi_\beta(Q,\lambda)$ takes the form
\begin{equation}\label{phi_infty}
\Phi_\beta(q,\lambda)=
\sum_{l,d=0}^\infty
(-\partial_\lambda)^l
\big\langle \psi^{l-1}
\widetilde{I}^{(-m)}(\lambda) q^{e} e^{-\pi\ii e}\beta, 1\big\rangle_{0,2,d\ell} (-q)^d.
\end{equation}
Furthermore, it is sufficient to prove Proposition \ref{prop:mon_Phi} only in the case when $k=0$, because the general case will follow from that one by taking an appropriate analytic continuation with respect to $q$ around $q=0$. Let us assume $k=0$, so that
\[
e^{-\pi\ii e}\beta=
(2\pi)^{(1-n)/2}
\Gamma(1-e)^{n-1}
(2\pi\ii e).
\]
Note that $\gamma_0(q)$ is a simple loop approaching
$u_0(q)=(n-1)q^{1/(n-1)}\in \RR_{>0}$ along the positive real
axis. From now on we assume the above settings and denote
$\Phi_\beta(Q,\lambda)$ and $u_0(q)$ simply by respectively
$\Phi(q,\lambda)$ and $u(q)$. We have to prove that the analytic continuation of $\Phi(q,\lambda)$ along $\gamma_0(q)$ is $-\Phi(q,\lambda)$. Finally, let us point out that in the previous sections we denoted by~$q$ the sequence of Novikov variables $(q_1,\dots,q_r)$ of $X$, while in this section we denote by $q$ just a~positive real number. We will never have to deal with $X$, so there will be no confusion in doing~so.

\subsection{Contour integral}\label{sec:ci}
The Gromov--Witten invariants involved in the definition of the series $\Phi(q,\lambda)$ can be extracted from formula \eqref{blS}. Indeed, we have
\begin{align*}
\Phi(q,\lambda)={}&
\sum_{l=0}^\infty
(-\partial_\lambda)^l
\big(\leftexp{bl}{S}_l(-q,0) \widetilde{I}^{(-m)}(\lambda) q^{e} e^{-\pi\ii e}\beta,1\big)\\
 ={}&
\sum_{l=0}^\infty
(-\partial_\lambda)^l
\big( \widetilde{I}^{(-m)}(\lambda) q^{e} e^{-\pi\ii e}\beta, \leftexp{bl}{S}_l(-q,0)^T 1\big).
\end{align*}
Recall that $
\leftexp{bl}{S}(-q,0,-\partial_\lambda)^T=
\leftexp{bl}{S}(-q,0,\partial_\lambda)^{-1}$. Using formula \eqref{blS2}, we get
\begin{equation}\label{Phi:1}
\Phi(q,\lambda)=\int_{\operatorname{Bl}(X)}
\sum_{d=0}^\infty
\frac{(-1)^{dn} (-q)^d e \partial_\lambda }{
(e\partial_\lambda+d)^n
\prod_{i=1}^{d-1}(e\partial_\lambda +i)^{n-1} } \
\partial_\lambda^{d(n-1)} \
\widetilde{I}^{(-m)}(\lambda) q^{e} e^{-\pi\ii e}\beta.
\end{equation}
Using that $\theta e= e (\theta-1)$, we get the relation
$\widetilde{I}^{(-m)}(\lambda) e = e\partial_\lambda
\widetilde{I}^{(-m)}(\lambda) $ (see Lemma \ref{le:cp-homog}\,(a)),
where slightly abusing the notation we denoted by $e$ the operator of
classical cup product multiplication by $e$. Therefore,
\begin{align*}
\widetilde{I}^{(-m)}(\lambda) q^{e} e^{-\pi\ii e}\beta & =
q^{e\partial_\lambda}
(2\pi)^{(1-n)/2}
\Gamma(1-e\partial_\lambda )^{n-1}
(2\pi\ii e\partial_\lambda )
\widetilde{I}^{(-m)}(\lambda) 1\\
& =
q^{e\partial_\lambda}
(2\pi)^{(1-n)/2}
\Gamma(1-e\partial_\lambda )^{n-1}
(2\pi\ii e\partial_\lambda )
e^{-(n-1) e \partial_\lambda \partial_m}\!
\left(
\frac{\lambda^{\frac{n}{2}+m-\frac{1}{2} } }{
\Gamma\big( \frac{n}{2}+m+\frac{1}{2} \big) }
\right).
\end{align*}
Let us substitute the above formula for \smash{$\widetilde{I}^{(-m)}(\lambda)
q^{e} e^{-\pi\ii e}\beta$} in \eqref{Phi:1}. Note that everywhere the
operator $e$ comes together with the differentiation operator
$\partial_\lambda$. On the other hand, since in the entire expression only
the coefficient in front of $e^n$ contributes, we may remove
$\partial_\lambda$ from $e\partial_\lambda$ and apply to the entire
expression the differential operator $\partial_\lambda^n$, that is,
change \smash{$\partial_\lambda^{d(n-1)}$} to \smash{$\partial_\lambda^{d(n-1)
 +n}$}. We get the following formula for $\Phi(q,\lambda)$
\begin{gather*}
(2\pi)^{(1-n)/2}
2\pi\ii
\sum_{d=0}^\infty
\int_{\operatorname{Bl}(X)}
\frac{(-1)^{dn+d} q^{d+e} e^2 }{
(e+d)^n
\prod_{i=1}^{d-1}(e +i)^{n-1} }
\Gamma(1-e )^{n-1}\\
\qquad\times
\partial_\lambda^{d(n-1)+n}
e^{-(n-1) e \partial_m}
\left(
\frac{\lambda^{\frac{n}{2}+m-\frac{1}{2} } }{
\Gamma\big( \frac{n}{2}+m+\frac{1}{2} \big) }
\right).
\end{gather*}
Note that
\[
\partial_\lambda^{d(n-1)+n} \,
e^{-(n-1) e \partial_m}
\left(
\frac{\lambda^{\frac{n}{2}+m-\frac{1}{2} } }{
\Gamma\big( \frac{n}{2}+m+\frac{1}{2} \big) }
\right) =
\frac{\lambda^{-\frac{n}{2}-(n-1)(d+e)+m-\frac{1}{2} } }{
\Gamma\big(-\frac{n}{2}-(n-1)(d+e)+m+\frac{1}{2} \big) }
\]
and
\[
\Gamma(1-e) = (-e)(-e-1)\cdots (-e-d)\Gamma(-e-d) = (-1)^{d+1}
e(e+1)\cdots (e+d) \Gamma(-e-d).
\]
Since $\int_{\operatorname{Bl}(X)} e^n = (-1)^{n-1}$, we can replace
$\int_{\operatorname{Bl}(X)}$ with $(-1)^{n-1}\operatorname{Res}_{e=0}
\frac{de}{e^{n+1}}$. Note that $dn+d +(d+1)(n-1) +n-1=2dn+2n-2$ is an
even number, so that the signs that appear in our formula cancel out
exactly. We get
\begin{align*}
\Phi(q,\lambda)={}&
(2\pi)^{(1-n)/2}
2\pi\ii
\sum_{d=0}^\infty
\operatorname{Res}_{e=0}
\frac{de}{d+e}
q^{d+e}\Gamma(-d-e)^{n-1}\\
&\times
\frac{\lambda^{-\frac{n}{2}-(n-1)(d+e)+m-\frac{1}{2} } }{
\Gamma(-\frac{n}{2}-(n-1)(d+e)+m+\frac{1}{2} ) }.
\end{align*}
Let us substitute $x:=-e-d$, then the above formula becomes
\[
\Phi(q,\lambda)=
(2\pi)^{(1-n)/2}
2\pi\ii
\sum_{d=0}^\infty
\operatorname{Res}_{x=-d}
\frac{{\rm d}x}{x}
q^{-x}\Gamma(x)^{n-1}
\frac{\lambda^{-\frac{n}{2}+(n-1)x+m-\frac{1}{2} } }{
\Gamma\bigl(-\frac{n}{2}+(n-1)x+m+\frac{1}{2} \bigr) }.
\]
The sum of infinitely many residues can be replaced with an integral
of the form $\int_{\epsilon-\ii \infty}^{\epsilon+\ii \infty} {\rm d}x$, where $\epsilon >0$ is a positive real number. Let us sketch the proof of this claim. Let us fix a real number
$\delta\in \big(\tfrac{1}{2},1\big)$, such that,
$\mu:=(n-1)\big(\delta-1/2\big)\in \big(\tfrac{1}{2},1\big)$. Suppose that $K\geq 1$ is an integer. Let us consider the rectangular contour given by the boundary of the rectangle with vertices $\epsilon-\ii K, \delta-K-\ii K, \delta-K+\ii K, $ and $\epsilon+\ii K$. The contour is divided into two parts: the straight line segment $[\epsilon-\ii K,\epsilon+\ii K]$ and its complement which we denote by $C_K$ --- see Figure~\ref{Phi:contour} where these two pieces are colored respectively with blue and red.
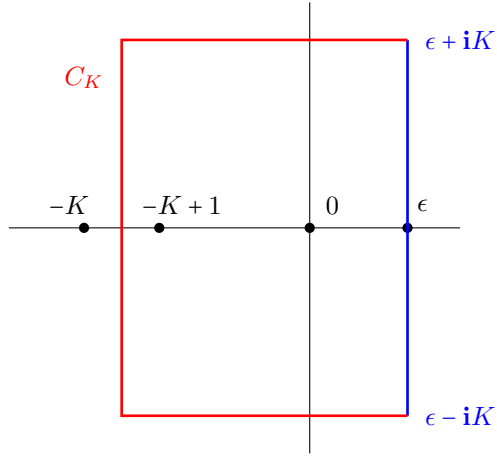
\begin{figure}[t]\centering
\begin{tikzpicture}[line width=1pt]
%origin and axis
\filldraw (0,0) circle [fill=black, radius=0.05cm];
\draw (0.3,0.3) node {\small $0$};
\filldraw (1.3,0) circle [fill=black, radius=0.05cm];
\draw (1.5,0.3) node {\small $\epsilon$};
\filldraw (-2,0) circle [fill=black, radius=0.05cm];
\draw (-1.7,0.3) node {\small $-K+1$};
\filldraw (-3,0) circle [fill=black, radius=0.05cm];
\draw (-3.2,0.3) node {\small $-K$};
\draw[thin] (-4,0) -- (2,0);
\draw[thin] (0,-3) -- (0,3);
%streight segment
\draw[blue] (1.3,-2.5)--(1.3,2.5cm);
\draw[blue] (2,-2.5cm) node {\small $\epsilon-\ii K$};
\draw[blue] (2,2.5cm) node {\small $\epsilon+\ii K$};
%integration path
\draw[red] (1.3,2.5cm) -- (-2.5,2.5cm)-- (-2.5,-2.5)--(1.3,-2.5);
\draw[red] (-3,2.0 cm) node {\small $C_K$};
\end{tikzpicture}

\caption{Integration contours.}\label{Phi:contour}
\end{figure}
By the Cauchy
residue formula, the integral along this contour coincides with the
partial sum $2\pi \ii \sum_{d=0}^{K-1} \operatorname{Res}_{x=-d}$. On the
other hand, using the standard asymptotic estimates for the
$\Gamma$-function (see Appendix \ref{app:bc}), one can prove
that if $\lambda>(n-1) q^{1/(n-1)}$
is a real number then the integral along $C_K$
tends to $0$ when $K\to \infty$. We get
\begin{equation}\label{Phi:2}
\Phi(q,\lambda) =
(2\pi)^{(1-n)/2}
\int_{\epsilon -\ii\infty}^{\epsilon + \ii \infty}
q^{-x}\Gamma(x)^{n-1}
\frac{\lambda^{-\frac{n}{2}+(n-1)x+m-\frac{1}{2} } }{
\Gamma\big(-\frac{n}{2}+(n-1)x+m+\frac{1}{2} \big) }
\frac{{\rm d}x}{x}.
\end{equation}
Let us denote by $G(q,\lambda)$ the right-hand side of \eqref{Phi:2}. Note that $G(q,\lambda)$, after replacing $1/x$ with $\Gamma(x)/\Gamma(x+1)$, becomes a {\em Mellin--Barnes} integral. The analytic properties of such integrals are well known (see \cite{Bat,PK2001}). Using the standard asymptotic estimates for the $\Gamma$-function it is easy to prove that the integral is convergent for all positive real $\lambda$ and that it is divergent for $\operatorname{Im}(\lambda)\neq 0$. Since the series \eqref{phi_infty}, viewed as a Laurent series in $\lambda^{-1}$, is convergent for~${|\lambda|>u(q)=(n-1) q^{1/(n-1)}}$, we get that $\Phi(q,\lambda)$ is the analytic continuation of the restriction of $G(q,\lambda)$ to the interval $\big[(n-1) q^{1/{(n-1)}}, +\infty\big)$.
\begin{Lemma}\label{le:ac_G}
The Mellin--Barnes integral $G(q,\lambda)$ is $0$ for all $0<\lambda \leq (n-1) q^{1/(n-1)}$.
\end{Lemma}
\begin{proof}
Let $R>0$ be a sufficiently big positive number. Let us fix $\delta\in (0,1)$. We would like to deform the contour $\epsilon + \ii \RR$ into the contour consisting of the 3 linear pieces
$-\ii - s$ ($-\infty < s \leq -\epsilon$),
$s \ii + \epsilon$ ( $-1\leq s\leq 1$), and
$\ii + s$ ($\epsilon\leq s<+\infty$). The integral $G(q,\lambda)$ is a limit as $R\to \infty$ of the integral over $s\ii +\epsilon$ ($-T\leq s\leq T$) where $T:=\sqrt{R^2-\epsilon^2}$, while the integral over the deformed contour is a limit as $R\to \infty$ of the integral over the contour consisting of the 3 linear pieces
$-\ii - s$ ($-\sqrt{R^2-1} < s \leq -\epsilon$),
$s \ii + \epsilon$ ( $-1\leq s\leq 1$), and
$\ii + s$ \smash{$\big(\epsilon\leq s<\sqrt{R^2-1}\big)$}. The difference between the two integrals is an integral over the two arcs
\smash{$C_R\colon R e^{\ii \theta}$} ($\arcsin(1/R)\leq \theta\leq \arcsin{T/R}$) and
\smash{$\overline{C}_R\colon R e^{\ii \theta}$} $(-\arcsin(T/R)\leq \theta\leq -\arcsin{1/R})$. One has to prove that
\[
\lim_{R\to +\infty} \int_{C_R \mbox{ or } \overline{C}_R}
q^{-x}\Gamma(x)^{n-1}
\frac{\lambda^{-\frac{n}{2}+(n-1)x+m-\frac{1}{2} } }{
\Gamma\big(-\frac{n}{2}+(n-1)x+m+\frac{1}{2} \big) }
\frac{{\rm d}x}{x} =0.
\]
This is proved in the same way as in \cite[Section 5]{DF}. Namely, divide $C_R$ into two pieces
$C_{R}'\colon R e^{\ii \theta}$ ($\arcsin(1/R)\leq \theta\leq \delta$) and
$C_{R}''\colon R e^{\ii \theta}$ ($\delta \leq \theta\leq \arcsin{T/R}$) and then use the standard asymptotic estimates for the $\Gamma$-function and the assumption $|\lambda|\leq (n-1) q^{1/(n-1)}$.

Finally, to complete the proof. Note that the integral is independent of $\epsilon>0$, because the $\Gamma$-functions in $G(q,\lambda)$ do not have poles on the positive real axis. Letting $\epsilon\to +\infty$ and using again the standard asymptotic estimates for the $\Gamma$-function, we get that $G(q,\lambda)=0$ for~${\lambda\leq (n-1) q^{1/(n-1)}}$.
\end{proof}

Using the above lemma, we get
\[
\int_{u(q)}^\infty e^{-\lambda s} G(q,\lambda) {\rm d}\lambda =
\int_{0}^\infty e^{-\lambda s} G(q,\lambda){\rm d}\lambda.
\]
Substituting $G(q,\lambda)$ with the corresponding Mellin--Barnes integral, exchanging the order of integration and using that
\[
\int_{0}^\infty e^{-\lambda s}
\frac{\lambda^{-\frac{n}{2}+(n-1)x+m-\frac{1}{2} } }{
\Gamma\bigl(-\frac{n}{2}+(n-1)x+m+\frac{1}{2} \bigr) } {\rm d}\lambda =
s^{\frac{n}{2}-(n-1)x-m-\frac{1}{2}},
\]
we get
\begin{equation}\label{LT-G}
\int_{u(q)}^\infty e^{-\lambda s} G(q,\lambda) {\rm d}\lambda =
(2\pi)^{(1-n)/2}
\int_{\epsilon -\ii\infty}^{\epsilon + \ii \infty}
q^{-x}\Gamma(x)^{n-1}
s^{\frac{n}{2}-(n-1)x-m-\frac{1}{2}}
\frac{{\rm d}x}{x}.
\end{equation}

\subsection{Oscillatory integral}

Let us consider the following family of functions
\[
f(x,q)=x_1+\cdots + x_{n-2} +
\frac{q}{ x_1\cdots x_{n-2}} \big(1+x_{n-1}^2+x_n^2\big),
\]
where $q$ is a positive real number and
\[
x=(x_1,\dots, x_n)\in V:=\CC^n\setminus{\big\{
x_1\cdots x_{n-2} \big(1+x_{n-1}^2+x_n^2\big)= 0\big\} }.
\]
Let $\Gamma:=\RR^{n-2}_{>0}\times \RR^2 \subset V$, that is, $\Gamma$
is the real $n$-dimensional cycle in $V$ consisting of points $x=(x_1,\dots,x_n)$,
such that, the first $n-2$ coordinates are positive real numbers and
the last two ones are arbitrary real numbers. Note that the cycle
$\Gamma$ belongs to the following group of
semi-infinite homology cycles:
\[
\varprojlim H_n (V, \operatorname{Re}(f(x,q))>M,\ZZ)\cong \ZZ^{n-1},
\]
where the inverse limit is taken over all $M\in \RR$.
\begin{Proposition}\label{prop:oscil_int}
Under the above notation the following identity holds:
\[
2\ii
\int_{\Gamma} e^{-f(x,q)}
\frac{{\rm d}x_1\wedge \cdots \wedge {\rm d}x_n}{
x_1\cdots x_{n-2} (1+x_{n-1}^2+x_n^2) } =
\int_{\epsilon-\ii\infty}^{\epsilon+\ii\infty}
q^{-x} \Gamma(x)^{n-1} \frac{{\rm d}x}{x},
\]
where the orientation of $\Gamma$ is induced from the standard
orientation on $\RR^n$.
\end{Proposition}
\begin{proof}
Let us integrate out $x_{n-1}$ and $x_{n}$. Using polar coordinates
$x_{n-1}= r \cos \theta$ and $x_n=r\sin \theta$, since ${\rm d}x_{n-1}\wedge
{\rm d}x_n=r {\rm d}r\wedge\theta$, we get
\[
\int_{\RR^2} e^{-K (1+x_{n-1}^2+x_n^2)}
\frac{{\rm d}x_{n-1}\wedge {\rm d}x_n}{
1+x_{n-1}^2+x_n^2} =
\int_0^\infty e^{-K(1+r^2)}\int_{0}^{2\pi}
\frac{r {\rm d}r\wedge {\rm d}\theta }{1+r^2} =
\pi \int_1^\infty e^{-Ku} \frac{{\rm d}u}{u},
\]
where $K$ is a positive real number and for the second equality we used
the substitution $u=1+r^2$. Applying the above formula to our
oscillatory integral, we get
\begin{gather}
\int_{\Gamma} e^{-f(x,q)}
\frac{{\rm d}x_1\wedge \cdots \wedge {\rm d}x_n}{
x_1\cdots x_{n-2} \big(1+x_{n-1}^2+x_n^2\big) }\nonumber\\
\qquad = \pi
\int_{\RR^{n-2}_{>0}} \ \int_1^\infty
e^{-(
x_1+\cdots + x_{n-2} + \frac{q u}{ x_1\cdots x_{n-2}} )}
\frac{{\rm d}u}{u}\ \frac{{\rm d}x_1\cdots {\rm d}x_{n-2}}{x_1\cdots x_{n-2}},\label{osc_int:1}
\end{gather}
where ${\rm d}x_1\cdots {\rm d}x_{n-2}$ is the standard Lebesgue measure on
$\RR^{n-2}_{>0}$. On the other hand, let us recall the oscillatory
integral
\[
J(q):= \int_{\RR^{n-2}_{>0}}
\exp\left( -\left(
x_1+\cdots + x_{n-2} + \frac{q}{x_1\cdots x_{n-2}}
\right) \right) \frac{{\rm d}x_1\cdots {\rm d}x_{n-2}}{x_1\cdots x_{n-2}}.
\]
Note that the Mellin transform of $J(q)$ is
\[
\{\mathcal{M}J\}(x) = \int_0^\infty q^{x-1} J(q) {\rm d}q = \Gamma(x)^{n-1}.
\]
Recalling the Mellin inversion theorem, we get
\[
J(q) = \frac{1}{2\pi \ii} \int_{\epsilon-\ii\infty}^{\epsilon+\ii\infty} q^{-x}
\Gamma(x)^{n-1} {\rm d}x,
\]
where $\epsilon>0$ is a positive real number. Let us apply the above
formula to \eqref{osc_int:1}. Namely, on the right-hand side of \eqref{osc_int:1},
after exchanging the order of the integration, we get
\[
\pi \int_1^\infty J(qu) \frac{{\rm d}u}{u} = \frac{1}{2\ii}
\int_1^\infty
\int_{\epsilon-\ii\infty}^{\epsilon+\ii\infty} (q u) ^{-x}
\Gamma(x)^{n-1} {\rm d}x \frac{{\rm d}u}{u}.
\]
Exchanging again the order of integration and using that
\[
\int_1^\infty u^{-x} \frac{{\rm d}u}{u} = \left.
\frac{u^{-x} }{-x}\right|^{u=\infty}_{u=1} = \frac{1}{x},
\]
we get the formula stated in the proposition.
\end{proof}

\subsection{Laplace transform}
The function $f(x,q)$ has a minimum over $x\in \Gamma$ achieved at the
critical point $x_1=\cdots = x_{n-2}=q^{1/(n-1)}$,
$x_{n-1}=x_{n-2}=0$. Note that the corresponding critical value is
$u(q) = (n-1) q^{1/(n-1)}$. Let us consider the map $\Gamma\to
[u(q),+\infty)$, $x\mapsto f(x,q)$. The fiber over $\lambda \in
(u(q),+\infty)$ is the real algebraic hypersurface $\Gamma_\lambda\subset
\Gamma$ defined by
\[
x_1+\cdots + x_{n-2} +
\frac{q}{x_1\cdots x_{n-2}}
\bigl(1+x_{n-1}^2+x_n^2\bigr) = \lambda.
\]
It is easy to see that $\Gamma_\lambda$ is compact and it has the homotopy
type of a sphere. Indeed, the map
\[
\Gamma\setminus{\{u(q)\}}\to (u(q),+\infty),\qquad
x\mapsto f(x,q)
\]
is proper and regular. Therefore, according to the
Ehresmann's fibration theorem, it must be a~locally trivial fibration and hence
a trivial fibration, because $(u(q),+\infty)$ is a contractible
manifold. If $\lambda$ is sufficiently close to $u(q)$, then
$\Gamma_\lambda$ is contained in a Morse coordinate neighborhood of the
critical point $\big(q^{1/(n-1)},\dots, q^{1/(n-1)},0,0\big)$. Switching to
Morse coordinates for $f$, we get that the fiber $\Gamma_\lambda$ is
diffeomorphic to the $(n-1)$-dimensional sphere.

Let use denote by $\Gamma_{\leq \lambda}$ the subset of $\Gamma$ defined by
the inequality
\[
x_1+\cdots + x_{n-2} +
\frac{q}{x_1\cdots x_{n-2}}
\big(1+x_{n-1}^2+x_n^2\big) \leq \lambda.
\]
Note that $\Gamma_{\leq \lambda}$ is a manifold with boundary and its
boundary is precisely $\partial \Gamma_{\leq
 \lambda}=\Gamma_\lambda$. Put
\[
\mathcal{I}(q,\lambda) :=
\int_{\Gamma_{\leq \lambda}} \frac{
(\lambda-f(x,q))^{m-\frac{n}{2}-\frac{1}{2}} }{
\Gamma\big( m-\frac{n}{2}+\frac{1}{2} \big) }
\omega,
\]
where
\[
\omega:= \frac{{\rm d}x_1\wedge \cdots \wedge {\rm d}x_n}{
x_1\cdots x_{n-2}\big(1+x_{n-1}^2+ x_n^2\big) }.
\]
\begin{Lemma}\label{le:Lapl_I}
The following formula holds:
\[
\int_\Gamma e^{-f(x,q) s} \omega=
s^{m-\frac{n}{2}+\frac{1}{2}}
\int_{u(q)}^\infty e^{-\lambda s} \mathcal{I}(q,\lambda) {\rm d}\lambda.
\]
\end{Lemma}
\begin{proof}
Using Fubini's theorem, we transform
\[
\mathcal{I}(q,\lambda) = \int_{u(q)}^\lambda
\frac{(\lambda-\mu)^{m-\frac{n}{2}-\frac{1}{2}} }{
\Gamma\big(m-\frac{n}{2}+\frac{1}{2}\big) }
\int_{\Gamma_\mu} \frac{\omega}{{\rm d}f} {\rm d}\mu.
\]
Therefore,
\[
\int_{u(q)}^\infty e^{-\lambda s} \mathcal{I}(q,\lambda) {\rm d}\lambda =
\int_{u(q)}^\infty
\int_{u(q)}^\lambda \ e^{-\lambda s}
\frac{(\lambda-\mu)^{m-\frac{n}{2}-\frac{1}{2}} }{
\Gamma\big(m-\frac{n}{2}+\frac{1}{2}\big) }
\int_{\Gamma_\mu} \frac{\omega}{{\rm d}f} {\rm d}\mu {\rm d}\lambda.
\]
Exchanging the order of the integration, we get
\[
\int_{u(q)}^\infty
\left(
\int_\mu^\infty
e^{-\lambda s}
\frac{(\lambda-\mu)^{m-\frac{n}{2}-\frac{1}{2}} }{
\Gamma\big(m-\frac{n}{2}+\frac{1}{2}\big) } {\rm d}\lambda \right)
\int_{\Gamma_\mu} \frac{\omega}{{\rm d}f} {\rm d}\mu =
s^{-m+\frac{n}{2} -\frac{1}{2}}
\int_{u(q)}^\infty
e^{-\mu s}
\int_{\Gamma_\mu} \frac{\omega}{{\rm d}f} {\rm d}\mu.
\]
Recalling again Fubini's theorem we get that the above iterated
integral coincides with \[\int_\Gamma e^{-f(x,q) s} \omega.\] The
formula stated in the lemma follows.
\end{proof}

According to the above lemma, the Laplace transform of the integral
$\mathcal{I}(q,\lambda)$ is given by the following formula:
\[
\int_{u(q)}^\infty e^{-\lambda s} \mathcal{I}(q,\lambda)=
s^{-m+\frac{n}{2} -\frac{1}{2}}
\int_\Gamma e^{-f(x,q) s} \omega=:F(s).
\]
Let us recall Proposition \ref{prop:oscil_int} and note that $f(x,q)$
has the following rescaling symmetry:
\[
f\big(s\cdot x, s^{n-1} q\big)= f(x,q)s,
\]
 where
\[
s\cdot (x_1,\dots, x_n) = (s x_1,\dots, s x_{n-2}, x_{n-1}, x_n).
\]
Note that if $s>0$ is a positive real number, then the integration
cycle and the holomorphic form~$\omega$ are invariant under the
rescaling action by $s$. Therefore, the formula from Proposition~\ref{prop:oscil_int} yields the following formula:
\[
\int_{\Gamma} e^{-f(x,q) s} \omega = \frac{1}{2\ii}
\int_{\epsilon-\ii\infty}^{\epsilon+\ii\infty}
q^{-x} \Gamma(x)^{n-1} s^{-(n-1) x} \frac{{\rm d}x}{x}.
\]
Therefore, the function
\[
F(s) = \frac{1}{2\ii}
\int_{\epsilon-\ii\infty}^{\epsilon+\ii\infty}
q^{-x} \Gamma(x)^{n-1} s^{-(n-1) x -m+\frac{n}{2} -\frac{1}{2} } \frac{{\rm d}x}{x}.
\]
Comparing the above formula with \eqref{LT-G} and using that the Laplace transformation is injective on smooth functions, we get that
$G(q,\lambda)= 2\ii (2\pi)^{(n-1)/2} \I(q,\lambda)$. Finally, in order
to complete the proof of Proposition \ref{prop:mon_Phi}, we need only
to check that the analytic continuation of $\I(q,\lambda)$ around
$\lambda=u(q)=(n-1) q^{1/(n-1)}$ transforms $\I(q,\lambda)$ into
$-\I(q,\lambda)$. This however is a local computation. Indeed, if
$\lambda$ is sufficiently close to $(n-1) q^{1/(n-1)}$, then the
integration cycle defining $\I(q,\lambda)$ is sufficiently close to
the critical point $\big(q^{1/(n-1)},\dots, q^{1/(n-1)},0,0\big)$. By
switching to Morse coordinates, we get
\[
\int_{\Gamma_\mu} \omega/df = (\mu-u(q))^{\frac{n}{2}-1} P(q,\mu),
\]
where $P(q,\mu)$ is holomorphic at $\mu=u(q)$
(see \cite[Section 12.1, Lemma 2]{AGuV}). Therefore,
\begin{equation}\label{vanishing_period}
\mathcal{I}(q,\lambda) = \int_{u(q)}^\lambda
\frac{(\lambda-\mu)^{m-\frac{n}{2}-\frac{1}{2}} }{
 \Gamma\big(m-\frac{n}{2}+\frac{1}{2}\big) }
(\mu-u(q))^{\frac{n}{2}-1} P(q,\mu){\rm d}\mu.
\end{equation}
Changing the variables $\mu-u(q) = t (\lambda-u(q))$, we get
$\lambda-\mu= (1-t)(\lambda-u(q))$ and ${\rm d}\mu = (\lambda-u(q)) {\rm d}t$, we
get
\begin{align*}
\int_{u(q)}^\lambda
\frac{(\lambda-\mu)^{m-\frac{n}{2}-\frac{1}{2}} }{
 \Gamma\big(m-\frac{n}{2}+\frac{1}{2}\big) }
(\mu-u(q))^{i+\frac{n}{2}-1} {\rm d}\mu & = \int_0^1
\frac{(1-t)^{m-\frac{n}{2}-\frac{1}{2}} }{
 \Gamma\big(m-\frac{n}{2}+\frac{1}{2}\big) } t^{i+\frac{n}{2}-1} {\rm d}t \,
 (\lambda-u(q))^{i+m-1/2}
 \\
 &
 = \Gamma(i+n/2)
\frac{(\lambda-u(q))^{i+m-1/2} }{ \Gamma(i+m+1/2)}.
\end{align*}
Substituting the Taylor series expansion of
$P(q,\mu)=\sum_{i=0}^\infty P_i(q) (\mu-u(q))^i$ at $\mu=u(q)$
in \eqref{vanishing_period} and using the above formula, we get
\[
\I(q,\lambda) = (\lambda-u(q))^{m-1/2}
\sum_{i=0}^\infty
\frac{\Gamma(i+n/2)}{ \Gamma(i+m+1/2)} P_i(q)
(\lambda-u(q))^i.
\]
The above expansion is clearly anti-invariant under the analytic
continuation around $\lambda=u(q)$.

\appendix

\section{Bending the contour} \label{app:bc}
For the sake of completeness we would like to prove that if $\lambda$ is a positive real number, such that, $\lambda>(n-1) q^{1/(n-1)}$, then
\[
\lim_{K\to +\infty} \int_{C_K}
q^{-x}\lambda^{(n-1)x}
\frac{ \Gamma(x)^{n-1} }{
\Gamma\bigl(-\frac{n}{2}+(n-1)x+m+\frac{1}{2} \bigr) }
\frac{{\rm d}x}{x} =0,
\]
where $C_K$ is the contour defined in Section \ref{sec:ci} (see Figure \ref{Phi:contour}). The integrand of the above integral differs from the integrand in \eqref{Phi:2} by the constant factor \smash{$\lambda^{-\frac{n}{2}+m-\frac{1}{2}}$}. Therefore, the vanishing result needed in the derivation of \eqref{Phi:2} follows from the above statement.

Let us consider first the upper horizontal part of $C_K$, that is, $x=a + \ii K$, $\delta-K\leq a \leq \epsilon$. The estimate in this case is a direct consequence of the Stirling's formula for the gamma function. Namely, recall that if $x=a+\ii b\notin (-\infty, 0]$, then
\[
|\Gamma(x)|=
\sqrt{2\pi}
e^{-a-|b| |{\operatorname{Arg}}(x)|}|x|^{a-1/2}(1+o(1)),
\]
where $-\pi< \operatorname{Arg}(x)<\pi$ and $o(1)\to 0$ uniformly when $|x|\to \infty$ in any proper subsector $-\pi<\alpha\leq \operatorname{Arg}(x)\leq \beta<\pi$.
Put $c:=-\tfrac{n}{2} + m +\tfrac{1}{2}$. Using Stirling's formula, we get
\begin{align*}
|\Gamma((n-1)x+c)|={}& \sqrt{2\pi}
e^{-(n-1) a -c -(n-1) |b| |{\operatorname{Arg}}(x+c/(n-1))|}\\
&\times
|(n-1)x+c|^{(n-1)a+c-1/2}(1+o(1)).
\end{align*}
Note that $|{\operatorname{Arg}}(x+c/(n-1))|\leq |{\operatorname{Arg}}(x)|$ because we may choose $m$ so big that $c>0$ while
\[
|(n-1)x+c|^{(n-1)a+c-1/2}= (n-1)^{(n-1) a}
|x|^{(n-1) a + c-1/2} O(1).
\]
Moreover, both $(n-1)x+c$ and $x$ belong to the sector $-\frac{3\pi}{4}\leq \operatorname{Arg}(x)\leq \frac{3\pi}{4}$ for all $x$ in the horizontal integration contour. Therefore, we have an estimate of the form
\[
|\Gamma((n-1)x+c)|^{-1}\leq \operatorname{const}
(n-1)^{-(n-1) a}
|x|^{-(n-1) a - c+1/2}
e^{(n-1) a +(n-1) |b| |{\operatorname{Arg}}(x+c/(n-1))|},
\]
for all $x$ in the upper horizontal part of $C_K$, where the constant is independent of $K$.
Note that~${|q^{-x} \lambda^{(n-1)x}|=q^{-a} \lambda^{(n-1)a}}$. Combining all these estimates together, we get that the absolute value of the integrand along the upper horizontal contour can be bounded from above by
\[
\operatorname{const}
((n-1)q^{1/(n-1)} /\lambda)^{(n-1) a}
|x|^{-m-1/2} |{\rm d}a|\leq
\operatorname{const}
K^{-m-1/2} |{\rm d}a|,
\]
where we used that $\lambda>(n-1) q^{1/(n-1)}$ and $|x|^2\leq K^2+(K+\epsilon-\delta)^2\leq (1+|\epsilon-\delta|)^2 K^2$ for all~${x=a+\ii K}$ ($\delta-K\leq a\leq \epsilon$). Therefore, up to a constant independent of $K$ the integral is bounded by $K^{-m+1/2}$ which proves that the integral vanishes in the limit $K\to \infty$.

The estimate for the lower horizontal part of $C_K$, that is, $x=a-\ii K$, $\delta-K\leq a\leq \epsilon$ is the same as above. Let us consider the vertical part $x=\delta-K+\ii b$, $-K\leq b\leq K$. In order to apply Stirling's formula, let us first recall the reflection formula for the gamma function
\[
\Gamma(x)=
\Gamma(1-x)^{-1}
\frac{2\pi \ii}{e^{2\pi\ii x} -1}
e^{\pi\ii x}.
\]
If $x$ is on the vertical part of the integration contour, then $-x$ belongs to a proper subsector of~${-\pi< \operatorname{Arg}(x)<\pi}$ in which the Stirling's formula for $\Gamma(1-x)= (-x) \Gamma(-x)$ can be applied, that~is,
\[
|\Gamma(x)|= \sqrt{2\pi}
\frac{e^{-\pi b}}{\big|e^{2\pi\ii a} e^{-2\pi b}-1\big|}
|x|^{a-1/2}
e^{-a+|b| |{\operatorname{Arg}}(-x)|}
(1+o(1)),
\]
where $x=a+\ii b$. Similarly,
\begin{align*}
|\Gamma((n-1)x+c)| ={} &
\sqrt{2\pi}
\frac{e^{-\pi (n-1) b}}{\big|e^{2\pi\ii ((n-1)a+c)} e^{-2\pi (n-1)b}-1\big|}
|(n-1)x+c|^{(n-1)a+c-1/2}\\
&
\times
e^{-(n-1)a-c+(n-1)|b| |{\operatorname{Arg}}(-x-c/(n-1))|}
(1+o(1)).
\end{align*}
Note that if $x=a+\ii b$ is on the integration contour, then $a=\delta-K$ and $(n-1)a+c= \mu +m-K$, where $\mu=(n-1)(\delta-1/2)$. Therefore, $e^{2\pi\ii a}=e^{2\pi\ii\delta}$ and
$e^{2\pi\ii ((n-1)a+c)} = e^{2\pi\ii \mu}$ are constants independent of $K$. Moreover, we chose both $\mu$ and $\delta$ to be non-integers, so $e^{2\pi\ii\delta}-1$ and $e^{2\pi\ii\mu}-1$ are non-zero.
We get
\begin{align*}
\frac{|\Gamma(x)|^{n-1}}{
|\Gamma((n-1)x+c) x| }\leq {}& \operatorname{const}
\frac{\big| e^{2\pi\ii \mu} e^{-2\pi (n-1) b}-1\big|}{
\big| e^{2\pi\ii \delta} e^{-2\pi b}-1\big|^{n-1}}
\frac{|x|^{(n-1)(a-1/2)-1}}{
|(n-1)x+c|^{(n-1)a+c-1/2}} \\
&
\times
e^{(n-1)|b|(
|{\operatorname{Arg}}(-x))|-|{\operatorname{Arg}}(-x-c/(n-1) )|}
(1+o(1)).
\end{align*}
The first fraction is clearly a bounded function in $b\in \RR$. For the second one, we have
\[
\frac{|x|^{(n-1)(a-1/2)-1}}{
|(n-1)x+c|^{(n-1)a+c-1/2}}\leq \operatorname{const}
\frac{|x|^{-m-1/2}}{
(n-1)^{(n-1)a}}.
\]
Finally, for the exponential term, let us look at the triangle formed by vectors $-x$ and $-x-c/(n-1)$. The area of this triangle is
$\frac{|b|c}{2(n-1)}$. On the other hand, the difference $\theta:=|{\operatorname{Arg}}(-x))|-|{\operatorname{Arg}}(-x-c/(n-1) )|$ as $K\rightarrow \infty$ tends to $0$ uniformly in $x=\delta-K+\ii b$ for $|b|\leq K$. Therefore, up to a constant independent of $K$ we can bound $\theta$ from above by $\sin \theta$. Using that the area of the triangle is also $\tfrac{1}{2} |x| |x+c/(n-1)| \sin\theta$, we get
\begin{gather*}
(n-1)|b|(
|{\operatorname{Arg}}(-x))|-|{\operatorname{Arg}}(-x-c/(n-1) )| \\
\qquad=
(n-1)|b|\theta \leq
\operatorname{const} |b|\sin\theta \leq
\operatorname{const}
\frac{b^2 c}{|x| |(n-1)x+c|}.
\end{gather*}
The above expression is bounded by a constant independent of $K$. We get the following estimate:
\[
\frac{|\Gamma(x)|^{n-1}}{
|\Gamma((n-1)x+c) x| }\leq
\operatorname{const}
K^{-m-1/2}
(n-1)^{(n-1) K}
\]
for all $x=\delta-K + \ii b$, $-K\leq b\leq K$, where the constant is independent of $K$. Finally, since $\big|q^{-x}\lambda^{(n-1)x}\big|=q^{-a}\lambda^{(n-1)a}$, we get the following estimate:
\[
\left|
q^{-x}\lambda^{(n-1)x}
\frac{\Gamma(x)^{n-1}}{
\Gamma((n-1)x+c) x }
\right| \leq \operatorname{const}
\big((n-1) q^{\frac{1}{n-1}}/\lambda\big)^{(n-1)K}
K^{-m-1/2}.
\]
Since $\lambda> (n-1) q^{\frac{1}{n-1}}$ the integral along the vertical segment of $C_K$, up to a constant, is bounded by $K^{-m+1/2}$. Therefore, the integral vanishes in the limit $K\to \infty$.

\subsection*{Acknowledgments}
We would like to thank the anonymous referees for many useful
suggestions that helped us improve the exposition.
The research of T.M. is partially supported by Grant-in-Aid (Kiban C)
17K05193 and JP22K03265. This work is supported by the World Premier International Research
Center Initiative (WPI Initiative), MEXT, Japan.

\pdfbookmark[1]{References}{ref}
\LastPageEnding

\end{document}